\documentclass[11pt, final]{article}

\usepackage[utf8]{inputenc} 
\usepackage{lmodern} 
\usepackage{a4wide, amsmath, amssymb, mathrsfs}        
\usepackage{amsthm}         
\usepackage[numbers, sort&compress]{natbib}         
\usepackage{enumitem} 
\setlist[enumerate]{label={\rm(\roman*)}} 
\usepackage[a4paper, margin=1in]{geometry} 
\usepackage{stmaryrd} 

\usepackage[hyphens]{url} 
\usepackage{hyperref} 
\hypersetup{colorlinks=true, linkcolor=teal, citecolor=violet} 
\usepackage{authblk}

\usepackage{tikz}

\newtheorem*{theorem*}{Theorem}

\newtheorem{theorem}{Theorem}[section]

\theoremstyle{definition}
\newtheorem{definition}[theorem]{Definition}

\theoremstyle{definition}

\theoremstyle{plain}
\newtheorem{proposition}[theorem]{Proposition}

\theoremstyle{plain}
\newtheorem{lemma}[theorem]{Lemma}

\theoremstyle{plain}
\newtheorem{corollary}[theorem]{Corollary}

\theoremstyle{definition}
\newtheorem{remark}[theorem]{Remark}

\theoremstyle{definition}

\theoremstyle{plain}

\theoremstyle{plain}

\theoremstyle{plain}
\newtheorem{conjecture}[theorem]{Conjecture}

\numberwithin{equation}{section}

\def\diff{\mathrm{d}}
\def\Diff{\mathrm{D}}
\def\reals{\mathbb{R}}
\def\nat{\mathbb{N}}

\def\spt{\operatorname{spt}}

\def\conv{\operatorname{conv}}

\def\Em{\operatorname{Em}}
\def\H{\mathcal{H}}

\def\refine{\operatorname{refine}}

\def\Wedge{\textstyle\bigwedge}
\def\Rleft{l(R)}
\def\Rright{r(R)}

\newcommand{\AEm}{\textnormal{\AE}}
\def\Reflat{\textnormal{Ref}}

\newcommand{\tLip}{\textnormal{Lip}}
\newcommand{\tspan}{\operatorname{span}}

\newcommand{\Id}{\operatorname{Id}}
\newcommand{\rng}{\operatorname{rng}}

\title{Failure of Lang's Flat Chain Conjecture and non-regularity of the prescribed Jacobian equation}
\author{Jakub Tak\'a\v{c}}
\affil{University of Warwick}
\begin{document}
\maketitle

\begin{abstract}
We show that Lang's Flat Chain Conjecture (that is, without requiring finite mass of the underlying currents) fails for metric $k$-currents in $\reals^d$ whenever $d\geq 2$ and $k\in\{1, \dots, d\}$.
In all other cases, it holds. The original conjecture due to Ambrosio and Kirchheim remains open.
We first connect Lang's conjecture to a regularity statement concerning the prescribed Jacobian equation near $L^\infty$. 
We then show that the equation does not have the required regularity.
For a Lipschitz vector field $\pi$, its derivative $\Diff\pi$ exists a.e.~and is identified with a matrix.
Our non-regularity results for the prescribed Jacobian equation quantify how ``small'' the set
    \begin{equation*}
        \conv(\{\det\Diff \pi: \tLip(\pi)\leq L\})\subset L^\infty
    \end{equation*}
    is for every $L>0$.
The symbol ``$\conv$'' stands for the convex hull.
The ``smallness'' is quantified in topological terms and is used to show that Lang's Flat Chain Conjecture fails.
\end{abstract}
\section{Introduction}\label{S:intro}
The main purpose of this paper is to provide a complete resolution of the so-called ``Lang's Flat Chain Conjecture'' \cite[page 716]{UL}.
The conjecture concerns the structure of metric currents, which are objects of interest in Geometric Measure Theory.
They are related to a number of geometric problems (in metric spaces) such as the Plateau problem \cite{AKcur}, isoperimetric inequalities \cite{BWY}, flows of measures \cite{PaSteFlows} and even geometric flows \cite{Allen} (here currents are in particular used to define a suitable notion of distance \cite{SoWe}).
Theory of metric currents was first developed by Ambrosio and Kirchheim in their seminal paper \cite{AKcur}, where they also state what is popularly called the ``Flat Chain Conjecture'' \cite[page 68]{AKcur}.
The conjecture states that metric $k$-currents of \emph{finite mass} in $\reals^d$ with compact support are flat $k$-chains in the sense of Federer and Fleming (we provide relevant definitions as well as a detailed description of the problem later in Subsections \ref{SS:intro-currents} and \ref{SS:prel-currents}). 
Later a more general theory of metric currents, encompassing also currents not of finite mass, was developed by Lang \cite{UL} and an analogous question was asked about this type of metric currents. 
Our main result is that the conjecture due to Lang for $k$-currents in $\reals^d$ fails if $d\geq 2$ and $k\in\{1, \dots, d\}$ (see Theorem \ref{T:fcc-fails-all-cases}). The original conjecture due to Ambrosio and Kirchheim remains, in general, open. We discuss the special cases that have been resolved \cite{DeRindler, Sch} in the subsections below.

Remarkably, our method relies on connecting the conjecture to the (non-)regularity of a certain partial differential equation (PDE), which is usually called the \emph{prescribed Jacobian equation}:
\begin{equation}\label{E:PJE1}
    \det \Diff \pi = \rho\quad \text{on $Q=[0,1]^d$.}
\end{equation}
If $\pi$ is a Lipschitz vector field on $Q$, then by Rademacher's differentiation theorem, the function $\det\Diff\pi$ is defined a.e.~and lies in $L^\infty(Q)$.
Thus one could ask, for any $\rho \in L^\infty(Q)$, is there a Lipschitz vector field $\pi$ satisfying \eqref{E:PJE1} a.e.?
The only available result in this direction is due to Burago and Kleiner \cite{BK} and independently McMullen \cite{McM}, who, provided $d\geq 2$, construct a $\rho\in L^\infty(Q)$ attaining only two positive values such that there is no \emph{biLipschitz} solution to the equation.
We give a stronger result, asserting that the of set of $\rho\in L^\infty(Q)$ for which there exists a Lipschitz solution (as opposed to biLipschitz) is ``small'' (Proposition \ref{P:intro-strong-non-surj}).

It is this property of the prescribed Jacobian equation that eventually leads to the resolution of Lang's Flat Chain Conjecture. A number of connections of the conjecture to a Lipschitz regularity of PDEs have been recently observed by multiple groups \cite{DeMaMa, ArBou, BCTVW, MM} independently (including the author), as is also indicated in the recent survey by Marchese \cite{MarcheseSurvey}. We provide a more detailed discussion of these connections in Subsection \ref{SS:intro-char}.
As for the prescribed Jacobian equation, we give an overview of some previously available results, as well as of some related problems and of our own results, later in Subsection \ref{SS:intro-reg}, but for now, we turn our attention to metric currents.

\subsection{Gentle introduction to metric currents and the Flat Chain Conjecture}\label{SS:intro-currents}
The basic idea behind ``currents'' is that of generalized surfaces.
Consider the Plateau problem, which asks for a minimal $k$ dimensional manifold $M$ having the prescribed boundary $\partial M = N$.
Here, ``minimal'' is to be interpreted as having the least possible $k$ dimensional volume.
The standard way to show that such an object even exists is to take a sequence $M_n$, whose volumes converge to the infimal volume and then use some kind of compactness theorem to extract a ``convergent'' subsequence.
The limit of this subsequence can then be expected to be a solution.
The basic problem with this approach is that the space of (smooth) manifolds is not compact with respect to a useful notion of convergence and therefore we need to ``enlarge'' the space into something, where compactness is available.
This is usually achieved by various different spaces of currents and the first successful implementation of this idea (for the Plateau problem) is due to Federer and Fleming \cite{FF}.

General theory of currents dates back to de Rham \cite{dR}
and is based on the observation that every oriented $k$ dimensional manifold acts by integration on, say, smooth and compactly supported differential $k$-forms.
Moreover, this action is linear and (in an appropriate sense) continuous.
A $k$-current is, by definition, any such linear and continuous action.
The approach is analogous to the classical theory of distributions.
Since the theory requires differential forms, it implicitly uses the vector space structure of $\reals^d$ and therefore it only makes sense in a Euclidean space (or possibly a space with a suitable notion of a tangent field, such as a manifold).
Thus, if one wishes to make sense of a current in a general metric space a different approach needs to be chosen.

A theory of metric currents was first developed by Ambrosio and Kirchheim in their seminal paper \cite{AKcur} motivated by earlier ideas of De Giorgi.
Their theory imposed a finite mass condition on the currents and was later extended by Lang \cite{UL}, who developed a more general\footnote{The theory of Lang is not \emph{strictly} more general than the theory of Ambrosio and Kirchheim in non $\sigma$-compact metric spaces.
This has to do with set-theoretical considerations \cite[Lemma 2.9 and the paragraph above]{AKcur}, but under reasonable assumptions, every current of Ambrosio and Kirchheim can be obtained as a limit of Lang's currents \cite[page 711]{UL}.} theory encompassing also currents of infinite mass.

To understand the idea, let us first examine once again any $k$ dimensional oriented manifold $M\subset \reals^d$.
The manifold induces a multilinear functional:
\begin{equation*}
    \llparenthesis M \rrparenthesis (f,\pi^1, \dots, \pi^k) = \int_M f \diff \pi^1\wedge\dots \wedge\diff \pi^k.
\end{equation*}
Apriori, the expression makes sense if the functions $f, \pi^1, \dots, \pi^k$ are smooth and $f$ is, moreover, compactly supported.
However, the formula above can be written in more elementary terms.
If $F_1, \dots, F_k$ are some (smooth) unit vector fields providing the orientation of $M$, then we have
\begin{equation*}
    \llparenthesis M \rrparenthesis (f,\pi^1, \dots, \pi^k) = \int_M f\det(\frac{\partial \pi^i}{\partial F_j})_{i,j=1}^k \diff \H^k,
\end{equation*}
where the integral is a Lebesgue integral with respect to the Hausdorff $k$-measure $\H^k$.
By the Rademacher differentiation theorem, the expression above is well defined whenever $f$ is Lipschitz and compactly supported and $\pi^1, \dots, \pi^k$ are Lipschitz.

For a metric space $X$, we denote by $\tLip(X)$ the space of real valued Lipschitz maps defined on $X$ and by $\tLip_c(X)$ the subspace of those maps which are compactly supported.
The multilinear functional
\begin{equation*}
    \llparenthesis M \rrparenthesis \colon \tLip_c(\reals^d) \times [\tLip(\reals^d)]^k \to \reals
\end{equation*}
has several properties which are used as axioms to define metric currents.
Precisely, if $X$ is a locally compact\footnote{The theory may be extended into spaces which are not locally compact, but for us the locally compact setting suffices.} metric space, we call the \emph{multilinear} functional
\begin{equation*}
    T\colon \tLip_c(X) \times [\tLip(X)]^k\to \reals
\end{equation*}
a \emph{metric current}, if
\begin{enumerate}
    \item (locality) whenever there is some $j\in\{1, \dots, k\}$ such that $\pi^j$ is constant on a neighbourhood of $\spt f$, then $T(f, \pi^1, \dots, \pi^k)=0$,
    \item (joint continuity) if $f_i \to f$, $\pi^j_i \to \pi^j$ for each $j\in\{1, \dots, k\}$, pointwise, with a uniform bound on Lipschitz constants, 
    and there is some compact set $K$ such that $\spt f_i, \spt f \subset K$, then
    \begin{equation*}
        T(f_i, \pi^1_i, \dots, \pi^k_i)\to T(f, \pi^1, \dots, \pi^k).
    \end{equation*}
\end{enumerate}
The functional $\llparenthesis M \rrparenthesis$ is indeed a metric $k$-current in $\reals^d$.
We see immediately that locality is satisfied, while the fact that joint continuity holds is non-trivial, but well known \cite[Theorem 2.16]{AFP}.
The currents which we introduced are equivalent to the definition of Lang \cite[Lemma 2.2]{UL}.

To understand appropriately the statement of the Flat Chain Conjecture, we need to introduce also the so-called flat chains.
However, the general case is quite involved and therefore we postpone the relevant definitions to Subsection \ref{SS:prel-currents}.
We continue here by considering only the top-dimensional case of $d$-currents in $\reals^d$, which is the most important for this paper.
Fortunately, this case also requires no knowledge of the Federer--Fleming theory of currents and we may restrict our attention to only classical distributions.

Every metric $d$-current $T$ in $\reals^d$ defines a distribution $\tilde{C}_d T$ via the formula
\begin{equation*}
    \tilde C_d T(\varphi) = T(\varphi, x^1, \dots, x^d) \quad \text{for $\varphi\in C^\infty_c(\reals^d)$.}
\end{equation*}
Here, for each $j=1, \dots, d$, the symbol $x^j$ stands for the coordinate function $x\mapsto x^j$ and $C^\infty_c(\reals^d)$ is the space of smooth compactly supported functions.
It is instructive to verify that if $M\subset \reals^d$ is a $d$-dimensional manifold with the canonical orientation, then $ \tilde C_d \llparenthesis M\rrparenthesis $ is the distribution given by the $L^1_\textnormal{loc}$-function $\chi_M$ (characteristic function of $M$).

Using several well-established identifications (see Subsection \ref{SS:prel-currents}), the top-dimensional case of the Flat Chain Conjecture of Lang \cite[page 716]{UL} may be reduced to the following statement.
For the general statement, see Conjecture \ref{Con:Lang-general}.
The symbol $\mathcal{L}^d$ stands for the Lebesgue measure.
\begin{conjecture}\label{Con:Lang-TD}
    Let $T$ be a metric $d$-current on $\reals^d$ supported inside $Q=[0,1]^d$.
Then there exists some $u\in L^1(Q)$ such that
    \begin{equation*}
        \tilde C_d T(\varphi) = \int_Q u\varphi \;\diff \mathcal{L}^d \quad \text{for all $\varphi \in C^\infty_c(\reals^d)$.}
    \end{equation*}
\end{conjecture}
In other words, the conjecture states that if the space of metric $d$-currents supported inside $Q$ is identified with a space of distributions (supported inside $Q$), then the resulting space of distributions coincides with $L^1(Q)$.
We also remark that the converse statement is known to be true \cite[Theorem 2.16]{AFP} \cite[Proposition 2.6]{UL}, i.e.~given $u\in L^1(Q)$, the functional
\begin{equation*}
    T_u(f,\pi^1, \dots, \pi^d)= \int_Q u f \det\Diff\pi\;\diff \mathcal{L}^d
\end{equation*}
is a metric $d$-current and clearly $\tilde{C}_{d} T_u = u$.
That is, $u\mapsto T_u$ is a restriction of the inverse of $\tilde{C}_d$.

DePhilippis and Rindler \cite{DeRindler} show that the conjecture holds if one also, additionally, assumes that there exists some finite Borel measure $\mu$ on $Q$ such that
\begin{equation*}
    |T(f,\pi^1, \dots, \pi^d)|\leq \int_Q |f| \diff \mu \prod_{j=1}^d \tLip(\pi^j).
\end{equation*}
This is equivalent to requiring that the distribution $\tilde C_d T$ be a signed measure.
In other words, any signed measure $\mu$ that arises from a metric $d$-current $T$ as $\tilde{C}_d T = \mu$ must satisfy $\mu \ll \mathcal{L}^d$.

The principal result that we obtain (and from which the complete resolution of Lang's conjecture in all dimensions and codimensions is obtained) is that Conjecture \ref{Con:Lang-TD} fails whenever $d\geq 2$, which is explicitly available in Corollary \ref{C:failure-of-everything}.
The proof is non-constructive and so we are not able to demonstrate an explicit example of such a metric $d$-current.
The proof in fact relies on two non-constructive steps; we use the open mapping theorem as well as a version of the Hahn-Banach separation theorem.
Before describing how Conjecture \ref{Con:Lang-TD} is linked to the prescribed Jacobian equation in Subsection \ref{SS:intro-char}, we discuss the equation itself.

\subsection{The prescribed Jacobian equation and Lipschitz regularity}\label{SS:intro-reg}

Due to its geometric meaning, the Jacobian appears when dealing with many different types of problems in mathematical analysis.
Here we point out a couple of connections that are the most relevant for the present paper.
The prescribed Jacobian equataion is of interest in the Calculus of Variations and in elasticity, for which a standard reference is the book of Dacorogna \cite[Chapter 14]{Daco}.
Specifically in elasticity, a lower estimate on the Jacobian is a natural constraint, since zero Jacobian corresponds to compressing the underlying material infinitely, which would require infinite energy \cite[(2.7)]{Ball}.
In the recent work of Bevan et al \cite{BKV}  a constrained minimisation problem is studied and the constraint is given by the equation \eqref{E:PJE1}.
Namely, a functional $\mathcal{F}(\pi)$ is minimised over vector fields $\pi$ which solve the equation \eqref{E:PJE1} for a fixed right hand side $\rho$.

Then there is the link of the prescribed Jacobian equation to the theory of optimal transport (via the Monge-Amp\`ere equation \cite[Lecture 7]{ABS}).
Interestingly, optimal transport is also directly connected to metric currents - in particular (normal) metric $1$-currents, which can be seen as transporting the positive part of their boundary to the negative part.
This connection is explored in the recent works of Arroyo-Rabasa and Bouchitt\'e \cite{ArBou} and, independently, the paper of the author jointly with Bate, Caputo, Valentine and Wald \cite{BCTVW}. Similar connection is also explored in the work of De Pauw \cite{dePauw}.

Finally, there is the somewhat surprising connection of the prescribed Jacobian equation to large-scale geometry.
This was first observed independently by Burago and Kleiner \cite{BK} and McMullen \cite{McM} and is of particular interest to our work.
Recall that $Z\subset \reals^d$ is a \emph{separated net}, if there are two constants $\varepsilon>0$ and $\delta>0$ such that for any $x\in \reals^d$, there is $z\in Z$ with $\lVert x-z\rVert\leq \varepsilon$ and whenever $z,w\in Z$, $z\not= w$, then $\lVert z - w \rVert\geq \delta$.
Burago and Kleiner show that existence of \emph{biLipschitz solutions} (i.e.~Lipschitz, injective solutions, with Lipschitz inverses) to \eqref{E:PJE1} given a right hand side $\rho\in L^\infty(Q)$, corresponds to the existence of biLipschitz bijections $F\colon Z \to \mathbb{Z}^d$, where $Z$ is a separated net.

They continue to show that there exists $\rho\in L^\infty(Q)$, which, moreover, only attains two positive values, such that there is no biLipschitz solution to \eqref{E:PJE1}.
This is used to infer that there exists a separated net $Z$, for which there is no biLipschitz bijection $F\colon Z \to \mathbb{Z}^d$.
Existence of such bijections was previously an open question of Gromov.

Later, these ideas were refined  by Dymond et al \cite{DKK}.
They study the pushforward equation
\begin{equation}\label{E:intro-pw-equation}
    \pi_\# \mathcal{L}^d_{|Q}=\rho \mathcal{L}^d_{|\pi(Q)}.
\end{equation}
Note that if $\pi$ is injective, then \eqref{E:intro-pw-equation} is the same as \eqref{E:PJE1} by the change of variables formula.
They show that there exists $\rho \in L^\infty(Q)$ (once again, it may be assumed that $\rho$ only attains two positive values) such that \eqref{E:intro-pw-equation} admits no \emph{Lipschitz regular solutions} (for the relevant definitions, see the paper \cite{DKK}).
This is used to assert a statement about Lipschitz maps from $n^d$ point subsets of $\mathbb{Z}^d$ to the regular $n^d$ square in $\mathbb{Z}^d$ as $n\to\infty$.

Concerning the (lack of) regularity of the equation \eqref{E:PJE1}, there are multiple ``scales'' of relevant questions.
Since the equation is of first order, one may for example ask, whether given $\rho\in L^p$, there is a solution $\pi \in W^{1,p}$.
This was resolved affirmatively in the case $d<p<\infty$ by Ye \cite{Ye}.
In the less regular regime of $p<d$, the Jacobian loses its geometric meaning and has very pathological behaviour.
This can be seen in the work of Koumatos et al \cite[1.2]{KRW}.
A particular special case of their result asserts that a Sobolev map $\pi$ can be found which maps $Q=[0,1]^d$ to $2Q=[0,2]^d$ satisfying
\begin{equation*}
    \det\Diff\pi = 1\quad\text{in $Q$} \quad\text{and} \quad \pi_{|\partial Q}=2 \Id.
\end{equation*}
In fact, the map $\pi$ may be required to lie in any Sobolev space $W^{1,p}$ for $p<d$.
This is a kind of infinitesimal analogue of the Banach--Tarski paradox, allowing one to double the global volume, while on the infinitesimal scale the volume stays preserved.
In a similar fashion, an earlier construction of Hencl \cite{Standa} produces, for any $p<d$, a \emph{homeomorphism} $\pi\in W^{1,p}$ such that $\det\Diff\pi=0$ a.e.

Given the scaling behaviour of determinants, it is also natural to ask whether for a $\rho \in L^{p}$ there exists a solution in $W^{1, pd}$.
In fact, the previously mentioned results of Koumatos et al resolve this affirmatively when $p<1$ (i.e.~$pd<d$).
For $p=1$ (i.e.~$pd=d$), this is known to fail due to Coifman et al \cite{CLMS} (in French), who show that whenever $\pi\in W^{1,d}$, then $\det\Diff\pi$ is actually a member of a Hardy space which is known to be strictly smaller than $L^1$ (see also the earlier work of M\"uller \cite{Muller}).
In the case $p=\infty$ the question of existence of solutions in $W^{1,pd}$ and $W^{1,p}$ agree and simply reduce to the question of existence of Lipschitz solutions for $L^\infty$ right hand sides.

Another natural question is whether given $\rho \in C^{k,\alpha}$ (H\"older space), there is a solution $\pi \in C^{k+1,\alpha}$, for which an affirmative answer was found by Dacorogna and Moser \cite{DM} if $\alpha\in(0,1)$.
However, the question of Lipschitz (i.e.~$C^{0,1}$) solutions given a $\rho\in L^\infty(Q)$ or even $\rho\in C(Q)$ remained open, although Riviere and Ye \cite{RiYe} find less regular H\"older solutions.
We remark that it is sometimes also of interest to construct solutions which are invertible and their inverses also have satisfactory regularity (in particular, in some of the reference given above).

We are interested in Lipschitz solutions for $L^\infty$ right hand sides and, motivated by resolving the Flat Chain Conjecture, we actually consider a generalisation of the equation \eqref{E:PJE1}, given by the following ``linearisation'':
\begin{equation}\label{E:intro-the-equation1}
    \sum_i f_i \det\Diff\pi_i = \rho \quad \text{a.e.~on $Q=[0,1]^d$.}
\end{equation}
Here, $\rho\in L^\infty(Q)$ is a prescribed function and $f_i, \pi_i$ are unknown, possibly infinite, sequences of Lipschitz maps.
Naturally, $\pi_i$ are vector fields and $f_i$ are scalar fields.

\begin{definition}[Lipschitz solutions]\label{D:solution}
    Given a function $\rho\in L^\infty(Q)$,
    a sequence of Lipschitz maps $f_i\colon Q \to \reals$ and $\pi_i \colon Q \to \reals^d$ is called a \emph{Lipschitz solution to \eqref{E:intro-the-equation1} (with the right hand side $\rho$)} if
    \begin{equation}\label{E:intro-inequality}
        \sum_i\max\{\tLip(f_i), \lVert f_i \rVert_\infty\} \prod_{j=1}^d \tLip(\pi_i^j)<\infty,
    \end{equation}
    and \eqref{E:intro-the-equation1} holds with absolute convergence in $L^\infty(Q)$.
\end{definition}
Note that the estimate \eqref{E:intro-inequality} immediately implies that the sum on the left hand side of \eqref{E:intro-the-equation1} converges absolutely in $L^\infty$.

We now continue by defining three notions of \emph{non-regularity} which we will study.
It is the case that in order to resolve Conjecture \ref{Con:Lang-TD}, we only need to consider one of these notions (strong Lipschitz non-regularity), however, it is very instructive to discuss all three.

\begin{definition}\label{D:Lip-reg}
    \begin{enumerate}
        \item[(1)] We say that the prescribed Jacobian equation has \emph{weak Lipschitz non-regularity} if there exists some $\rho\in L^\infty(Q)$ for which there does not exists a Lipschitz vector field $\pi$ solving the equation \eqref{E:PJE1} in the sense of equality almost everywhere.
        \item[(2)] We say that the prescribed Jacobian equation has \emph{linearised Lipschitz non-regularity} if there exists some $\rho\in L^\infty(Q)$ for which there does not exists a Lipschitz solution $(f_i, \pi_i)$ to \eqref{E:intro-the-equation1}.
        \item[(3)] We say that the prescribed Jacobian equation has \emph{strong Lipschitz non-regularity} if a sequence of $L^1(Q)$ functions $u_n$, $n\in \nat$ can be found satisfying the two following inequalities simultaneously:
        \begin{enumerate}
        \item\label{Enum:strong-non-reg1} for each $n\in \nat$, $\lVert u_n \rVert_{L^1(Q)}\geq 1$
        \item \label{Enum:strong-non-reg2} for each $n\in \nat$,
        \begin{equation*}
            \sup\left\{\left|\int_Q u_n f\det\Diff\pi\;\diff \mathcal{L}^d\right|: \max\{\tLip(f), \lVert f \rVert_\infty\} \prod_{j=1}^d\tLip(\pi^j)\leq 1 \right\}\leq \frac{1}{n}.
        \end{equation*}
    \end{enumerate}
    \end{enumerate}
\end{definition}
It is clear that linearised Lipschitz non-regularity implies weak Lipschitz non-regularity.
To see that strong Lipschitz non-regularity implies linearised Lipschitz non-regularity we essentially ``repeat'' the proof of the open mapping theorem in a different setting.
Let
\begin{equation*}
    K= \{f\det\Diff \pi: \max\{\tLip(f), \lVert f \rVert_\infty\} \prod_{j=1}^d \tLip(\pi^j)\leq 1\}\subset L^\infty(Q).
\end{equation*}
If the prescribed Jacobian equation does \emph{not} have linearised Lipschitz non-regularity, then in particular
\begin{equation*}
    \bigcup_{n\in\nat} \{n\rho: \rho \in \overline{\conv K}\}= L^\infty(Q).
\end{equation*}
Thus, an application of the Baire category theorem shows that for at least one $n\in \nat$, the set
\begin{equation*}
    \{n\rho: \rho \in \overline{\conv K}\}
\end{equation*}
has non-empty interior.
By the homogeneity of the determinant, we may freely rescale the set and see that in fact the set $\overline{\conv K}$ contains an open neighbourhood of $0$.
By definition, therefore, some
$\varepsilon>0$ exists for which
\begin{equation*}
    \overline{\conv K} \supset \{\rho \in L^\infty(Q): \lVert \rho \rVert_{\infty}\leq \varepsilon\}.
\end{equation*}
It follows that for any $u\in L^1(Q)$,
\begin{equation*}
    \varepsilon\lVert u \rVert_1\leq \sup_{\rho \in \overline{\conv K}} \left|\int u\rho \;\diff \mathcal{L}^d\right|=\sup_{\rho \in \conv K} \left|\int u\rho \;\diff \mathcal{L}^d\right|=\sup_{\rho \in K} \left|\int u\rho \;\diff \mathcal{L}^d\right|,
\end{equation*}
where in the first equality we use that the functional $\rho \mapsto \int u \rho $ is continuous and in the second equality that it is linear.
By definition of $K$, this contradicts strong Lipschitz non-regularity.

It is not known to us whether there exists a simple way of proving strong Lipschitz non-regularity under the assumption of linearised Lipschitz non-regularity, but in light of the following result, this is not needed.

\begin{theorem}\label{T:intro-reg}
    The prescribed Jacobian equation has strong Lipschitz non-regularity.
Therefore, it also has linearised Lipschitz non-regularity and weak Lipschitz non-regularity.
\end{theorem}

We actually also prove linearised Lipschitz non-regularity more ``directly'' (it follows immediately from Theorem \ref{T:nowhere-dense} together with Baire category theorem, or from Corollary \ref{C:failure-of-everything} together with the fact that strong non-surjectivity implies non-surjectivity; cf.~Section \ref{S:HB}).

Before we discuss some aspects of the proof, we discuss related results in literature.
The lack of regularity of PDEs near $L^\infty$ as well as near $L^1$ is not something specific to our particular case.
Consider the so-called Ornstein's non-inequality \cite{Orn} (see also the work of Kirchheim and Kristensen \cite[Theorem 1.3]{KK} for a more general statement with a different proof).
One possible way to view the inequality is as saying that under determined PDEs lack regularity gain near $L^1$.
In more explicit way, if we have an under determined linear differential operator of order $k$, say $\mathscr{A}$, then for the equation
\begin{equation*}
    \mathscr{A} u =f,
\end{equation*}
we may find a sequence $f_i$ bounded in $L^1$ for which any sequence of solutions $u_i$ must satisfy $\lVert \nabla^k u_i \rVert_{L^1}\to \infty$.
In elasticity theory, this is of interest when $\mathscr{A}$ is the symmetric part of the gradient (cf.~Korn's inequality \cite{Horgan} asserting instead that for $p\in(1,\infty)$ a regularity gain can be expected for the symmetric part of the gradient).
In the theory of currents, it is useful to consider $\mathscr{A}=\diff$, the exterior derivative.
Except for the case when $\diff$ coincides with the Frech\'et derivative, the operator is indeed under determined.

The non-inequality of Ornstein can be dualised to obtain an analogous result near $L^\infty$.
It appears that the first instance of an $L^\infty$ result is due to De Leeuw and Mirkil \cite{dLM} (in French) and came before Ornstein's result.
A duality argument in combination with Ornstein's result is carried out by McMullen \cite{McM} who uses this idea to show lack of Lipschitz solutions for the \emph{divergence equation}:
\begin{equation*}
    \operatorname{div} u = f,
\end{equation*}
for some right sides $f\in L^\infty$.
This result was, however, first obtained by Preiss \cite{PreissDiv} using a different technique.
Bouragin and Brezis \cite{BB} provide a number of results for the divergence equation in various limiting cases and in particular provide a simpler duality argument, akin to McMullen's approach, showing the lack of Lipschitz solutions.
It appears that a similar method could be used for any under determined linear differential operator, but this is not available explicitly in literature.

Our Theorem \ref{T:intro-reg} may thus be seen as a multilinear (as opposed to linear) version of these types of statements.
It is clear that linearised Lipschitz non-regularity is a statement about lack of regularity near $L^\infty$.
On the other hand, it is not clear, whether strong Lipschitz non-regularity is a result ``near $L^\infty$'' or ``near $L^1$''.
It concerns $L^1$ maps, but no operator on acting on these maps is really considered.
In fact, to obtain an operator of this kind, theory of metric currents needs to be developed.
Or rather, a development of such an operator leads immediately to the development of metric currents and eventually to the resolution of Lang's Flat Chain Conjecture.

It is also not clear whether our result is strictly stronger than available linear results (say, for operators of order $1$ near $L^\infty$).
However, a large scale of linear operators can be considered and non-regularity results for them obtained using our multilinear result.
An example of these types of operators is given by the exterior derivative $\diff$ acting on $k$-forms for $k\geq 1$.
This is further discussed in Section \ref{S:prescribed-exterior} where we also provide the relevant results.
It appears that the results we collect concerning the operator $\diff$ are ``colloquially known'' since it should be possible to obtain them by dualizing the result of Ornstein as we discussed above.
But these results are not available explicitly, thus we provide them together with detailed proofs.

We continue by outlining the basic ideas behind Theorem \ref{T:intro-reg}.
There are, broadly speaking, two main steps that we need to carry out.
The first is purely functional-analytic and it is here that we rely on the Hahn-Banach separation theorem.
The second step is proving the following proposition.

\begin{proposition}\label{P:intro-strong-non-surj}
    For any $S>0$, there exists a smooth map $\rho_0 \in L^\infty(Q)$ with $\lVert \rho_0 \rVert_\infty\leq 2$ together with a weak$^*$ open set $U$ with $\rho_0\in U$ and such that for every $\rho \in U$, there is no solution $(f_i, \pi_i)$ to the equation \eqref{E:intro-the-equation1} satisfying
    \begin{equation}\label{E:leqS}
        \sum_i\max\{\tLip(f_i), \lVert f_i \rVert_\infty\} \prod_{j=1}^d \tLip(\pi_i^j)\leq S.
    \end{equation}
\end{proposition}
An important remark is that the reason we require $\rho_0$ to be smooth is only to be certain that there exists \emph{some} Lipschitz solution to \eqref{E:intro-the-equation1} with right hand side $\rho_0$.

Once we obtain a proof of Proposition \ref{P:intro-strong-non-surj}, the Hahn-Banach theorem is used to construct a functional separating the set $U$ from the set of $L^\infty(Q)$ maps for which there exists a solution satisfying \eqref{E:leqS}.
Due to working with the weak$^*$ topology, this separating functional lies in $L^1(Q)$ and as we increase $S\to \infty$ we obtain the sequence required in the definition of strong Lipschitz non-regularity.

The proof of Proposition \ref{P:intro-strong-non-surj} is overall quite similar to the proof, due to Burago and Kleiner \cite{BK}, that the prescribed Jacobian equation lacks \emph{biLipschitz} solutions.
Therefore it is also quite similar to the main result concerning the pushforward equation of Dymond at el \cite{DKK}.
However, four problems arise, some more significant than others, requiring new ideas that need to be implemented in the entire procedure.
As these adjustments cannot really be separated from the rest of the argument, we provide in Section \ref{S:reg} a complete detailed proof, even if it has large overlaps with the aforementioned works.
Here we make the best attempt to point out the changes made to the argument, without getting into too many details.
Our goal is to construct some $\rho$ having the prescribed properties together with a weak$^*$ open neighbourhood $U$.
The changes done to the construction of $\rho$ from \cite{BK} and \cite{DKK} can broadly be characterised as
\begin{enumerate}
    \item\label{Enum:Lip} considering Lipschitz maps as opposed to biLipschitz,
    \item\label{Enum:sums} considering possibly infinite sums of Jacobians,
    \item\label{Enum:coefficients} allowing the Jacobians $\det\Diff\pi_i$ to be multiplied with the coefficients $f_i$ (which are Lipschitz functions),
    \item\label{Enum:open} finding an entire weak$^*$ open set of right hand sides $\rho$ as opposed to a single one.
\end{enumerate}

The first and the second item are the most complicated to deal with.
First we consider item \ref{Enum:sums}.
An intermediate results of Dymond et al. \cite[Lemma 3.1]{DKK}  essentially deals with the case of finite sums of Jacobians:
\begin{equation}\label{E:fin-sum}
    \sum_{i=1}^k \det\Diff \pi_i,
\end{equation}
where we need a control over both the number $k$ and the biLipschitz constants of $\pi_i$'s.
The definition of $g$ on page 614 of their paper is where the ``sums'' implicitly appear and our counterpart to this idea is in Lemma \ref{L:def-h-into-ell2}.
One of the adjustments we make is that we take their central intermediate results \cite[Lemma 3.3]{DKK} for maps from $\reals^d \to \reals^n$ with $n\geq d$ and improve it to hold even for maps from $\reals^d \to \ell^2$, the separable Hilbert space, which allows us to consider infinite sums instead of finite ones.
This is available in Lemma \ref{L:dichotomy}.

To deal with item \ref{Enum:Lip} we must find a way to make the argument work for Lipschitz maps as opposed to biLipschitz.
This problem appears at two separate instances of the argument.
Firstly, the biLipschitz constant appears in a certain quantitative estimate concerning the sum \eqref{E:fin-sum} (which in our case may be infinite).
This can be resolved by a simple trick, which we illustrate in the following.
Suppose we have a Lipschitz map $F\colon \reals^d \to \reals^n$.
Then the map
\begin{equation*}
    \tilde F \colon \reals^d \to \reals^{d+n};\quad \tilde F(x)= (x, F(x)),
\end{equation*}
is biLipschitz.
Remarkably, the entire argument which we need to carry out is completely invariant under such modifications of functions (which we carry out in Lemma \ref{L:def-h-into-ell2}, where the first $d$ variables correspond to the identity).
Moreover, the quantitative estimates needed, which in the approach of Dymond et al \cite{DKK} depend on the biLipschitz constants, can be made to only depend on the Lipschitz constants.

The second instance where biLipschitzness appears is of qualitative nature.
Let $Q$ be a cube inside $\reals^d$ of side-length $r>0$ and $\pi, \kappa \colon Q \to \reals^d$ $L$-Lipschitz maps.
Then
\begin{equation*}
    |\int_Q \det\Diff\pi - \int_Q\det\Diff\kappa|\leq C r^{d-1} \lVert \pi-\kappa\rVert_\infty L^{d-1}.
\end{equation*}
This is shown already by Burago and Kleiner \cite{BK} when $\pi, \kappa$ are also homeomorphisms (equivalently injective) using a kind of topological argument.
The same argument is also used by Dymond et al \cite{DKK}.
However, as it turns out, the exact same estimate (with possibly a different constant) may be obtained for general Lipschitz maps by a use of Stokes theorem - see Lemma \ref{L:Average-det} (cf. \cite[Proposition 5.1]{AKcur}).

The item \ref{Enum:coefficients} is resolved fairly simply.
The right hand sides $\rho$ which are constructed cause a problem by oscillating very fast on small scales.
These oscillations cannot be disrupted (i.e.~cancelled out) by the coefficients $f_i$, because they are required to be Lipschitz.
This idea can be seen in the estimate in Lemma~\ref{L:est-on-squares-with-coef}.

Finally, item \ref{Enum:open} is implicitly already resolved in the original paper of Burago and Kleiner (although it is likely very difficult to see until the correct perspective arising from weak$^*$ topology is taken).
The point is that a map $\rho$ is constructed, for which there are no solutions with controlled constants (in our case, this means solutions satisfying the estimate \eqref{E:leqS}).
However, the way it is \emph{proven} that $\rho$ has this property only relies on asserting that the quantities
\begin{equation*}
    \int_{Q_i} \rho
\end{equation*}
for some \emph{finite} collection of cubes $Q_i$, $i=1, \dots, N$ lie in some open intervals $I_i\subset \reals$.
Such a condition can be stated as belonging to the weak$^*$ open set
\begin{equation*}
    U=\bigcap_{i=1}^N\{\rho: \int_{Q_i} \rho\in I_i\}.
\end{equation*}
Then, the proof that \emph{any} element of $U$ has the same property (lack of a solution with good constants) is exactly the same as it is for the original $\rho$ which is constructed.
See Lemma \ref{L:smoothing} for the explicit definition of $U$.

\subsection{Connecting regularity to the Flat Chain Conjecture and further discussion}\label{SS:intro-char}

The fact that there is \emph{some} connection between metric currents and Jacobians is not deep, as it can be immediately seen from \cite[Theorem 5.5]{UL}, nor new.
The first instance where this is explored seems to be due to De Lellis \cite{DeL}, where a theory of Jacobians is developed in a metric space using the theory of metric currents.
Our observations, however, link Jacobians to metric currents closely enough to allow us to obtain invalidity of Lang's Flat Chain Conjecture.
The most precise statement we are able to make about this connection is Theorem \ref{T:f-a-characterisation}, special case of which follows.

\begin{theorem}\label{T:intro-char}
    Suppose $d\in \nat$.
If the prescribed Jacobian equation has strong Lipschitz non-regularity, then the Conjecture \ref{Con:Lang-TD} fails.
\end{theorem}

The proof of Theorem \ref{T:intro-char} is non-constructive, as it relies on the open mapping theorem.
The basic idea is that if one attempts to ``forcefully'' dualise the statement of Conjecture \ref{Con:Lang-TD} (which is about metric currents and $L^1$ maps), one would expect to obtain a statement about $L^\infty$ maps and Jacobians of Lipschitz maps.
This is exactly what happens and the dual statement is essentially the negation of strong Lipschitz non-regularity.
By combining Theorem \ref{T:intro-char} and Theorem \ref{T:intro-reg}, we see that Conjecture \ref{Con:Lang-TD} fails whenever $d\geq 2$.
Using some basic theory of metric currents and flat chains, a complete resolution of the conjecture in all dimensions and codimensions is obtained.
Our main result is therefore the following.

\begin{theorem}\label{T:fcc-fails-all-cases}
     Suppose $d\in\nat$ and $k\in\nat\cup \{0\}$, $k\leq d$.
Then every compactly supported metric $k$-current in $\reals^d$ is a flat $k$-chain if and only if either $k=0$ or $k=d=1$.
 \end{theorem}

To conclude this subsection, we discuss different statements sharing the name ``Flat Chain Conjecture'' and the relationships between them.
We also give a historical overview of the very limited progress on these questions.
Basic knowledge of flat chains and metric currents is assumed (consult Section \ref{S:prel} for the required definitions).

The first instance of a Flat Chain Conjecture is due to Ambrosio and Kirchheim \cite[page 68]{AKcur}.
They conjectured (in our language), that every compactly supported metric $k$-current of \emph{finite mass} in $\reals^d$ is a flat $k$-chain.
We shall abbreviate this conjecture to $\operatorname{FCC}(\textnormal{AK})[k,d]$.
They also proved the converse statement that every flat $k$-chain is a metric current.
Later, the theory of metric currents was extended by Lang to include currents not of finite mass and an analogous conjecture was stated for this type of currents.
To be explicit, Lang conjectures \cite[page 716]{UL} that every compactly supported metric $k$-current in $\reals^d$ is a flat $k$-chain.
We shall abbreviate this conjecture to $\operatorname{FCC}(\textnormal{L})[k,d]$.
It is obvious that $\operatorname{FCC}(\textnormal{L})[k,d]$ implies $\operatorname{FCC}(\textnormal{AK})[k,d]$ for any $k$ and $d$.
This implication, in light of our results, is not very useful.

It appears that no results concerning $\operatorname{FCC}(\textnormal{L})[k,d]$ are available in the literature other than the description of the problem itself due to Lang.
On the other hand, $\operatorname{FCC}(\textnormal{AK})[k,d]$ has enjoyed much attention and has been resolved affirmatively in some non-trivial cases.
The case of $\operatorname{FCC}(\textnormal{AK})[1,d]$ is known to hold due to Schioppa \cite[Theorem 1.6]{Sch} (with simpler proofs appearing later due to Marchese and Merlo \cite{MM}, Arroyo-Rabasa and Bouchitte \cite{ArBou} and Bate et al \cite{BCTVW}).
As we already stated in Subsection \ref{SS:intro-currents}, $\operatorname{FCC}(\textnormal{AK})[d,d]$ was proven for every $d\in\nat$ by DePhilippis and Rindler \cite{DeRindler}.
Remarkably, their approach also relies on PDEs (although, at least seemingly, a different type of a PDE problem is resolved).

It also appears that a number of PDE type conditions are known to some researchers, which would be sufficient (but not necessary) for the validity of $\operatorname{FCC}(\textnormal{AK})[k,d]$, although not all of these findings are published. A PDE type condition is showed to be sufficient in Marchese and Merlo \cite{MM}. Implicitly, this type of condition can also be seen in the aforementioned work of Arroyo-Rabasa and Bouchitt\'e \cite{ArBou} and Bate et al \cite{BCTVW}. It should be stated that these different groups (including the present author) realized these types of connections to PDEs independently as is also indicated in a recent survey by Marchese \cite{MarcheseSurvey}.
The main problem is that in most cases the PDE type conditions are actually known to fail when considering pointwise estimates. One contribution of the present paper is in showing the \emph{necesssity} of a pointwise PDE type condition for a Flat Chain Conjecture, even if one must consider the conjecture without the finite mass condition.

On the other hand, the results of Marchese and Merlo \cite{MM} and De Masi and Marchese \cite{DeMaMa} provide a plausible sufficient PDE type condition for $\operatorname{FCC}(\textnormal{AK})[k,d]$. The main point is that they do not require finding a Lipschitz map solving the underlying PDE in a pointwise sense, but instead only in a very weak measure-theoretic sense.

Finally, it should be stated that in all cases in which we disprove $\operatorname{FCC}(\textnormal{L})[k,d]$ we actually produce a metric $k$-current with compact support but \emph{infinite flat norm}.
Thus one might conjecture that any compactly supported metric $k$-current in $\reals^d$ with finite flat norm is a flat chain.
We shall refer to this conjecture as $\operatorname{FCC}_{FF}(\textnormal{L})[k,d]$, the letters $FF$ standing for ``finite flat''.
It is natural to expect that $\operatorname{FCC}_{FF}(\textnormal{L})[k,d]$ would be very closely related to $\operatorname{FCC}(\textnormal{AK})[k,d]$, however, except for the obvious implication, we are not able to assert any general relationship between the two conjectures.
The exception to this is the top dimensional case $k=d$, where the flat norm coincides with the mass norm, and therefore $\operatorname{FCC}_{FF}(\textnormal{L})[d,d]$ is obviously equivalent to $\operatorname{FCC}(\textnormal{AK})[d,d]$ and they both hold due to DePhilippis and Rindler \cite{DeRindler}.

\subsection{Structure of the paper}

The paper is structured as follows.
In Section \ref{S:prel} we collect some basic background material which will be needed later.
In Section \ref{S:char} we develop basic theory of metric currents and potential forms (which arise when we try to linearise the prescribed Jacobian equation).
It is here where we obtain a proof of Theorem \ref{T:intro-char}.
Section \ref{S:HB} is dedicated to proving a version of the Hahn-Banach separation theorem, where it is our goal to be able to separate sets in a dual Banach space with functionals that lie in its predual.
This section is purely functional-analytic and the results work for general dual Banach spaces, although we only need to apply it to $L^\infty([0,1]^d)$.
In Section \ref{S:reg} we obtain all of the relevant non-regularity results and in particular the proof of Proposition \ref{P:intro-strong-non-surj} and Theorem \ref{T:intro-reg}.
The combination of these results is used to assert that Conjecture \ref{Con:Lang-TD} fails.
In Section \ref{S:all-dim} we develop some basic theory of flat chains and use it to obtain a proof of Theorem \ref{T:fcc-fails-all-cases}.
The very short Section \ref{S:prescribed-exterior} is dedicated to obtaining a Lipschitz non-regularity result for the prescribed exterior derivative equation.

\subsection{Acknowledgements}
I would like to express my gratitude to David Bate for bringing the Flat Chain Conjecture to my attention and for many fruitful discussions and careful reading of the manuscript.
I am grateful to Filip Rindler for his deep insights into the (lack of) regularity of PDEs and a number of discussions where he pointed out particularly interesting results in the literature.
Finally, I wish to thank Michael Dymond for discussing with me some of the more technical details of his previous joint work with Kalu\v{z}a and Kopeck\'a \cite{DKK}.

The author is supported by the Warwick Mathematics Institute Centre for Doctoral Training and by the European Union’s Horizon 2020 research and innovation programme (Grant agreement No.
948021)
\section{Preliminaries}\label{S:prel}

\subsection{Classical and metric currents}\label{SS:prel-currents}

\paragraph{Classical currents and flat chains}
We shall adopt the notation and conventions of Lang \cite{UL} with some slight adjustments.
For a more detailed overview of currents, see the book of Federer \cite{Federer}

Let $d\in \nat$, $k\in\nat\cup \{0\}$.
We denote by $e_1, \dots, e_d$ the canonical basis of $\reals^d$ and by $\diff x_1, \dots, \diff x_d$ the dual basis of $(\reals^d)^*$.
The exterior product spaces $\Wedge_k \reals^d$ and $\Wedge^k \reals^d$ are equipped with the \emph{mass} and the \emph{comass} norms respectively.
We recall that if $v_1, \dots v_k\in \reals^d$ have norm at most one, then the mass norm of $v_1 \wedge \dots \wedge v_k$ is at most one.
Analogous statement holds for the comass norm.
We identify $\Wedge^d \reals^d = \reals$ via the map
\begin{equation}\label{E:td-ext-prod}
    \lambda \mapsto \lambda \diff x_1 \wedge\dots \wedge \diff x_d.
\end{equation}

We recall the following fact concerning the size of pairings.

\begin{lemma}\label{L:prel-multivector-pairing-est}
    If $v^*_1, \dots, v^*_k \in (\reals^d)^*$ and $v_1, \dots, v_k \in\reals^d$, then
    \begin{equation*}
        |\langle v_1^* \wedge \dots \wedge v_k^*, v_1\wedge \dots \wedge v_k\rangle|\leq \prod_{j=1}^k\lVert v_j^*\rVert \lVert v_j\rVert,
    \end{equation*}
    where the norms on the right hand side are the Euclidean norm of $\reals^d$ and the dual (also Euclidean) norm of $(\reals^d)^*$.
\end{lemma}

For a map $f\colon \reals^n \to \reals^m$, the symbol $\Diff f (x)\colon \reals^n \to \reals^m$ stands for the Frech\'et differential of $f$ at $x$.
In particular, if $m=1$, then $\Diff f(x)\in(\reals^n)^*$.
For a differential $k$-form $\omega$, we denote its exterior derivative by $\diff \omega$.
In particular, if $\omega$ is a $0$-form, then $\diff \omega = \Diff \omega$.

We denote by $C^\infty_c(\reals^d, \Wedge^k\reals^d)$ the set of smooth and compactly supported differential $k$-forms on $\reals^d$.
Recall that a linear functional
\begin{equation*}
    T\colon C^\infty_c(\reals^d, \Wedge^k\reals^d) \to \reals
\end{equation*}
continuous in the sense of distributions is called a $k$-current.
The linear space of all $k$-currents in $\reals^d$ is denoted by $\overline{\mathcal D}_k(\reals^d)$.
We say that $T\in\overline{\mathcal{D}}_k(\reals^d)$ is of \emph{finite mass} if it is represented as an integration against a $\Wedge_k\reals^d$-valued measure.
If $T$ is of finite mass, then the induced measure is determined uniquely, its total variation measure, $\mu$, is called the \emph{mass measure} and we denote $\mu(\reals^d)=\lVert T \rVert$.
The quantity $\lVert T \rVert$ is called the \emph{mass} or \emph{mass norm} of $T$.

For a $k$-current $T$ ($k\geq 1$), we denote by $\partial T$ the boundary of $T$ which we recall is defined via:
\begin{equation*}
    \partial T(\omega) = T(\diff \omega) \quad \text{for $C^\infty_c(\reals^d, \Wedge^{k-1}\reals^d)$.}
\end{equation*}
A $k$-current $T$ is called \emph{normal} if both $T$ and $\partial T$ are of finite mass.
The support of $T\in \overline{\mathcal{D}}_k(\reals^d)$, $\spt T$, is the least closed set $F$ such that $T(\omega) = 0$ for all $\omega\in C^\infty_c(\reals^d, \Wedge^k\reals^d)$ vanishing in $F$.

\paragraph{Flat chains}
Given a compact set $K\subset\reals^d$, we define the \emph{flat seminorm (with respect to $K$)} of a smooth differential form and the \emph{flat norm (with respect to $K$)} of a current by
\begin{equation*}
    \lVert \omega \rVert_{\mathbb{F}^k(K)}=\max\{\sup_{x\in K}\lVert\omega(x) \rVert, \sup_{x\in K}\lVert \diff\omega \rVert\}
\end{equation*}
and
\begin{equation*}
    \lVert T \rVert_{\mathbb{F}_k(K)}=\sup_{\lVert\omega\rVert_{\mathbb{F}^k(K)}\leq 1} |T(\omega)|,
\end{equation*}
respectively.
Note that $\lVert T \rVert_{\mathbb{F}_k(K)}<\infty$ implies $\spt T \subset K$.

We recall the equivalent way of defining the flat norm via
\begin{equation*}
    \lVert T \rVert_{\mathbb{F}_k(K)}=\inf\{\lVert S \rVert + \lVert R \rVert: T=S+\partial R; \;\spt R, \spt S\subset K\}.
\end{equation*}

By $\mathbb{F}_k(K)$, we denote the set of \emph{flat $k$-chains in $K$}, i.e.~the closure with respect to $\lVert \cdot \rVert_{\mathbb{F}_k(K)}$ of normal currents supported inside $K$.
The space $\mathbb{F}_k(K)$ is equipped with the flat norm, making it into a Banach space.
We call a $k$-current \emph{flat}, or a \emph{flat chain}, if it lies in $\mathbb{F}_k(K)$ for some compact $K\subset \reals^d$.

Next, we identify flat chains in the top dimension.
The symbol $\mathcal{L}^d$ stands for the Lebesgue measure.
Via \eqref{E:td-ext-prod}, given a $d$-form $\omega$, there exists a unique real valued map which we, using abuse of notation, denote again by $\omega$, such that
\begin{equation*}
    \omega = \omega \diff x_1 \wedge \dots \wedge\diff x_d.
\end{equation*}

\begin{lemma}\label{L:prel-identify-with-L1}
    Suppose $Q\subset \reals^d$ is compact, convex and has non-empty interior.
Then the map
    \begin{equation*}
        L^1(Q)\ni \varphi \mapsto (\omega \mapsto \int\varphi\omega\;\diff\mathcal{L}^d)\in \mathbb{F}_d(Q)
    \end{equation*}
    is an isometry of Banach spaces.
\end{lemma}

In light of the previous lemma, we may simply identify $\mathbb{F}_d(Q)=L^1(Q)$.

Given $T\in \mathbb{F}_k(K)$ and a smooth map $F\colon \reals^d \to \reals^m$, one can define the pushforward of $T$ by $F$, denoted by $F_\# T$.
The definition is given by
\begin{equation*}
    F_\#T(\omega \diff x_{i_1}\wedge\dots \wedge\diff x_{i_k})=T(\eta(\omega\circ F) \diff F_{i_1} \wedge \dots \wedge\diff F_{i_k}) \quad \text{for $\omega\in C^\infty_c(\reals^m)$,}
\end{equation*}
where $\eta$ is any smooth, compactly supported function, which is equal to $1$ on a neighbourhood of $F(\spt T)$.
By the definition of support, the particular choice of $\eta$ does not matter.
It can be verified \cite[4.1.14]{Federer} that $F_\# T \in \mathbb{F}_k(F(K))$.
In particular, we have the following invariance of flatness under isometries.

\begin{lemma}\label{L:flat-injection}
    Suppose $T$ is a $k$-current in $\reals^d$ supported inside a compact set.
Let $I\colon \reals^d \to \reals^m$ be an isometry.
If $I_\# T$ is a flat chain, then so is $T$.
\end{lemma}
\begin{proof}
    Let $F\colon \reals^m \to \reals^d$ be given by $F=I^{-1}\circ P$, where $P$ is the orthonormal projection to $I(\reals^d)$.
Then one may verify that $F_\#I_\# T = T$.
Therefore, if $I_\# T$ is a flat chain, then by the discussion above, also $T= F_\#I_\# T$ is a flat chain.
\end{proof}

\paragraph{Metric currents}
We already provided the definition of metric currents in Subsection \ref{SS:intro-currents}.
Given a locally compact metric space $X$,
the vector space of all metric $k$-currents in $X$ is denoted by $\mathcal{D}_k(X)$.

One may naturally define \emph{mass} and \emph{boundary} of a metric $k$-current \cite{UL}, but we shall not need to work with these notions explicitly.
In case $X$ is a Euclidean space, these notions correspond to the mass and boundary of classical currents.

The support of $T\in \mathcal{D}_k(X)$ is the least closed set $F\subset X$ such that $T(f, \pi^1, \dots, \pi^k)=0$ whenever $f=0$ on $F$.

If $T$ is a compactly supported metric $k$-current on $X$ and $Y$ is a locally compact metric space, $F\colon X \to Y$ is Lipschitz, then we define the \emph{pushforward} of $T$ under $F$ by
\begin{equation*}
    F_\# T(f,\pi_1,\dots, \pi_k)=T(\sigma (f\circ F), \pi_1 \circ F, \dots, \pi_k \circ F) \quad \text{for $f\in\tLip_c(Y), \pi_i \in \tLip(Y)$.}
\end{equation*}
where $\sigma$ is any Lipschitz, compactly supported map on $Y$ with $\sigma = 1$ on a neighbourhood of $F(\spt T)$ (we need to assume that $Y$ is locally compact to guarantee existence of such $\sigma$).
The pushforward of a metric current is again a metric current by definition.

\paragraph{Relationship of classical and metric currents in the Euclidean space}

Metric currents in the Euclidean space can be identified with classical currents using the so-called comparison map.
The definition we use as well as its basic properties are due to Lang \cite[Theorem 5.5]{UL} though for finite mass metric currents with compact support, this has previously been done by Ambrosio and Kirchheim \cite[Theorem 11.1]{AKcur}.
For every $k\geq 0$, there exists an injective linear map $C_k \colon \mathcal{D}_k(\reals^n) \to \overline{\mathcal{D}}_k(\reals^n)$ such that
\begin{equation*}
    C_k(T)(f \diff g_1 \wedge \dots \wedge \diff g_k)=T(f,g_1,\dots,g_k),
\end{equation*}
for all $f\in C^{\infty}_c(\reals^n)$ and $g_1,\dots, g_k \in (C^{\infty}(\reals^n)\cap\tLip(\reals^n))^k$.
The map $C_{k}$ is called the \emph{comparison map} and it preserves all of the important structure of currents, such as the boundary operation and mass.
In particular, Lang shows \cite[Theorem 5.5]{UL} that
\begin{equation*}
    \rng C_k \supset \bigcup_{K\subset \reals^d}\mathbb{F}_k(K),
\end{equation*}
where the union is taken over compact subsets of $\reals^d$.
In other words, every flat chain is in the range of $C_k$.
The following is the general statement of the Flat Chain Conjecture of Lang \cite[p. 716]{UL}.

\begin{conjecture}\label{Con:Lang-general}
    For every $d\in \nat$ and $k\in\nat\cup\{0\}$, if $T\in\mathcal{D}_k(\reals^d)$ has compact support, then
    \begin{equation*}
        C_k(T)\in \mathbb{F}_k(K) \quad \text{for some compact $K\subset\reals^d$.}
    \end{equation*}
\end{conjecture}

We identify $T$ and $C_k T$ whenever no confusion can arise.
In particular, we say that a metric $d$-current $T$ in $\reals^d$ is flat (or that it is a flat chain), if $C_k(T)$ is a flat chain.

\begin{remark}
    \begin{enumerate}
        \item There are metric currents $T\in\mathcal{D}_k(\reals^d)$ with a compact support, such that $C_k(T)\in \mathbb{F}_k(K)$ for \emph{some} compact $K\subset\reals^d$, however, $C_k(T)\not\in \mathbb{F}_k(\spt T)$.
        \item If $Q\subset \reals^d$ is compact, convex and has non-empty interior, then $T\in\mathcal{D}_k(\reals^d)$ with $\spt T\subset Q$ is a flat chain if and only if $T\in \mathbb{F}_k(Q)$.
Indeed, one implication is obvious and for the other, consider a sequence of normal currents $N_i$ approximating $C_k T$ in a compact set $K\supset Q$.
If $F$ is the nearest-point projection to $Q$, then $F_\# N_i$ is a sequence of normal currents approximating $T$ in the space $\mathbb{F}_k(Q)$.
We leave the details to the reader.
    \end{enumerate}
\end{remark}

In light of the preceding remark, we consider the special case of the conjecture when we only study $d$-currents supported in the unit cube $Q=[0,1]^d$.
In this case, the conjecture would in particular imply that the map
\begin{equation*}
    C_d^{-1}\colon \mathbb{F}_d(Q)\to \mathcal{D}_d(Q)
\end{equation*}
is surjective.
When combined with the identification of $L^1(Q)=\mathbb{F}_d(Q)$ from Lemma \ref{L:prel-identify-with-L1}, stating that this map is surjective is the same as the statement of Conjecture \ref{Con:Lang-TD} from preliminaries.

Finally, we shall need to pass between metric and classical currents when considering pushforwards, which is enabled by the following lemma, proof of which follows immediately from the definitions.

\begin{lemma}\label{L:commutativity-pw}
    Suppose $d,n\in\nat$.
    Suppose $T$ is a compactly supported metric $k$-current in $\reals^d$ and $F\colon \reals^d \to \reals^n$ is smooth and Lipschitz.
Then
    \begin{equation*}
        F_\# C_k(T)=C_k(F_\# T).
    \end{equation*}
\end{lemma}

\subsection{Properties of Lipschitz maps and their determinants}\label{ss:Lip-det}

If $f\colon \reals^d \to \reals$ is Lispchitz, then by Rademacher's theorem its Frech\'et differetial $\Diff f(x)$ exists for almost every $x\in \reals^d$.
If $E\subset \reals^d$ and $f\colon E \to \reals$ is Lipschitz, then there is a Lipschitz extension $\tilde f\colon \reals^d\to \reals$ of $f$.
Moreover, at almost every point of $E$, $\tilde f$ is Frech\'et differetible with the differential only dependending on $f$.
Thus, if we let $\Diff f(x)$ be the Frech\'et differential at $x$ of any Lipschitz extension of $f$, then $\Diff f(x)$ is well defined almost everywhere in $E$.
If $\pi\colon E \to \reals^d$ is Lipschitz, then the differential is identified with the matrix
\begin{equation*}
    \Diff \pi(x)=(\frac{\partial \pi^j}{\partial e_i}(x))_{i,j=1}^d.
\end{equation*}
Recall the basic inequality
\begin{equation*}
    |\det\Diff \pi(x)|\leq \prod_{j=1}^d \lVert\Diff\pi^j(x)\rVert,
\end{equation*}
whenever the left side is well defined, where on the right hand side the norm is the dual norm of $(\reals^d)^*$.
In particular, we have
\begin{equation}\label{E:det-linfty-est}
    \lVert \det\Diff \pi \rVert_{L^\infty(E)}\leq \prod_{j=1}^d\tLip(\pi^j).
\end{equation}

Suppose $\pi^1, \dots, \pi^k\colon E \to \reals$ and $x\in E$ is such that $\Diff \pi^j(x)$ exists for all $j\in\{1, \dots, k\}$.
Recall that $\diff \pi^j(x)= \Diff\pi^j(x)$.
By Lemma \ref{L:prel-multivector-pairing-est}, we have
\begin{equation}\label{E:est-of-product}
    |\langle \diff \pi_1 \wedge \dots \wedge \diff \pi_k, v_1\wedge \dots \wedge v_k \rangle|\leq \prod_{j=1}^k \tLip(\pi^j) \quad \text{for $v_1, \dots, v_k\in\reals^d$ of norm at most $1$.}
\end{equation}

We will convene to usually omit the differential $\diff \mathcal{L}^d$ when integrating with respect to the Lebesgue measure.
Thus, we write, for example,
\begin{equation*}
    \int_E f\det\Diff \pi=\int_E f\det\Diff \pi\;\diff\mathcal{L}^d.
\end{equation*}

The following lemma is a well known property of determinants, which may be found e.g.~in the book of Ambrosio at al \cite[Theorem 2.16]{AFP}, see also Lang \cite[Proposition 2.6]{UL}.

\begin{lemma}[Joint continuity of determinants]\label{L:j-c-for-dets}
    Suppose $d\in\nat$, $E\subset \reals^d$ is Borel, $\pi_i, \pi\colon E\to \reals^d$ are Lipschitz and $\pi_i\to \pi$ pointwise with bounded Lipschitz constants on $E$ and $f_i, f\in L^1(E)$ with $f_i \to f$ in $L^1(Q)$.
Then
    \begin{equation*}
        \int_E f_i\det \Diff \pi_i \to \int_E f\det\Diff \pi.
    \end{equation*}
\end{lemma}

Suppose $Q=[0,1]^d$ and let $Q$ have the standard orientation given by the basis vectors $e_1, \dots, e_d$.
Let $\partial Q$ stand for the boundary of $Q$.
Naturally, $\partial Q$ is a non-overlapping union of its faces, each being an oriented $(d-1)$-manifold.
We use $\sigma$ for the surface measure on $\partial Q$.
We denote by $\operatorname{Tan}_{\partial Q}$ the unit oriented tangent to $\partial Q$, whose orientation is induced by the orientation of $Q$.
Note that since $\partial Q$ is not quite a smooth manifold, the map $\operatorname{Tan}_{\partial Q}$ is only defined on the interior of each face, in particular $\sigma$-a.e.
We recall the Stokes theorem:
\begin{equation*}
    \int_Q \diff \omega=\int_{\partial Q} \omega.
\end{equation*}
In particular, if $\omega = \pi_1 \diff \pi_2 \wedge \dots \wedge \diff \pi_d$ for some smooth functions $\pi_1, \dots, \pi_d$, then $\diff \omega = \diff \pi_1 \wedge \dots \wedge \diff \pi_d$.
The above equality of Stokes reads in this case as
\begin{equation}\label{E:stokes}
    \int_Q \det\Diff\pi \;\diff \mathcal{L}^d= \int_{\partial Q} \pi_1\langle \diff \pi_2 \wedge \dots \wedge \diff \pi_d, \operatorname{Tan}_{\partial Q}\rangle \;\diff\sigma.
\end{equation}

\subsection{Functional-analytic structure of Lipschitz maps}\label{SS:fa-lip}

Given any metric space $X$ with a distinguished point $0\in X$, we use the notation
\begin{enumerate}
    \item $\tLip_0(X)=\{f\colon X \to \reals: f(0)=0\; \text{and f is Lipschitz}\}$,
    \item $\tLip_b(X)=\{f\colon X \to \reals:\text{f is Lipschitz and bounded}\}$.
\end{enumerate}

The considerations below for the space $\tLip_0(X)$ follow from the book of Weaver \cite[Chapter 3]{Weaver}. The space $\tLip_b(X)$ may be identified with a $\tLip_0(Y)$ space (over some metric space $Y$) via \cite[Corollary 2.14]{Weaver}. Thus, all of the following assertions can be obtained from the given reference.
We warn the reader that the space $\tLip_b(X)$ space is denoted by $\tLip(X)$ therein.

The former space $\tLip_0(X)$ is equipped with the Lipschitz constant as a norm and is a Banach space.
The latter space $\tLip_b(X)$ is equipped with the norm
\begin{equation*}
    \lVert f \rVert_L=\max\{\lVert f \rVert_\infty, \tLip(f)\},
\end{equation*}
and it is also a Banach space.

Both of the Banach spaces are isometric to dual Banach spaces, whose preduals we may denote by $\AEm_0(X)$ and $\AEm_b(X)$ respectively.
If $X$ is separable, the then both $\AEm_0(X)$ and $\AEm_b(X)$ are separable, and hence the
unit balls of the Lipschitz spaces, equipped with the weak$^*$ topologies are compact and metrizable.
We denote
\begin{enumerate}
    \item by $\tLip_1(X)$ the unit ball of $\tLip_0(X)$, equipped with the weak$^*$ topology,
    \item by $\tLip_{b,1}(X)$ the unit ball of $\tLip_b(X)$, equipped with the weak$^*$ topology.
\end{enumerate}

For a Banach space $B$, we write $\epsilon \colon B \to B^{**}$ for the canonical embedding defined by $\epsilon(x)(\varphi)=\varphi(x)$ for $x\in X, \varphi\in X^*$.

\begin{lemma}\label{L:w*}
    If we identify $\AEm_0(X)^*=\tLip_0(X)$ and $\AEm_b(x)^*=\tLip_b(X)$, then the span of evaluation functionals is dense in both $\epsilon(\AEm_0(X))$ and $\epsilon(\AEm_b(X))$.
    In either of the spaces $\tLip_0(X)$ or $\tLip_b(X)$ a sequence converges weak$^*$ if and only if it is bounded and it converges pointwise.
\end{lemma}

\begin{proposition}\label{P:0currents}
    If $X$ is compact,
    then $\epsilon(\AEm_b(X))= \mathcal{D}_0(X)$.
\end{proposition}
\begin{proof}
    For any Banach space $B$, we recall
    the basic fact:
    \begin{equation*}
        \epsilon(B)=\{\varphi\in B^{**}: \varphi \;\text{acts weak$^*$ continuously on $B^*$}\}.
    \end{equation*}
    Since $X$ is compact, by applying this fact to $B=\AEm_b(X)$, we conclude that $\epsilon(\AEm_b(X))$ is the set of those maps
    \begin{equation*}
        T\colon \tLip_c(X)= \tLip_b(X)\to \reals,
    \end{equation*}
    which are linear and weak$^*$ continuous.
By the Kre\v{\i}n-Shmul'yan theorem \cite[Chapter 5, 12.1]{Conway}, weak$^*$ continuity is in this case equivalent to weak$^*$ sequential continuity.
By Lemma \ref{L:w*}, $\epsilon(\AEm_b(X))$ therefore consists of those functionals on $\tLip_c(X)=\tLip_b(X)$ which are linear and continuous with respect to pointwise convergence with bounded Lipschitz constants.
By definition, this is the space $\mathcal{D}_0(X)$.
\end{proof}

\begin{corollary}\label{C:density-of-normal-in-D0}
    For any $x\in X$, the evaluation functional $\pi \mapsto \pi(x)$ on $\tLip_b(X)$ is a measure in the sense that
    \begin{equation*}
        \pi\mapsto \pi(x)=\int \pi \;\diff \delta_x \quad\text{for every $\pi\in\tLip(X)$,}
    \end{equation*}
    where $\delta_x$ is the Dirac measure at $x$.
    If $X$ is compact, then for every metric $0$-current $T$ on $X$, there exists a sequence of measures $\mu_i$ on $X$ such that
    \begin{equation*}
        \sup_{\pi\in\tLip_{b,1}(X)} |\int\pi\;\diff \mu_i-T(\pi)|\to 0.
    \end{equation*}
\end{corollary}
\begin{proof}
    The first statement is trivial.
The second statement then follows by the first part of Lemma \ref{L:w*} and Proposition \ref{P:0currents}.
\end{proof}

\section{The comparison map in the top dimension and sums of Jacobians}\label{S:char}
The goal of this section is to characterise a particular topological property of the comparison map in top dimension via the prescribed Jacobian equation.
This is the content of Theorem \ref{T:f-a-characterisation}.
First, however, some theory needs to be developed.

\subsection{Metric compactum}
For a compact space $K$, we denote by $C(K)$ the space of continuous real valued functions on $K$.
The space is equipped with the supremum norm denoted by $\lVert \cdot \rVert_{C(K)}$ making it into a Banach space.
We recall that the space $C(K)$ is separable if (and only if) the space $K$ is metrisable.

We shall identify the space of metric currents on a compact metric space $X$ with a subspace of $C(\mathcal{K})$ for a suitable compact (metrisable) space $\mathcal{K}$.
This will allow us to define a ``natural'' norm on the space of currents via inheritance.

\begin{definition}
    Given a metric space $X$ and $k\in\nat$, the \emph{metric $k$-compactum} is the topological space
    \begin{equation*}
        \mathcal{K}=\mathcal{K}(X)=\tLip_{b,1}(X)\times [\tLip_1(X)]^k,
    \end{equation*}
    equipped with the product topology.
\end{definition}
Let $X$ be compact.
Then
it follows from the considerations in Subsection \ref{ss:Lip-det} that the space $\mathcal{K}(X)$ is compact and metrisable.
Recall that $\mathcal{D}_k(X)$ stands for the space of metric $k$-currents on $X$.
Since $\tLip_b(X)=\tLip_c(X)$, any element $T$ of $\mathcal{D}_k(X)$ restricts, by definition, to a map $T\colon \mathcal{K}(X)\to \reals$.
We shall identify $T=T_{|\mathcal{K}(X)}$.

\begin{theorem}
    Let $X$ be a compact metric space and $k\in\nat$.
Then $\mathcal{D}_k(X)\subset C(\mathcal{K}(X))$.
\end{theorem}
\begin{proof}
    This follows from the joint continuity property.
\end{proof}

\begin{corollary}\label{C:norm-is-finite}
    Suppose $X$ is compact and $k\in\nat$.
    For any current $T\in \mathcal{D}_k(X)$,
    \begin{equation*}
        |T(f,\pi^1, \dots, \pi^k)|\leq \lVert T \rVert_{C(\mathcal{K})} \lVert f \rVert_L \prod_{j=1}^k \tLip(\pi^j),
    \end{equation*}
    for every $(f, \pi^1, \dots, \pi^1)\in \tLip_b(X)\times[\tLip(X)]^k$.
\end{corollary}
\begin{proof}
    By locality, we may subtract constant maps from each of the last $d$ variables of $T$.
Thus, it is sufficient to prove the inequality assuming $\pi_1, \dots,\pi_d \in \tLip_0(X)$.
    By homogeneity in each variable (on both sides of the inequality), it is sufficient to prove the statement under the additional assumption that $(f,\pi^1, \dots, \pi^k)\in \mathcal{K}(X)$.
Then the statement holds by definition of $\lVert T \rVert_{C(\mathcal{K})}$.
\end{proof}

\subsection{Potential forms and the PDE}\label{SS:potential-forms}
In this subsection we develop some basic theory of what we call potential forms, which make up a linear space that is in natural duality with the space of metric currents (cf.~the work of Pankka and Soultanis \cite{PS}, where a different notion of duality is explored).

We fix $d\in \nat$ and denote $Q=[0,1]^d\subset \reals^d$.
This is the exclusive underlying compact pointed metric space with which we shall be concerned.
The distinguished point is $0$.
We shall also only be discussing $d$-currents and therefore we fix (consistently with the previous subsection) the metric $d$-compactum
\begin{equation*}
    \mathcal{K}= \tLip_{b,1}(Q)\times [\tLip_1(Q)]^d.
\end{equation*}
When working with spaces over $Q$, such as $L^1(Q)$, $L^\infty(Q)$ various Lipschitz spaces, or spaces of \emph{forms} and \emph{currents} which we define later, we will often omit writing out the symbol $Q$ and write only $L^1$, $L^\infty$ etc.

\begin{definition}\label{D:fls}
For $d\in\nat$, any expression of the form
\begin{equation*}
    \sum_i(f_i,\pi_i),
\end{equation*}
where
\begin{enumerate}
    \item the summation is implicitly assumed to be over $i\in\nat$,
    \item for each $i$, $f_i\in\tLip_b(Q)$ and $\pi_i\in [\tLip(Q)]^d$ and
    \item \label{Enum:s-est}
    \begin{equation*}
        s(\sum_i(f_i,\pi_i))=\sum_i \lVert f_i \rVert_L \prod_{j=1}^d\tLip(\pi_i^j)<\infty.
    \end{equation*}
\end{enumerate}
is called a \emph{formal Lipschitz sum}.
We denote by $\mathcal{A}_d$ the \emph{set} of formal Lipschitz sums.
\end{definition}

Expressions such as $\sum_{i=1}^n(f_i,\pi_i)$ for $n\in\nat$ are identified with the formal Lipschitz sum $\sum_i(f_i,\pi_i)$ where $f_i=0$ and $\pi_i=0$ for $i>n$.

Next, we define an equivalence relation $\sim$ on $\mathcal{A}_d$ and then a vector space structure on $\mathcal{A}_d\slash\sim$.
The equivalence relation is given by $\sum_i(f_i,\pi_i)\sim\sum_i(g_i,\kappa_i)$ if
\begin{equation*}
    \sum_i T(f_i ,\pi_i) = \sum_i T(g_i,\kappa_i) \quad \text{for every $T\in\mathcal{D}_d$.}
\end{equation*}
The definition if valid due to \ref{Enum:s-est} and Corollary \ref{C:norm-is-finite}.
Indeed, the sums appearing in the definition are even guaranteed to be converging absolutely.

Given a countable set $S$, any family $\{(f_s, \pi_s)\}_{s\in S}$ such that
\begin{equation*}
    \sum_{s\in S} \lVert f_s \rVert_L \prod_{j=1}^d \tLip(\pi_s^d) <\infty,
\end{equation*}
and any two bijections $\rho_1, \rho_2\colon \nat \to S$,
\begin{equation*}
    \sum_{i}(f_{\rho_1(i)}, \pi_{\rho_1(i)})\sim \sum_{i}(f_{\rho_2(i)}, \pi_{\rho_2(i)}).
\end{equation*}
Thus, we define the common value to be $[\sum_s(f_s,\pi_s)]$.

Using this notation, we may define addition on $\mathcal{A}_d\slash \sim$ via
\begin{equation*}
    [\sum_i(f_i, \pi_i)] + [\sum_i(g_i, \kappa_i)]=[\sum_{i,j\in\nat}(f_i,\pi_i)+(g_j,\kappa_j)],
\end{equation*}
We also define scalar multiplication
\begin{equation*}
    \lambda [\sum_i(f_i,\pi_i)]=[\sum_i(\lambda f_i, \pi_i)].
\end{equation*}
It is immediately verified that under these operations, $\mathcal{A}_d\slash \sim$ is a vector space.

Observe the inequality (cf.~\eqref{E:det-linfty-est})
\begin{equation*}
        \lVert f \det \Diff \pi \rVert_{L^\infty(Q)}\leq \lVert f \rVert_L \prod_{j=1}^d \tLip(\pi^j),
\end{equation*}
which guarantees, that whenever $\sum_i (f_i,\pi_i)$ is a formal Lipschitz sum, the sum $\sum_i f_i \det\Diff \pi_i$ converges absolutely in $L^\infty(Q)$.

\begin{definition}
    Two formal Lipschitz sums $\sum_i(f_i, \pi_i)$ and $\sum_i(f_i,\pi_i)$ are called \emph{weakly equivalent} if
    \begin{equation*}
        \sum_i f_i \det\Diff \pi_i = \sum_i g_i \det\Diff \kappa_i \quad \text{a.e.~in $Q$,}
    \end{equation*}
    where the convergence is absolute in $L^\infty(Q)$.
\end{definition}

It is very natural to expect that two formal Lipschitz sums are equivalent if and only if they are weakly equivalent.
We are, however, not able to assert the validity of this statement.

Recall the convention that we usually omit writing out the differential $\diff \mathcal{L}^d$ when integrating with respect to the Lebesgue measure.

\begin{lemma}\label{L:equivalences}
    If two formal Lipschitz sums are equivalent, then they are also weakly equivalent.
\end{lemma}
\begin{proof}
    For a function $u\in L^1$, we denote by $T_u$ the metric $d$-current
    \begin{equation*}
        T_u(f, \pi^1, \dots, \pi^d)=\int_Q u f \det\Diff\pi \quad \text{for $(f, \pi^1, \dots, \pi^1)\in \tLip_b\times[\tLip]^d$}.
    \end{equation*}
    Suppose now $\sum_i(f_i, \pi_i)$ and $\sum_i(g_i, \kappa_i)$ are formal Lipschitz sums.
Then we have
    \begin{equation*}
        \sum_i T_u (f_i, \pi_i) =\sum_i \int_Q u f_i \det\Diff\pi_i=\int_Q u \sum_i f_i\det\Diff\pi_i,
    \end{equation*}
    by Lebesgue dominated convergence theorem.
Similar identity holds for $\sum_i(g_i, \kappa_i)$.
Therefore, since every $T_u$ is a metric $d$-current, if $\sum_i(f_i, \pi_i)$ and $\sum_i(g_i, \kappa_i)$ are equivalent, then in particular
    \begin{equation*}
        \int_Q u \sum_i f_i\det\Diff\pi_i = \int_Q u \sum_i g_i\det\Diff\kappa_i \quad \text{for every $u\in L^1(Q)$.}
    \end{equation*}
    Thus, the sums are weakly equivalent.
\end{proof}

\begin{definition}\label{D:pot-forms}
    We define the \emph{potential norm} on $\mathcal{A}_d\slash \sim$ via
    \begin{equation*}
    \lVert \omega \rVert_{\mathbb{P}^d}=\lVert \omega \rVert_{\mathbb{P}^d(Q)}=\inf\{s(\sum_i(f_i,\pi_i)): \sum_i(f_i,\pi_i)\in\omega\}\quad \text{for $\omega \in \mathcal{A}_d\slash \sim$.}
    \end{equation*}
    The set $\mathbb{P}^d(Q)=\mathbb{P}^d=\mathcal{A}_d\slash \sim$, equipped with the potential norm, is called \emph{the space of (top dimensional) potential (differential) forms}\footnote{If one is familiar with the basic theory of tensor products, it can be easily verified, that the space $\mathbb{P}^d$ may alternatively obtained from the tensor product, say $\mathcal{T}= \tLip_b(Q)\otimes \bigotimes_{j=1}^d \tLip(Q)$.
This is due to the fact that the equivalence relation we define has strictly larger equivalence classes than the equivalence relation used to define the tensor product.
Thus $\mathbb{P}^d= \overline{\mathcal{T}} \slash \sim$}.
Its elements are called \emph{(top dimensional) potential (differential) forms}.
The map
    \begin{equation*}
        \Em \colon \mathbb{P}^d(Q)\to L^\infty(Q),
    \end{equation*}
    given by $\Em([\sum_i(f_i,\pi_i)])= \sum_i f_i \det\Diff\pi_i$ is called \emph{the canonical embedding (for potential forms)}.
\end{definition}

The definition of $\Em$ is valid due to Lemma \ref{L:equivalences} which asserts that the choice of the representative is irrelevant.
Since we do not know whether weak equivalence is equivalent to equivalence, we are not able to assert that $\Em$ is injective.
We shall nevertheless call it an ``embedding''.

From now on, we will use the shorthand notation
\begin{equation*}
    T(f, \pi) = T(f, \pi^1, \dots, \pi^d) \quad \text{for $(f, \pi^1, \dots, \pi^d) \in \tLip_b\times [\tLip]^d$.}
\end{equation*}

\begin{definition}\label{D:duality}
    Given a current $T\in \mathcal{D}_d(Q)$ and a potential form $\omega \in \mathbb{P}^d(Q)$, we define their \emph{duality pairing} as
    \begin{equation*}
        \langle \omega, T \rangle = \sum_i T(f_i, \pi_i),
    \end{equation*}
    where $\sum_i (f_i, \pi_i)\in\omega$.
\end{definition}
 Once again, by definition of $\sim$, the choice of the representative is irrelevant.

\begin{theorem}
    The space $\mathbb{P}^d$ is complete.
\end{theorem}
\begin{proof}
    We use the Riesz completeness criterion.
Let $\sum_j \omega_j$ be an absolutely convergent sum in $\mathbb{P}^d$.
Then, by definition of the potential norm, we may find for each $j\in\nat$ a representative
\begin{equation}\label{E:repre}
    \sum_i(f_{i,j}, \pi_{i,j})\in \omega_j
\end{equation}
such that
\begin{equation}\label{E:doubly-sum-abs}
    \sum_j \sum_i \lVert f_{i,j} \rVert_L \prod_{k=1}^d \tLip(\pi_{i,j}^k)<\infty.
\end{equation}
We claim that
\begin{equation}\label{E:double-single}
    [\sum_j\sum_i(f_{i,j}, \pi_{i,j})] = \sum_j \omega_j,
\end{equation}
which would imply that $\sum_j \omega_j$ converges in $(\mathbb{P}_d)^*$ and conclude the proof.
By \eqref{E:doubly-sum-abs}, we see that the left hand side of \eqref{E:double-single} is well defined and, by the definition of $\sim$, it equals the right hand side if and only if
\begin{equation*}
    \sum_j \sum_i T(f_{i,j}, \pi_{i,j}) = \sum_j \langle \omega_j, T\rangle \quad \text{for every $T\in\mathbb{P}_d$.}
\end{equation*}
However, this equality follows immediately from the definition of the pairing, \eqref{E:repre} and \eqref{E:doubly-sum-abs}.
\end{proof}

Here we connect potential forms and the canonical embedding to Lipschitz regularity.

\begin{theorem}
    The following statements are equivalent.
    \begin{enumerate}
        \item The canonical embedding $\Em\colon \mathbb{P}^d \to L^\infty$ is non-surjective.
        \item The prescribed Jacobian equation has linearised Lipschitz non-regularity.
    \end{enumerate}
\end{theorem}
\begin{proof}
    This is an immediate consequence of the Definitions \ref{D:pot-forms} and \ref{D:Lip-reg}.
\end{proof}

 \begin{definition}
     For a current $T\in \mathcal{D}_d(Q)$, we define its \emph{potential norm} by
     \begin{equation*}
         \lVert T \rVert_{\mathbb{P}_d} = \sup_{\lVert \omega \rVert_{\mathbb{P}^d}\leq 1} |\langle \omega, T \rangle|.
     \end{equation*}
     The space $\mathcal{D}_d(Q)$ equipped with the potential norm is called the space of potential $d$-currents and denoted by $\mathbb{P}_d=\mathbb{P}_d(Q)$.
 \end{definition}

 \begin{theorem}\label{T:norm-equal}
    For every metric $d$-current $T\in\mathbb{P}_d$, we have
    \begin{equation*}
        \lVert T \rVert_{\mathbb{P}_d} = \lVert T \rVert_{C(\mathcal{K})}.
    \end{equation*}
    In particular, the former quantity is finite.
\end{theorem}
\begin{proof}
    The inequality
    \begin{equation*}
        \lVert T \rVert_{\mathbb{P}_d}\geq \lVert T \rVert_{C(\mathcal{K})}
    \end{equation*}
    is immediate, as the supremum on the left hand side is taken over a larger set.
We show the converse inequality.
Utilizing Corollary \ref{C:norm-is-finite}, we obtain
    \begin{equation*}
        \begin{split}
            \lVert T \rVert_{\mathbb{P}_d} &= \sup_{\lVert \omega \rVert_{\mathbb{P}^d}\leq 1} |T(\omega)|
            =\sup\{\sum_i T(f_i,\pi_i): s(\sum_i (f_i, \pi_i))\leq 1\}\\
            &=\sup\{\sum_i T(f_i,\pi_i): \sum_i\lVert f_i \rVert_L \prod_{j=1}^d\tLip(\pi^j_i)\leq 1\}\\
            &{\leq}\sup\{\sum_i \lVert T \rVert_{C(\mathcal{K})}\lVert f_i \rVert_L \prod_{j=1}^d\tLip(\pi^j_i): \sum_i\lVert f_i \rVert_L \prod_{j=1}^d\tLip(\pi^j_i)\leq 1 \}\\
            &\leq \lVert T \rVert_{C(\mathcal{K})}.
        \end{split}
    \end{equation*}
\end{proof}

 \begin{remark}
     A modified version of this construction can be carried out for a set $Q\subset \reals^d$ which is, say, open.
In this case, the space of potential currents (those for which the potential norm is finite) would consist of those metric currents, that have only a controlled ``explosion'' at the boundary of $Q$ and therefore, in general, the space of potential currents cannot be expected to coincide with all metric currents.
This justifies the naming convention.
 \end{remark}

\begin{theorem}
    The space $\mathbb{P}_d$ is an isometric closed subspace of $C(\mathcal{K})$, therefore it is a separable Banach space.
\end{theorem}
\begin{proof}
    The fact that $\mathbb{P}_d$ is an isometric subspace of $C(\mathcal{K})$ follows from Theorem \ref{T:norm-equal}.
We show that it is a closed subspace.
To that end, let $T_i \in \mathbb{P}_d$ be a sequence converging to $T\in C(\mathcal{K})$ with respect to the norm $\lVert \cdot \rVert_{C(\mathcal{K})}$.
Note that since each $T_i$ is actually a functional on all of $\tLip_b\times [\tLip]^d$, the map $T$ has a natural extension to $\tLip_b\times [\tLip]^d$ given, for example, by
    \begin{equation*}
        T(f, \pi) = \lim_i T_i(f, \pi).
    \end{equation*}
    Since pointwise limits of linear functionals are linear, it follows that $T$ is multilinear.
Since $T$ lies in the space $C(\mathcal{K})$, it is jointly continuous.
It remains to show that $T$ is local.
This is immediate since whenever $f\in\tLip_b$ and $\pi^j\in \tLip$ are such that one of ${\pi^j}$'s is constant on a neighbourhood of $\spt f$, then $T_i(f, \pi^1, \dots, \pi^d)=0$ for every $i$.
Hence also $T(f, \pi^1 ,\dots, \pi^d)=0$.
\end{proof}

\begin{remark}\label{R:duality}
    Several interesting open questions concerning the functional-analytic structure of the introduced spaces arise.
Firstly, by Theorem \ref{T:norm-equal} we see that we have a canonical embedding $\mathbb{P}_d\to (\mathbb{P}^d)^*$, which is isometric.
On the other hand,
    we have the canonical embedding $I\colon \mathbb{P}^d \to(\mathbb{P}_d)^*$,
    given by
    \begin{equation}\label{E:I}
        I(\omega)(T)=\langle \omega, T \rangle.
    \end{equation}
    Is $I$ onto?
Is $I$ an isomorphism of Banach spaces?
If $I$ is onto, then (since it is injective), it is also a Banach space isomorphism, which equates to an identification of the dual space $(\mathbb{P}_d)^*$ (up to constants).
    This would be a result akin to Wolffe's theorem on duality of flat forms \cite[Section 5]{Heinonen}.
\end{remark}

 In the following, we provide a simple characterisation of the surjectivity of $I$ via a compactness statement.
This result is not needed or used further in the present paper, but it is of independent interest.

\begin{proposition}\label{P:duality-of-potential}
The following two statements are equivalent.
\begin{enumerate}
    \item\label{Enum:dual1} For any bounded sequence $\omega_j\in \mathbb{P}^d$ a subsequence $\omega_{j_k}$ and a potential form $\omega \in \mathbb{P}^d$ may be found in such a way that for every $T\in \mathbb{P}_d$,
    \begin{equation*}
        \langle \omega_{j_k}, T \rangle \to \langle \omega, T \rangle.
    \end{equation*}
    \item\label{Enum:dual2} The operator $I$ is surjective and a Banach space isomorphism.
\end{enumerate}
\end{proposition}
\begin{proof}
    For $\lambda\geq 0$, we denote by $\lambda B_{\mathbb{P}^d}$ the ball in $\mathbb{P}^d$ around $0$ of radius $\lambda$.
    Since the linear space $I(\mathbb{P}^d)$ separates points of $\mathbb{P}_d$ (by definition), it is weak$^*$ dense in $(\mathbb{P}_d)^*$ by the bipolar theorem \cite[Chapter 5, 1.8]{Conway}.
Thus, $I$ is surjective if and only $I(\mathbb{P}^d)$ is weak$^*$ closed.
By the Kre\v{\i}n-Shmul'yan theorem \cite[Chapter 5, 12.1]{Conway}, this is the case if and only if $I(\mathbb{P}^d) \cap \lambda B_{\mathbb{P}^d}$ is weak$^*$ closed for any $\lambda\geq 0$.
Since $\mathbb{P}_d$ is separable, the weak$^*$ topology restricted to any ball is metrisable and hence by the Banach-Alaoglu theorem, the above is equivalent to requiring that $I(\mathbb{P}^d) \cap \lambda B_{\mathbb{P}^d}$ be sequentially weak$^*$ compact.

    We now show the implication \ref{Enum:dual1} $\implies$ $\ref{Enum:dual2}$.
Since $I$ is bounded, the sequence $I(\omega_j)$ is bounded in $(\mathbb{P}_d)^*$, say it lies in $\lambda B_{(\mathbb{P}_d)^*}$, therefore it has a weak$^*$ convergent subsequence $I(\omega_{j_k})\to \Phi \in \lambda B_{(\mathbb{P}_d)^*}$.
Using \ref{Enum:dual1}, we may extract a further subsequence and arrive at $\Phi \in I(\mathbb{P}_d)$.
We have shown that $I(\mathbb{P}^d) \cap \lambda B_{\mathbb{P}^d}$ is sequentially weak$^*$ compact and therefore \ref{Enum:dual2} follows by the paragraph above.

    The implication \ref{Enum:dual2} $\implies$ $\ref{Enum:dual1}$ is an immediate consequence of the open mapping theorem and Banach-Alaoglu.
\end{proof}

We now continue developing the theory of potential forms and currents with the goal of connecting the Flat Chain Conjecture and regularity of the prescribed Jacobian equation.

\begin{definition}\label{D:Em-star}
    The map $\Em_* \colon L^1(Q) \to \mathbb{P}_d(Q)$ is given by the formula
    \begin{equation*}
        \Em_*(u)(f,\pi_1, \dots, \pi_d)=\int u f \det\Diff \pi \;\diff\mathcal{L}^d\quad \text{for $(f, \pi^1, \dots, \pi^1)\in \tLip_b(Q)\times[\tLip(Q)]^d$,}
    \end{equation*}
    and is called \emph{the canonical embedding (for potential currents)}.
\end{definition}

Recall Lemma \ref{L:prel-identify-with-L1}, stating that the space of top-dimensional flat chains supported inside $Q=[0,1]^d$ coincides with $L^1(Q)$ in the sense where we identify an $L^1(Q)$ map $\varphi$ with a top dimensional flat chain $T$ via the formula
\begin{equation*}
    T(\omega \diff x_1 \wedge \dots \wedge \diff x_d)= \int_Q \varphi \omega\;\diff\mathcal{L}^d.
\end{equation*}
If this identification is denoted by, say, $J\colon L^1(Q) \to \mathbb{F}_d(Q)$, then $\Em_*=C_d^{-1}\circ J$, where $C_d^{-1}\colon \mathbb{F}_d(Q) \to \mathbb{P}_d(Q)$ is the inverse to the comparison map.

We also observe that $\Em=(\Em_*)^*\circ I$, where $I$ is given by \eqref{E:I}, which follows by unpacking the definitions (namely Definitions \ref{D:pot-forms}, \ref{D:duality}, \ref{D:Em-star}).
Here we recommend consulting the diagram in Figure \ref{fig:duality-diagram}.

\begin{figure}[ht]
    \centering
    \includegraphics[scale=1.1]{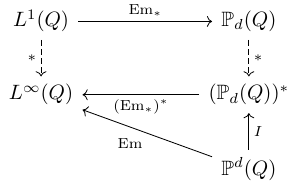}
    \caption{A diagram indicating the relationship between function spaces and potential spaces.
The dashed arrows only indicate direction of duality.
The full arrows correspond to operators.
The bottom part of the diagram commutes.}
    \label{fig:duality-diagram}
\end{figure}

The following theorem trivially implies Theorem \ref{T:intro-char}.

\begin{theorem}\label{T:f-a-characterisation}
The prescribed Jacobian equation has strong Lipschitz non-regularity, if and only if the range of $\Em_*$ is not closed.
In particular, if the prescribed Jacobian equation has strong Lipschitz non-regularity, then Conjecture \ref{Con:Lang-TD} fails.
\end{theorem}
\begin{proof}
    By the open mapping theorem, the operator $\Em_*$ has closed range if and only if it is an isomorphism into $\mathbb{P}_d(Q)$.
    This is the case, if and only if there exists a constant $C$ such that, whenever $u_n \in L^1(Q)$ satisfies \ref{Enum:strong-non-reg1} from Definition \ref{D:Lip-reg}, then
    \begin{equation}\label{E:est-from-below}
        \lVert \Em_*(u_n)\rVert_{\mathbb{P}_d}\geq C.
    \end{equation}
    Indeed, $C$ corresponds to the operator norm of the inverse of $\Em_*$.
We calculate
    \begin{equation*}
    \begin{split}
         \lVert \Em_*(u_n)\rVert_{\mathbb{P}_d}
        &=\sup\left\{\left|\Em_*(u_n)(f,\pi^1, \dots, \pi^d)\right|: \lVert f \rVert_L \prod_{j=1}^d\tLip(\pi^j)\leq 1 \right\}\\
        &=\sup\left\{\left|\int_Q u_n f\det\Diff\pi\;\diff \mathcal{L}^d\right|: \lVert f \rVert_L \prod_{j=1}^d\tLip(\pi^j)\leq 1 \right\}.
    \end{split}
    \end{equation*}
    Therefore, we see that if a constant $C$ satisfying \eqref{E:est-from-below} exists, then the sequence $u_n$ does not satisfy \ref{Enum:strong-non-reg2} of Definition \ref{D:Lip-reg}.
On the other hand, if a sequence $u_n$ can be found such that both \ref{Enum:strong-non-reg1} and \ref{Enum:strong-non-reg2} of Definition \ref{D:Lip-reg} hold, no such constant $C$ can exist.

    For the final statement, we observe that if Conjecture \ref{Con:Lang-TD} holds, then $\Em_*$ is surjective and therefore has closed range.
\end{proof}

\section{A Hahn-Banach separation theorem}\label{S:HB}

The point of this section is to develop a tool from functional analysis, which will allow us to pass from a particularly strong result on the operator $\Em$ (into $L^\infty$) to a result for the operator $\Em_*$ (on $L^1(Q)$), asserting eventually strong Lipschitz non-regularity for the prescribed Jacobian equation.

We use the following notation.
For a normed linear space $X$ and $\lambda\geq 0$, we denote by $\lambda B_X$ the ball in $X$ around $0$ of radius $\lambda$.
We also use the convention $B_X=1 B_X$.
The weak$^*$ topology on a dual Banach space $Y^*$ is denoted by $w^*$ and the symbol $(Y^*, w^*)$ stands for the locally convex topological space $Y^*$ equipped with the weak$^*$ topology induced by the predual $Y$.
The predual $Y$ is always assumed to be a Banach space.
We will identify $Y$ with $\epsilon (Y)\subset Y^{**}$, where $\epsilon$ is the canonical embedding.
That is, any element $y$ of $Y$ is identified with the functional acting on $Y^*$ via
\begin{equation*}
    y=(Y^*\ni \rho \mapsto \rho(y) \in \reals).
\end{equation*}

For the proof of the following version of a Hahn-Banach separation theorem, see the book of Conway \cite[Chapter IV, 3.7 Theorem]{Conway}.

\begin{lemma}\label{L:H-B}
    Suppose $X$ is a topological vector space and $A, B \subset X$ are disjoint convex sets.
If $A$ is open, then there exists $f\in X^*$ such that
    \begin{equation*}
        f(x)< \inf_B f \quad \text{for each $x\in A$.}
    \end{equation*}
\end{lemma}

Firstly, we apply the lemma in case $K$ is an  absolutely convex subset of a dual Banach space.
Recall that $K$ is absolutely convex if $K$ is convex and whenever $x\in K$, then also $-x\in K$.
Given a Banach space $Y$, $u\in Y$ and $\rho \in Y^*$, we use the interchangeable notation
\begin{equation*}
    \langle \rho ,u\rangle = u(\rho) = \rho(u).
\end{equation*}

\begin{proposition}\label{P:sep-predual}
    Suppose $Y$ is a Banach space and $K\subset Y^*$ is absolutely convex.
Suppose $A\subset Y^*$ is weak$^*$ open, convex, $\rho \in A$ an $A\cap K = \emptyset$.
Then there exists $u\in Y$ such that
    \begin{equation*}
        u(\rho)> \sup_{k\in K} |u(k)|.
    \end{equation*}
\end{proposition}
\begin{proof}
    The space $(Y^*, w^*)$ is a topological vector space such that $(Y^*, w^*)^*=Y$.
We find $f\in Y$ from Lemma \ref{L:H-B} achieving
    \begin{equation*}
        f(\eta)< \inf_{k\in K} f(k) \quad \text{for every $\eta\in A$.}
    \end{equation*}
    In particular,
    \begin{equation*}
         f(\rho)< \inf_{k\in K} f(k).
    \end{equation*}
    By multiplying both sides with $-1$, we obtain
    \begin{equation*}
        -f(\rho)>-\inf_{k\in K}f(k)=\sup_{k\in K} -f(k)= \sup_{k\in K} f(k),
    \end{equation*}
    where the last equality follows from the fact that $K$ is absolutely convex.
The last equality also implies that
    \begin{equation*}
        \sup_{k\in K} f(k) = \sup_{k\in K} |f(k)|,
    \end{equation*}
    which in combination with the above yields for $u=-f$
    \begin{equation*}
        u(\rho)> \sup_{k\in K}|-u(k)|=\sup_{k\in K}|u(k)|.
    \end{equation*}
\end{proof}

\begin{definition}
    Suppose $X$ and $Y$ are Banach spaces and $T\colon X \to Y^*$ is a bounded linear operator.
We say that $T$ is \emph{strongly non-surjective} if for every $n\in \nat$, there exists $\rho_n \in \frac{1}{n}B_{Y^*}\cap T(X)$ and a weak$^*$ open neighbourhood $U_n$ of $\rho_n$ such that
    \begin{equation*}
        U\cap T(B_X)=\emptyset.
    \end{equation*}
\end{definition}

It follows from the open mapping theorem, that whenever $T$ is surjective, it cannot be strongly non-surjective, justifying the naming convention.
In general, the property of strong non-surjectivity is interesting in cases where $T(X)$ is weak$^*$ dense in $Y^*$ in which case it is strictly stronger than non-surjectivity.
Indeed, the identity operator
\begin{equation}\label{E:cont-to-linfty}
    \Id:C([0,1])\to L^\infty([0,1]),
\end{equation}
is non-surjective, has weak$^*$ dense range, but it is not strongly non-surjective, because $B_{C([0,1])}$ is weak$^*$ dense inside $B_{L^\infty([0,1])}$.
One can verify the non-trivial implication between non-surjectivity and strong non-surjectivity if, additionally, $X$ is a dual Banach space and $T$ is weak$^*$ to weak$^*$ continuous.
In more precise terms, suppose $X=Z^*$ for some Banach space $Z$ and the operator $T$ is weak$^*$ to weak$^*$ continuous and onto a weak$^*$ dense set.
Then it can be shown that if $T$ is non-surjective, then it is strongly non-surjective.
The considerations needed for this assertion are standard, although they require some fairly advanced machinery, such as the Kre\v{\i}n-Shmul'yan theorem.
Since this is not something we require in the sequel, we leave the details to the reader.

From these considerations, we see that the notion of strong non-surjectivity is useful, whenever we are not able to verify that the space $X$ is a dual Banach space, which is the situation to which we wish to apply the developed theory.
Indeed, recall the diagram in Figure \ref{fig:duality-diagram}.
The desired application for the theory developed in this section is for $X=\mathbb{P}^d$, $Y^*=L^\infty$ and $T=\Em$.
However, since we are not able to assert that $\mathbb{P}^d$ is a dual Banach space (cf.~Remark \ref{R:duality} and Proposition \ref{P:duality-of-potential}), strong non-surjectivity does not necessarily follow from non-surjectivity.
On the other hand, we can use the following theorem to obtain strong Lipschitz non-regularity from strong non-surjectivity of $\Em$ (see Corollary \ref{C:non-surj0} for the explicit statement).

\begin{theorem}[Predual separation theorem]\label{T:quant-separation}
    Let $X$ and $Y$ be Banach spaces and $T\colon X \to Y^*$ a bounded linear operator, which is additionally strongly non-surjective.
Then there exists a sequence $v_n \in Y$ such that
    \begin{enumerate}
        \item\label{Enum:sep1} $\lVert v_n \rVert_Y\geq 1$,
        \item\label{Enum:sep2} $\sup_{x\in B_X} |\langle T(x), v_n\rangle|\leq \frac{1}{n}$.
    \end{enumerate}
\end{theorem}
\begin{proof}
    Let $K=T(B_X)$.
Since $B_X$ is absolutely convex and $T$ is linear, $K$ is also absolutely convex.
By definition of the weak$^*$ topology (the fact that it is locally convex suffices) there exists a convex open set $V_n\subset U_n$ with $\rho\in V_n$.
Let us take $A=V_n$ and apply Proposition \ref{P:sep-predual} obtaining thereby $u_n \in Y$ such that
    \begin{equation}\label{E:sep-func}
        u_n(\rho_n)> \sup_{x\in B_X} |u_n(T(x))|.
    \end{equation}
    Since the right hand side is larger than or equal to $0$, we in particular have $u_n(\rho_n)>0$.
Let
    \begin{equation*}
        v_n=\frac{1}{n u_n(\rho_n)}u_n \in Y.
    \end{equation*}
    Then we have
    \begin{equation*}
        \lVert v_n \rVert_Y \geq |v_n(n\rho_n)|=\frac{1}{n u_n(\rho_n)}u_n(n\rho_n)=1,
    \end{equation*}
    by definition of $v_n$, using also the fact that $\lVert n\rho_n \rVert_{Y^*}\leq 1$.
This asserts item \ref{Enum:sep1}.
For item \ref{Enum:sep2}, we have
    \begin{equation*}
        \sup_{x\in B_X} |\langle T(x), v_n\rangle|=\sup_{x\in B_X} |\langle T(x), \frac{1}{n u_n(\rho_n)}u_n\rangle|
        < \frac{1}{n u_n(\rho_n)}u_n(\rho_n)=\frac{1}{n},
    \end{equation*}
    where in the last inequality we used \eqref{E:sep-func}.
\end{proof}

\begin{remark}
    If $X$ is a normed linear space, $Y$ is a Banach space and $T\colon X \to Y^*$ a bounded linear operator which has weak$^*$ range, then it can be shown that if a sequence $v_n$ as in Theorem \ref{T:quant-separation} exists, then $T$ is strongly non-surjective.
For each $n\in \nat$, $n\geq 3$, one may take
    \begin{equation*}
        U_n = v_n^{-1}(\tfrac{2}{n}, \infty)
    \end{equation*}
    and verify, by \ref{Enum:sep2}, that $U_n\cap T(B_X)=\emptyset$ and, by \ref{Enum:sep1} and weak$^*$ density of the range of $T$ that there is an element $\rho_n \in U_n\cap T(X)\cap \frac{3}{n} B_{Y^*}$.
Since the factor $3$ is irrelevant, strong non-surjectivity is thus proven.

    This allows us to provide an example of an isometric operator with weak$^*$ dense range which is strongly non-surjective (cf.~the previous example \eqref{E:cont-to-linfty}).
Let $Y$ be any Banach space with $\dim Y^{**}\slash Y = \infty$ and let $M\subset Y^*$ be the norm closed subspace constructed by Davis and Lindenstrauss \cite{WL}, which is \emph{total} and \emph{nonnorming} (using terminology of the cited paper).
Then we have the operator
    \begin{equation*}
        \Id \colon M \to Y^*,
    \end{equation*}
    which is an isometry, since we equip $M$ with the inherited norm.
By definition, the fact that $M$ is total means $\Id$ has weak$^*$ dense range.
By definition, the fact that $M$ is nonnorming, means that a sequence $v_n$ as in Theorem \ref{T:quant-separation} can be found.
Thus, by our previous argument, $\Id$ is strongly non-surjective.
\end{remark}

The following corollary justifies the developed theory of potential forms.
The relevant spaces are once again considered over $Q=[0,1]^d$ and this dependence is suppressed in the notation. 

\begin{corollary}\label{C:non-surj0}
    Suppose the operator $\Em\colon \mathbb{P}^d \to L^\infty$ is strongly non-surjective.
Then the prescribed Jacobian equation has strong Lipschitz non-regularity.
\end{corollary}
\begin{proof}
    Applying the Predual separation Theorem \ref{T:quant-separation} to $T=\Em$, $X=\mathbb{P}^d$ and $Y=L^1$, we may obtain a sequence of maps $v_n\in L^1$ such that
    \begin{enumerate}
        \item $\lVert v_n \rVert_{L^1}\geq 1$,
        \item $\sup_{\omega\in B_{\mathbb{P}^d}}|\langle \Em(\omega), v_n \rangle|\leq \frac{1}{n}$.
    \end{enumerate}
    Using the second item, we have
    \begin{equation*}
    \begin{split}
        &\sup\left\{\left|\int_Q u_n f\det\Diff\pi\diff \mathcal{L}^d\right|: \lVert f \rVert_L \prod_{j=1}^n\tLip(\pi^j)\leq 1 \right\}\\
        &\leq
        \sup\left\{\left|\int_Q u_n \sum_if_i\det\Diff\pi_i\diff \mathcal{L}^d\right|: s(\sum_i (f_i,\pi_i))\leq 1 \right\} \\
        &=\sup_{\omega\in B_{\mathbb{P}^d}}|\langle \Em(\omega), v_n \rangle|\leq \frac{1}{n},
    \end{split}
    \end{equation*}
    and so the sequence $v_n$ satisfies the required properties from Definition \ref{D:Lip-reg}.
\end{proof}

In light of Corollary \ref{C:non-surj0} and Theorem \ref{T:f-a-characterisation}, our next goal will be to assert that $\Em$ is strongly non-surjective.

\section{Strong Lipschitz non-regularity of the prescribed Jacobian equation}\label{S:reg}
In this section, we show that if $d\geq 2$, then $\Em$ is strongly non-surjective and that therefore the prescribed Jacobian equation has strong Lipschitz non-regularity.

We recall the convention that the differential $\diff \mathcal{L}^d$ is usually omitted from integration notation and so we write for example
\begin{equation*}
    \int_Q f\det\Diff\pi=\int_Q f\det\Diff\pi\diff \mathcal{L}^d.
\end{equation*}
Throughout this section, the letter $Q$ will generally be used to denote a generic cube in $\reals^d$ and not refer to the unit cube $[0,1]^d$.

\subsection{Average determinants}\label{SS:AD}
Given a set $E$ and a map $f\colon E  \to \reals$, we use the notation
\begin{equation*}
    \lVert f \rVert_{\ell^\infty(E)}=\sup_{x\in E} |f(x)|.
\end{equation*}
We shall also use $\sigma$ to denote the surface measure of the boundary of the unit cube $[0,1]^d$.

\begin{lemma}\label{L:Average-det}
    For every $d\in\nat$, $d\geq 2$, there exists a constant $c_d\in(0,\infty)$ such that the following holds.
For every cube $Q\subset \reals^d$ of side length $r>0$, and every pair of Lipschitz maps $\pi,\kappa\colon Q \to \reals^d$, it holds that
    \begin{equation}\label{E:principal-Lip-estimate}
        |\int_Q \det \Diff \pi
        - \int_Q \det\Diff \kappa|\leq c_d L^{d-1} r^{d-1}\lVert \pi-\kappa \rVert_{\ell^\infty(\partial Q)},
    \end{equation}
    where $L$ is the maximum of the Lipschitz constants of the functions $\pi^1, \dots, \pi^d, \kappa^1, \dots, \kappa^d$.
The constant $c_d$ may be taken as the product of the number $d$ and the number of faces a $d$ dimensional cube has.
\end{lemma}
\begin{proof}
    Firstly, by the joint continuity property of determinants (Lemma \ref{L:j-c-for-dets}), it is sufficient to prove the statement for maps $\pi$ and $\kappa$ which  are defined and smooth on $\reals^d$.
    Since determinant is a multilinear map, we have the formula
    \begin{equation*}
    \begin{split}
        &\det D (\pi^1,\dots, \pi^d) - \det\Diff(\kappa^1, \dots, \kappa^d)
        =\det\Diff(\pi^1-\kappa^1, \pi^2,\dots, \pi^d)\\
        &+\det\Diff(\kappa^1, \pi^2-\kappa^2, \pi^3, \dots, \pi^d)
        +\dots+\det\Diff(\kappa^1,\dots,\kappa^{i-1}, \pi^i-\kappa^i, \pi^{i+1}, \dots, \pi^d)\\
        &+\dots+\det\Diff(\kappa^1, \dots, \kappa^{d-1}, \pi^d-\kappa^d).
    \end{split}
    \end{equation*}
    Since determinant is also anti-symmetric, the above formula implies that also
    \begin{equation*}
        \begin{split}
        &\det D (\pi^1,\dots, \pi^d) - \det\Diff(\kappa^1, \dots, \kappa^d)
        =\det\Diff(\pi^1-\kappa^1, \pi^2,\dots, \pi^d)\\
        &-\det\Diff(\pi^2-\kappa^2, \kappa^1, \pi^3, \dots, \pi^d)
        +\dots+(-1)^{i-1}\det\Diff(\pi^i-\kappa^i,\kappa^1,\dots,\kappa^{i-1}, \pi^{i+1}, \dots, \pi^d)\\
        &+\dots+(-1)^{d-1}\det\Diff(\pi^d-\kappa^d, \kappa^1, \dots, \kappa^{d-1}).
    \end{split}
    \end{equation*}
    Therefore, using the Stokes theorem \eqref{E:stokes} for each of the summands yields
    \begin{equation*}
        \begin{split}
            \int_Q \det \Diff \pi\ - \int_Q \det\Diff \kappa\
            =\sum_{i=1}^d\int_{\partial Q} (\pi^i-\kappa^i)\langle \omega_i(\kappa,\pi), \operatorname{Tan}_{\partial Q}\rangle \diff \sigma,
        \end{split}
    \end{equation*}
    where
    \begin{equation*}
        \omega_i(\kappa,\pi)=(-1)^{i-1}\diff\kappa^1\wedge\dots\wedge\diff\kappa^{i-1}\wedge\diff\pi^{i+1}\wedge\dots\wedge\diff\pi^d.
    \end{equation*}
    In particular, by triangle inequality,
    \begin{equation}\label{E:est-with-omega}
        |\int_Q \det \Diff \pi
        - \int_Q \det\Diff \kappa\
        |\leq \sum_i|\int_{\partial Q} (\pi^i-\kappa^i)\langle \omega_i(\kappa,\pi), \operatorname{Tan}_{\partial Q}\rangle \diff \sigma|
    \end{equation}
    For each $i\in\{1,\dots, d\}$, by \eqref{E:est-of-product}, we have
    \begin{equation*}
        |\langle \omega_i(\kappa,\pi), \operatorname{Tan}_{\partial Q}\rangle|\leq L^{d-1},
    \end{equation*}
    so that \eqref{E:est-with-omega} implies
    \begin{equation*}
        |\int_Q \det \Diff \pi - \int_Q \det\Diff \kappa|\leq \sum_i \lVert \pi^i - \kappa^i\rVert_{\ell^\infty(\partial Q)}L^{d-1}\sigma(\partial Q).
    \end{equation*}
    Let us denote by $\mathcal{F}(d)$ the number of faces a $d$-dimensional cube has.
Then $\sigma(Q)=r^{d-1}\mathcal{F}(d)$ so that in particular
    \begin{equation*}
         |\int_Q \det \Diff \pi - \int_Q \det\Diff \kappa|\leq d \mathcal{F}(d) r^{d-1} L^{d-1}\lVert \pi - \kappa\rVert_{\ell^\infty(\partial Q)},
    \end{equation*}
    which is \eqref{E:principal-Lip-estimate}.
\end{proof}

\begin{lemma}\label{L:est-on-squares-with-coef}
    Given a cube $Q\subset \reals^d$ and a function $f\colon Q \to \reals$, we let $f_Q$ be the value of $f$ at the center of the cube $Q$.
For any cube $Q\subset \reals^d$ of side length $r>0$ and any Lipschitz maps $f,g\colon Q \to \reals$ and $\pi,\kappa\colon Q \to \reals^d$ it holds that
    \begin{equation}\label{E:principal-estimate-coef}
    \begin{split}
        \left| \int_Q f \det \Diff \pi - \int_Q g \det\Diff\kappa \right|
        &\leq r^{d+1} (\tLip(f)+\tLip(g))\frac{1}{2}\sqrt{d}L^d\\
        &+r^d |f_Q-g_Q| L^d + r^{d-1} |f_Q| c_d L^{d-1} \lVert\pi-\kappa \rVert_{\ell^\infty(\partial Q)},
    \end{split}
    \end{equation}
    where $L$ is the maximum of the Lipschitz constants of the functions $\pi^1, \dots, \pi^d, \kappa^1, \dots, \kappa^d$.
\end{lemma}
Before we continue with the proof of Lemma \ref{L:est-on-squares-with-coef}, since the estimate is quite technical, we first point out why it will later be useful.
There are four summands on the right hand side of \eqref{E:principal-estimate-coef}.
At the end of our argument, we achieve a contradiciton, by showing that each of the summands is much smaller than the quantity on the left hand side.
This will be achieved very easily for each of the summands, except the last one, where most work needs to be carried out.
The first two summands are dealt with easily simply by taking a small enough cube, due to the factor of $r^{d+1}$ (while the left hand side only scales as $r^d$).
Furthermore, in our setup, it will be easy to achieve $|f_Q-g_Q|\lesssim r$, so that the third summand will also be very small for small enough $r>0$.
To deal with the final summand, we will need to cleverly find particular cubes $Q$ on which the estimate $\lVert \pi - \kappa \rVert_{\ell^\infty}\lesssim \varepsilon r$ holds for small $\varepsilon$.
This means that the fourth summand, on these cubes, is asymptotically (for small $r>0$) equivalent to $\varepsilon r^d$, which leads to a contradiction for small enough $\varepsilon>0$.

Since in reality we are interested in infinite sums of functions of the form $\sum_i f_i \det\Diff \pi_i$, the particular powers of the number $L$ in \eqref{E:principal-estimate-coef} are also very important and it turns out that powers appearing in the estimate \eqref{E:principal-estimate-coef} are precisely what is needed to carry out the argument.
How exactly this works out is partially illuminated by Lemmma \ref{L:est-on-cubes-sums}.

\begin{proof}[Proof of Lemma \ref{L:est-on-squares-with-coef}]
    We write
    \begin{equation}\label{E:4-fold-est}
        \begin{split}
        &\left| \int_Q f \det \Diff \pi - \int_Q g \det\Diff\kappa \right|
        \leq \left| \int_Q f \det \Diff \pi - \int_Q f_Q \det\Diff\pi \right|\\
        &+ \left| \int_Q f_Q \det \Diff \pi - \int_Q f_Q \det\Diff\kappa \right|
        + \left| \int_Q f_Q \det \Diff \kappa - \int_Q g_Q \det\Diff\kappa \right|\\
        &+ \left| \int_Q g_Q \det \Diff \kappa - \int_Q g \det\Diff\kappa \right|.
    \end{split}
    \end{equation}
    Since $|\det \Diff \pi|\leq L^d$ a.e.~in $Q$ and $|f-f_Q|\leq r\tLip(f)\frac{1}{2}\sqrt{d}$, we have
    \begin{equation}\label{E:summand-1}
        \left| \int_Q f \det \Diff \pi - \int_Q f_Q \det\Diff\pi \right|\leq r^{d+1}\tLip(f)\frac{1}{2}\sqrt{d}L^d.
    \end{equation}
    Analogously, we also obtain
    \begin{equation}\label{E:summand-4}
        \left| \int_Q g_Q \det \Diff \kappa - \int_Q g \det\Diff\kappa \right|\leq r^{d+1}\tLip(g)\frac{1}{2}\sqrt{d}L^d.
    \end{equation}
    Using Lemma \ref{L:Average-det},
    \begin{equation}\label{E:summand-2}
        \begin{split}
            \left| \int_Q f_Q \det \Diff \pi - \int_Q f_Q 
            \det\Diff\kappa \right|&= |f_Q| \left| \int_Q 
            \det \Diff \pi - \int_Q \det\Diff\kappa \right|\\ 
            &\leq r^{d-1}|f_Q|c_d L^{d-1}\lVert \pi - 
            \kappa \rVert_{\ell^\infty(\partial Q)}.
        \end{split}
    \end{equation}
    Using once again the fact that $|\det \Diff \kappa|\leq L^d$, we obtain
    \begin{equation}\label{E:summand-3}
        \left| \int_Q f_Q \det \Diff \kappa - \int_Q g_Q \det\Diff\kappa \right| \leq \int_Q |f_Q- g_Q||\det\Diff\kappa|\leq r^d |f_Q- g_Q| L^d.
    \end{equation}
    Combining \eqref{E:4-fold-est} with the estimates \eqref{E:summand-1}, \eqref{E:summand-2}, \eqref{E:summand-3} and \eqref{E:summand-4} yields the required inequality.
\end{proof}

For the following lemma, recall the Definition \ref{D:fls} of formal Lipschitz sums.

\begin{lemma}\label{L:est-on-cubes-sums}
    Suppose $Q\subset [0,1]^d$ is a cube of side length $r>0$ and $\sum_i(f_i,\pi_i)$ and $\sum_i(g_i, \kappa_i)$ are formal Lipschitz sums.
Suppose further that $L_i$ stands for the maximum of the Lipschitz constants of the functions $\pi_i^1, \dots, \pi_i^d, \kappa^1_i,\dots,\kappa^d_i$ for $i\in\nat$.
Then
    \begin{equation}\label{E:principal-est-coef-sum}
        \begin{split}
            |\int_Q \sum_i f_i \det\Diff\pi_i\ &- \int_Q \sum_i g_i\det\Diff \kappa_i|\\
            &\leq r^{d+1}\frac{1}{2}\sqrt{d} \sum_i(\tLip(f_i)+\tLip(g_i))L^d_i \\
            &+r^d\sum_i |(f_i)_Q-(g_i)_Q| L_i^d\\
            &+r^{d-1}c_d\sup_i\{|(f_i)_Q|\}\sqrt{\sum_i L_i^d}\sqrt{\sum_i L_i^{d-2}\lVert\pi_i-\kappa_i\rVert_{\ell^\infty(\partial Q)}^2}.
        \end{split}
    \end{equation}
\end{lemma}

\begin{proof}
    Since integration over $Q$ is a continuous linear functional and $\sum_i(f_i,\pi_i)$ is a formal Lipschitz sum, it follows that
    \begin{equation*}
        \int_Q \sum_i f_i \det\Diff\pi_i =\sum_i\int_Q f_i \det\Diff\pi_i.
    \end{equation*}
    and similarly for the $g_i$'s and $\kappa_i$'s.
Therefore, from the triangle inequality and Lemma \ref{L:est-on-squares-with-coef} we have
    \begin{equation}\label{E:midway-est-for-coef-sums}
        \begin{split}
            &|\int_Q \sum_i f_i \det\Diff\pi_i  - \int_Q \sum_i g_i\det\Diff \kappa_i |
            \leq \sum_i |\int_Q f_i \det\Diff\pi_i  - \int_Q  g_i\det\Diff \kappa_i |\\
            &\leq r^{d+1}\frac{1}{2}\sqrt{d} \sum_i(\tLip(f_i)+\tLip(g_i))L^d_i
            +r^d\sum_i |(f_i)_Q-(g_i)_Q| L_i^d\\
            &+r^{d-1}c_d \sum_i |(f_i)_Q| L_i^{d-1}\lVert \pi_i - \kappa_i \rVert_{\ell^\infty(\partial Q)}.
        \end{split}
    \end{equation}
    Writing $L^{d-1}$ as $L^{\frac{d}{2}}L^{\frac{d}{2}-1}$ and using Holder's inequality yields
    \begin{equation*}
    \begin{split}
        \sum_i |(f_i)_Q| L_i^{d-1}\lVert \pi_i - \kappa_i \rVert_{\ell^\infty(\partial Q)}
        &\leq \sup_i\{|(f_i)_Q|\} \sum_i L^{\frac{d}{2}}(L^{\frac{d}{2}-1}\lVert \pi_i - \kappa_i \rVert_{\ell^\infty(\partial Q)})\\
        &\leq \sup_i\{|(f_i)_Q|\}\sqrt{\sum_i L_i^d}\sqrt{\sum_i L_i^{d-2}\lVert\pi_i-\kappa_i\rVert_{\ell^\infty(\partial Q)}^2}.
    \end{split}
    \end{equation*}
    The last estimate combined with \eqref{E:midway-est-for-coef-sums} yields the desired inequality \eqref{E:principal-est-coef-sum}.
\end{proof}

Note that even if it assumed that $\sum_i(f_i, \pi_i)$ and $\sum_i(g_i, \kappa_i)$ are formal Lipschitz sums, it is still possible that the right hand side of the estimate \eqref{E:principal-est-coef-sum} is infinite.
This may for example happen if $L_i$ is much larger than
\begin{equation*}
    \left(\prod_{j=1}^d \tLip(\pi_i^j)\right)^{\frac{1}{d}}.
\end{equation*}
To avoid this issue (which is really only of technical nature) we shall consider sums where $L_i$ is the correct number to use.

\begin{definition}\label{D:regular}
    A formal Lipschitz sum $\sum_i(f_i,\kappa_i)$ is called \emph{regular} if
    \begin{enumerate}
        \item for each $i$ either $f_i=\pi_i=0$ or
    \begin{equation*}
        \max\{\tLip(f_i), \rVert f_i \rVert_\infty\}=1
    \end{equation*}
    and
    \begin{equation*}
        \tLip(\pi_i^j)= \tLip(\pi_i^k) \quad \text{for all $j,k\in\{1,\dots, d\}$,}
    \end{equation*}
        \item for each $i\in\nat$, $\pi_i(0)=0$.
    \end{enumerate}
    If the regular formal Lipschitz sum with which we are working is clear from context, the symbol $L_i$ stands for the value
    \begin{equation*}
        \tLip(\pi_i^j)\quad \text{where $j\in\{1,\dots ,d\}$ is arbitrary.}
    \end{equation*}
\end{definition}

\begin{proposition}\label{P:regular}
    For every formal Lipschitz sum $\sum_i(g_i, \kappa_i)$, there exists a regular formal Lipschitz sum $\sum_i(f_i, \pi_i) \sim \sum_i (g_i, \kappa_i)$.
If $\sum(f_i,\pi_i)$ is regular, then
    \begin{equation*}
        s(\sum_i(f_i,\pi_i))= \sum_i L_i^d.
    \end{equation*}
\end{proposition}
\begin{proof}
    The latter statement follows immediately from the definitions.
For the former statement, suppose $\sum_i(g_i, \kappa_i)$ is a formal Lipschitz sum.
Let $I_0\subset \nat$ stand for the set of those indices, for which
    \begin{equation*}
        T(g_i, \kappa_i)=0 \quad\text{for every $T\in\mathbb{P}_d([0,1]^d)$.}
    \end{equation*}
    Let $I_1=\nat\setminus I_0$ and let $i\in I_1$.
Then it is necessary that $\tLip(\kappa_i^j)>0$ for every $j=1,\dots, d$ and $g_i\not=0$.
We let
    \begin{equation*}
        L_i=  \left(\max\{\tLip(g_i), \lVert g_i \rVert_\infty\}\prod_{j=1}^d \tLip(\kappa_i^j)\right)^{\frac{1}{d}}.
    \end{equation*}
    Furthermore, we define
    \begin{equation*}
        f_i= \frac{g_i}{\max\{\tLip(g_i), \lVert g_i \rVert_\infty\}};\quad \pi_i^j=\frac{\kappa^j_i}{\tLip(\kappa^j_i)}L_i-\frac{\kappa^j_i(0)}{\tLip(\kappa^j_i)}L_i.
    \end{equation*}
    For $i\in I_0$, we let $f_i=0$ and $\pi_i=0$.
It is now easily checked that $\sum_i(f_i,\pi_i)$ is regular and, consistently with the convention in Definition \ref{D:regular}, $L_i$ agrees with $\tLip(\pi_i^j)$.
Using multilinearity and locality of currents, we have for any $i\in I_1$ and $T\in \mathbb{P}_d([0,1]^d)$
    \begin{equation*}
        T(f_i, \pi_i)= \frac{1}{\max\{\tLip(g_i), \lVert g_i \rVert_\infty\}}\frac{1}{\prod_{j=1}^d\tLip(\kappa^j_i)}L_i^dT(g_i, \kappa_i)=T(g_i,\kappa_i).
    \end{equation*}
    If $i\in I_0$, then $T(g_i,\kappa_i)=0=T(0,0) = T(f_i, \pi_i)$.
Thus, in particular,
    \begin{equation*}
        \sum_i T(f_i, \pi_i)=\sum_iT(g_i, \kappa_i).
    \end{equation*}
    By definition, this means $\sum_i(f_i,\pi_i)\sim \sum_i(g_i,\kappa_i)$.
\end{proof}

In light of the previous proposition, until we state and prove the penultimate result, we shall restrict our attention to regular sums.

\begin{lemma}\label{L:principal-est-adjacent-squares}
    Suppose $Q$ is a cube in $\reals^d$ of side length $r>0$.
Suppose that $Q'$ is another cube of the same radius which shares a face with $Q$ and assume $Q,Q'\subset [0,1]^d$.
Let $\tau=c_{Q'}-c_Q$, where $c_Q$ and $c_{Q'}$ are the centres of $Q$ and $Q'$ respectively.
Suppose $W_i$, $i\in\nat$ is a sequence of vectors in $\reals^d$.
Let $\sum_i (f_i, \pi_i)$ be a regular formal Lipschitz sum.
For each $i$, let
    \begin{equation*}
        \tilde{\pi}_i(x)=\pi_i(x-\tau)-W_i.
    \end{equation*}
    Suppose that for each $i$, $L_i=\tLip(\pi_i^j)$, $j=1,\dots, d$ and
    \begin{equation*}
        S=s(\sum_i(f_i,\pi_i))=\sum_i L^d_i.
    \end{equation*}
    Then
    \begin{equation}\label{E:principal-est-squares}
    \begin{split}
        &\left| \int_Q \sum_i f_i \det \Diff \pi_i- \int_{Q'} \sum_i f_i \det \Diff \pi_i \right|\\
        &\leq r^{d+1}(\sqrt{d}+1)S + r^{d-1}c_d \sqrt{S}\sqrt{\sum_iL_i^{d-2} \lVert \pi_i - \tilde{\pi}_i\rVert^2_{\ell^\infty(\partial Q)}}.
    \end{split}
    \end{equation}
\end{lemma}

\begin{proof}
    Let us denote $f_i^\tau(x)=f_i(x-\tau)$.
Since translation by $\tau$ is an isometry and since $W^i$ is a constant vector, we have
    \begin{equation*}
        \int_{Q'} \sum_i f_i \det\Diff\pi_i \;\diff\mathcal{L}^d=\int_Q \sum_i f_i^\tau \det\Diff\tilde{\pi}_i \;\diff\mathcal{L}^d.
    \end{equation*}
    Therefore, applying \eqref{E:principal-est-coef-sum}, we obtain
    \begin{equation*}
        \begin{split}
            &\left| \int_Q \sum_i f_i \det \Diff \pi_i - \int_{Q'} \sum_i f_i \det \Diff \pi_i \right|
            \leq r^{d+1}\frac{1}{2}\sqrt{d} \sum_i(\tLip(f_i)+ \tLip(f_i^\tau))L^d_i \\
            &+r^d\sum_i |(f_i)_Q-(f_i)_{Q'}| L_i^d\\
            &+r^{d-1}c_d\sup_i\{|(f_i)_Q|\}\sqrt{\sum_i L_i^d}\sqrt{\sum_i L_i^{d-2}\lVert\pi_i-\tilde{\pi}_i\rVert_{\ell^\infty(\partial Q)}^2}.
        \end{split}
    \end{equation*}
    Using that the formal Lipschitz sum is regular, we have
    \begin{equation*}
        \tLip(f_i)+\tLip(f_i^\tau)\leq 2\tLip(f_i)\leq 2, \quad |(f_i)_{Q}-(f_i)_{Q'}|\leq r \tLip(f_i)\leq r,  \quad \sup_i|(f_i)_Q|\leq 1.
    \end{equation*}
     Thus, the above estimate implies
    \begin{equation*}
        \begin{split}
            \left| \int_Q \sum_i f_i \det \Diff \pi_i - \int_{Q'} \sum_i f_i \det \Diff \pi_i \right|
            &\leq r^{d+1}\sqrt{d} S
            +r^d r S\\
            &+r^{d-1}c_d\sqrt{S}\sqrt{\sum_i L_i^{d-2}\lVert\pi_i-\tilde{\pi}_i\rVert_{\ell^\infty(\partial Q)}^2},
        \end{split}
    \end{equation*}
    which is equivalent to \eqref{E:principal-est-squares}.
\end{proof}

We are now ready to outline in more detail how the estimate \eqref{E:principal-est-squares} will be used.
Consider we find a solution to the equation
\begin{equation*}
    \sum_i f_i \det\Diff\pi_i =\rho \quad\text{a.e.~in $[0,1]^d$.}
\end{equation*}
where $\rho\in L^\infty$ only attains values $1$ and $2$.
Suppose further that there is plenty of very small pairs of adjacent cubes $Q$ and $Q'$, as in Lemma \ref{L:principal-est-adjacent-squares} such that $\rho=1$ on $Q$ and $\rho=2$ on $Q'$.
Then the left-hand side of \eqref{E:principal-est-squares} is equal to $r^d$, where $r$ is the side length of the squares.
If $r>0$ is small enough, then the term $r^{d+1}(\sqrt{d}+1)$ is negligible compared to $r^d$.
Therefore, if we are able to find a particular pair of squares $Q$ and $Q'$ where we also have the estimate
\begin{equation*}
    \sqrt{\sum_i L_i^{d-2}\lVert \pi_i-\tilde{\pi}_i\rVert_{\ell^\infty(\partial Q)}}\lesssim r \varepsilon,
\end{equation*}
then we obtain a contradiction.
Indeed, we would then have that the right hand side of \eqref{E:principal-est-squares} is estimated by $\approx r^{d+1}+ r^{d}r\varepsilon\lesssim \varepsilon r^d$.
On the other hand, the left hand side itself is equal to $r^d$, which is indeed a contradiction.
This means that if such a pair of ``good squares" exists, then for this particular choice of $\rho$, there is no solution to the PDE.
This is exactly what also happens in the arguments of Burago and Kleiner \cite{BK} and Dymond et al \cite{DKK}.

\subsection{Behaviour of biLipschitz maps along thin rectangles}

In this subsection we shall collect a kind of dichotomy result on Lipschitz maps, which will later enable us to find the pairs of ``good squares'' which we mentioned earlier.
The first instance of this type of result is due to Burago and Kleiner \cite{BK} but a more refined version of the result appears in the work of Dymond et al \cite{DKK}.
In fact, their result is \emph{almost} good enough for our argument.
All we need to do is extend their result (which works for maps into $\reals^n$, $n\geq d$) to work for maps into $\ell^2$.
Importantly, the constants obtained in \cite{DKK} are \emph{independent} of $n$.
This means that extending the result to work for maps into $\ell^2$ boils down to what is essentially a standard approximation argument.
This is the content of the following lemma.

Since we simply apply the results of \cite{DKK} and combine them with an approximation argument, the geometric flavour of the lemma cannot be seen from our proof.
Therefore, a reader interested in the geometry of the argument should consult the proofs of \cite[Lemma 3.2, Lemma 3.3]{DKK} located in the appendix B of the named paper.
The basic geometric ideas are also the same as in \cite{BK}, where only case $d=n=2$ is discussed and therefore the argument is slightly more accessible.

\begin{definition}
    For each $n\in\nat$, suppose $P_n\colon \ell^2 \to \reals^n$ is the orthonormal projection to $\reals^n = \tspan\{e_1, \dots, e_n\}$.
    We say that a map $h\colon E \to \ell^2$, where $E\subset \reals^n$, is \emph{finitely $L$-biLipschitz}, if there exists some $n_0$, such that for all $n\geq n_0$, the map $P_n \circ h$ is $L$-biLipschitz.
\end{definition}

Note that by basic properties of projections, any finitely $L$-biLipschitz map is also $L$-biLipschitz.

In the following, $\lVert \cdot \rVert$ stands for both the Euclidean norm on $\reals^d$ and the norm induced by the canonical scalar product on $\ell^2$.
Which norm is in use is given by the argument.
Figure \ref{fig:adjacent-squares-dichotomy} illustrates the following lemma.

\begin{lemma}\label{L:dichotomy}
    Let $d\in\nat$, $L\geq 1$, $\varepsilon>0$.
Then there exist $M\in\nat$, $\varphi>0$ and $N_0\in\nat$, such that for every $c>0$, $N\in\nat$, $N\geq N_0$ and all finitely $L$-biLipschitz mappings
    \begin{equation*}
        h\colon [0,c]\times [0, \frac{c}{N}]^{d-1}\to \ell^2,
    \end{equation*}
    at least one of the following statements holds.
    \begin{enumerate}
        \item\label{Enum:DKK1} There is some $i\in\{1,\dots, N-1\}$ such that for every $x\in Q_i=[\frac{(i-1)c}{N}, \frac{ic}{N}]\times[0, \frac{c}{N}]^{d-1}$
        \begin{equation*}
            \lVert h(x+\frac{c}{N}e_1)-h(x)-\frac{1}{N}(h(ce_1)-h(0))\rVert \leq \frac{c\varepsilon}{N};
        \end{equation*}
        \item\label{Enum:DKK2} There is some $z$ which lies simultanously in the lattice $\frac{c}{NM}\mathbb{Z}^d$ and in the set
        \begin{equation*}
            [0,c-\frac{c}{NM}]\times [0, \frac{c}{N}-\frac{c}{NM}]^{d-1}
        \end{equation*}
        such that
        \begin{equation*}
            \frac{\lVert h(z+\frac{c}{NM}e_1)-h(z)\rVert}{\frac{c}{NM}}>(1+\varphi)\frac{\lVert h(c e_1)-h(0)\rVert}{c}.
        \end{equation*}
    \end{enumerate}
\end{lemma}

\begin{proof}[Proof of Lemma \ref{L:dichotomy}]
Let $N, M$ and $\varphi$ be the constants obtained in \cite[Lemma 3.3]{DKK}.
Let us call, for the purposes of the proof, points which lie simultanously in the lattice $\frac{c}{NM}\mathbb{Z}^d$ and in the set
        \begin{equation*}
            [0,c-\frac{c}{NM}]\times [0, \frac{c}{N}-\frac{c}{NM}]^{d-1},
        \end{equation*}
\emph{lattice points}.
Note that there are only finitely many lattice points.
We denote by $P_n\colon \ell^2\to\nolinebreak \reals^n$ the orthogonal projection to $\tspan\{e_1, \dots, e_n\}$.
Recall that then
\begin{equation}\label{E:ell-2-norm-convergence}
    \lVert P_{n} u \rVert \to \lVert u \rVert, \quad \text{for every $u\in\ell^2$.}
\end{equation}
Suppose $h$ is finitely $L$-biLipschitz.
Let $n_0\in\nat$ be such that $P_n\circ h$ is $L$-biLipschitz for every $n\geq n_0$.
By \cite[Lemma 3.3]{DKK}, for every $n\geq n_0$, at least one of the following statements holds:
\begin{enumerate}[label={(\roman*)$_{n}$}]
        \item\label{Enum:DKKn1} There is some $i\in\{1,\dots, N-1\}$ such that for every $x\in Q_i=[\frac{(i-1)c}{N}, \frac{ic}{N}]\times[0, \frac{c}{N}]^{d-1}$
        \begin{equation}\label{E:almost-affine-n}
            \lVert (P_n\circ h)(x+\frac{c}{N}e_1)-(P_n\circ h)(x)-\frac{1}{N}((P_n\circ h)(ce_1)-(P_n\circ h)(0))\rVert \leq \frac{c\varepsilon}{N};
        \end{equation}
        \item\label{Enum:DKKn2} There is some lattice point $z$
        such that
        \begin{equation}\label{E:extra-strech-n}
            \frac{\lVert (P_n\circ h)(z+\frac{c}{NM}e_1)-(P_n\circ h)(z)\rVert}{\frac{c}{NM}}>(1+\varphi)\frac{\lVert (P_n\circ h)(c e_1)-(P_n\circ h)(0)\rVert}{c}.
        \end{equation}
\end{enumerate}
Therefore, the set of those $n\geq n_0$ such that \ref{Enum:DKKn1} holds is infinite or the set of those $n\geq n_0$ such that \ref{Enum:DKKn2} holds is infinite.
Suppose the former is the case.
Suppose $n_k$ is an increasing sequence of indices, such that statement \ref{Enum:DKKn1} holds for every $n=n_k$.
For each $k\in \nat$, there exists some $i_k\in\{1, \dots, N-1\}$, such that \eqref{E:almost-affine-n} holds for every $x\in Q_{i_k}$ with $n=n_k$.
Thus, there is at least one $i_0\in\{1, \dots, N-1\}$ and an increasing sequence of indices $k_l\in\nat$, such that for every $x\in Q_{i_0}$ and every $l\in\nat$, \eqref{E:almost-affine-n} holds with $n=n_{k_{l}}$.
By passing to the limit on the left hand side using \eqref{E:ell-2-norm-convergence}, we obtain that also \ref{Enum:DKK1} holds with $i=i_0$.

Suppose now that the set of those $n\geq n_0$ such that \ref{Enum:DKKn2} holds is infinite.
Then there exists an increasing sequence of indices $n_k$, such that statement \ref{Enum:DKKn1} holds for every $n=n_k$.
That is, there is, for each $k\in\nat$ a lattice point $z_k$ such that \eqref{E:extra-strech-n} holds for $n=n_k$ and $z=z_k$.
Since the set of possible lattice points that we consider is finite, once again there exists at least one lattice point $z_0$ and an increasing sequence of indices $k_l\in\nat$ such that \eqref{E:extra-strech-n} holds for $n=k_{k_l}$ and $z=z_0$.
Thus, \ref{Enum:DKK2} holds for $z=z_0$ by passing to the limit using \eqref{E:ell-2-norm-convergence}.
\end{proof}

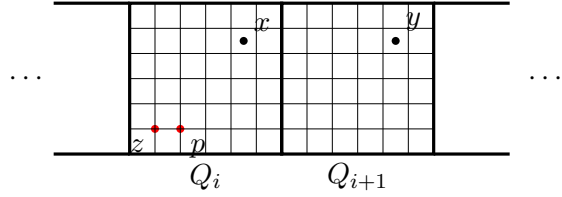
\begin{figure}[ht]
    \centering
    \begin{tikzpicture}
    \def\size{2}
    \def\step{1/3}

    \draw[very thick] (0,0) rectangle (\size,\size);
    \draw[very thick] (\size,0) rectangle (2*\size,\size);

    \node at (\size/2, -0.3) {$Q_i$};
    \node at (3*\size/2, -0.3) {$Q_{i+1}$};

    \coordinate (x) at (0.75*\size, 0.75*\size);
    \coordinate (y) at (x);
    \path (y) ++(\size,0) coordinate (y);
    \coordinate (z) at (1/6*\size, 1/6*\size);
    \coordinate (p) at (2/6*\size, 1/6*\size);

    \fill (x) circle (1.5pt);
    \node[above right] at (x) {$x$};
    \fill (y) circle (1.5pt);
    \node[above right] at (y) {$y$};
    \fill[red] (z) circle (1.5pt);
    \node[below left] at (z) {$z$};
    \fill[red] (p) circle (1.5pt);
    \node[below right] at (p) {$p$};
    \foreach \i in {1,2,3,4,5} {
        \draw[line width=0.3pt] (0,\i*\step) -- (\size,\i*\step);
        \draw[line width=0.3pt] (\i*\step,0) -- (\i*\step,\size);
        \draw[line width=0.3pt] (\size,\i*\step) -- (2*\size,\i*\step);
        \draw[line width=0.3pt] (\size+\i*\step,0) -- (\size+\i*\step,\size);
    }

    \draw[very thick] (-\size/2, \size) -- (0, \size);
    \draw[very thick] (-\size/2, 0) -- (0, 0);
    \node at (-\size/1.5, \size/2) {$\cdots$};

    \draw[very thick] (2*\size, \size) -- (2.5*\size, \size);
    \draw[very thick] (2*\size, 0) -- (2.5*\size, 0);
    \node at (2.75*\size, \size/2) {$\cdots$};
\end{tikzpicture}
    \caption{In the picture, $y = x+\frac{c}{N}e_1$ and $p=z+ \frac{c}{NM}e_1$.
The statement of Lemma \ref{L:dichotomy} is that if the stretching of $h$ on all pairs of points $z$ and $p$ is controlled by $(1+\varphi)$ times the stretching of $h$ on the endpoints $0$ and $ce_1$ (i.e.~statement \ref{Enum:DKK2} does not hold) then there must be some $Q_i$ such that $h(x)-h(y)\approx \textnormal{const}$ for all pairs $x, y$ where $y=x+\frac{1}{N}e_1$ and $x\in Q_i$.
If this is the case, this in particular implies that the image, under $h$, of $Q_{i+1}$ is nearly a translation of the image of $Q_i$.}
    \label{fig:adjacent-squares-dichotomy}
\end{figure}

In the following lemma, we connect formal Lipschitz sums with biLipschitz maps $h$ for which Lemma \ref{L:dichotomy} can be used.

\begin{lemma}\label{L:def-h-into-ell2}
    Let $d\in\nat$ and let $\sum_i(f_i,\pi_i)$ be a regular formal Lipschitz sum.
Denote by $L^i$ the Lipschitz constant of each $\pi_i^j$.
Define the map $h\colon [0,1]^d \to \ell^2$ given by
    \begin{equation*}
        h(x)=\sum_{j=1}^d x_j e_j + \sum_{i=1}^\infty\sum_{j=1}^{d}L_i^{\frac{d}{2}-1}\pi^j_i(x)e_{di+j} \quad \text{for $x\in [0,1]^d$.}
    \end{equation*}
    The definition is valid.
    If we denote
    \begin{equation*}
        S=\sum_i L_i^{d}=s(\sum_i(f_i,\pi_i)),
    \end{equation*}
    then $h$ is finitely $L$-biLipschitz for
    \begin{equation*}
        L=\sqrt{1+dS}.
    \end{equation*}
\end{lemma}
\begin{proof}
    For any $x, y \in [0,1]^d$ let us consider the sum
    \begin{equation*}
        \sum_{i=1}^\infty\sum_{j=1}^d L_i^{d-2}(\pi_i^j(x)-\pi_i^j(y))^2.
    \end{equation*}
    By the estimate
    \begin{equation}\label{E:basic-lip-ell2-est}
        (\pi_i^j(x)-\pi_i^j(y))^2\leq L_i^2 \lVert x- y \rVert^2,
    \end{equation}
    the sum converges.
Since $\pi_0(0)=0$ for every $i\in\nat$, it follows that the sum in the definition of $h$ converges absolutely and therefore $h$ is well defined.
    For any $x,y\in Q$, we have
    \begin{equation}\label{E:Lip-est-into-ell2}
        \lVert h(x)-h(y)\rVert^2=\lVert x-y\rVert^2+\sum_{i=1}^\infty\sum_{j=1}^d L_i^{d-2}(\pi_i^j(x)-\pi_i^j(y))^2
    \end{equation}
    and
    \begin{equation*}
        \lVert (P_n\circ h)(x)-(P_n\circ h)(y)\rVert^2\geq \lVert x-y\rVert^2,
    \end{equation*}
    whenever $n\geq d$.
    Therefore the inverse of $P_n\circ h$ is $1$-Lipschitz for every $n\geq d$.
On the other hand, using \eqref{E:basic-lip-ell2-est} and \eqref{E:Lip-est-into-ell2}, we have
    \begin{equation*}
        \lVert h(x)-h(y) \rVert^2 \leq \lVert x-y \rVert^2+\sum_{i=1}^\infty\sum_{j=1}^d L_i^{d-2}L_i^2\lVert x-y\rVert^2\leq \lVert x-y \rVert^2 + d S \lVert x-y \rVert^2,
    \end{equation*}
    so that $h$ is $L$-Lipschitz.
Thus, $P_n\circ h$ is $L$-biLipschitz for every $n\geq d$.
\end{proof}

\subsection{Iterative construction of thin rectangles}

In this subsection, we construct the function $\rho \in L^\infty$ together with the required weak$^*$ open neighbourhood to prove Proposition \ref{P:intro-strong-non-surj}.
The idea of the construction is not too complicated and essentially the same as in \cite{BK} and \cite{DKK}.
However, the required notation is fairly involved.
We begin by introducing all of which we shall need.

\begin{definition}
    Let $d\geq 1$ and suppose $Q\subset \reals^d$ is a cube of side length $r>0$ with its sides parallel to the coordinate axes.
We call the unique point $p\in \reals^d$ such that
    \begin{equation*}
        Q=[0,r]^d+p
    \end{equation*}
    the \emph{principal vertex} of $Q$.
We use the notation $p=p(Q)$.

    Any set $R\subset \reals^d$ of the form
    \begin{equation*}
        R=[0,c]\times[0,r]^d+p,
    \end{equation*}
    where $r<c$ is called a \emph{rectangle}.
The point $p$ is called the principal vertex of the rectangle $R$.
We use the notation $p=p(R)$.
We also call the principal vertex the \emph{left endpoint} of $R$ and the point $p+re_1$ is called the \emph{right endpoint} of $R$.
We use the notation $l(R)$ and $r(R)$ for the left and the right endpoints of $R$ respectively.
\end{definition}

For the remainder of this subsection we will consider $d\in\nat$, $d\geq 2$ to be fixed.
We define $P\colon \reals^d \to \reals^{d-1}$ as the orthogonal projection to the last $d-1$ coordinates.
Suppose $K$ and $M$ are given natural numbers with $K, M\geq 2$.
Consider the \emph{reference lattice}
\begin{equation*}
    \Reflat=\frac{1}{M}\mathbb{Z}^{d}\cap[0,1-\frac{1}{M}]^{d}.
\end{equation*}
Suppose $Q\subset\reals^d$ is a cube with sides parallel to the coordinate axes, of side length $r$ and its principal vertex denoted by $p$, i.e.
\begin{equation*}
    Q=[0,r]^d+p.
\end{equation*}
Given an element $z\in \Reflat$, we consider the rectangle
\begin{equation*}
    R(Q,z)=[0,\frac{r}{KM}]\times [0, \frac{r}{M}]^{d-1}+p+rz.
\end{equation*}
Here we recommend consulting Figure \ref{fig:constr-with-all-subrect}.
The rectangle $R(Q,z)$ has the property, that its long side is a integer multiple of its short sides.
Any rectangle having this property  is called \emph{integral} and is of the form
\begin{equation}\label{E:integral-rectangle}
    R=[0,c]\times[0, \frac{c}{n}]^{d-1} + p
\end{equation}
for some integer $n$.
We call the number $n$ the \emph{factor of $R$} and denote it by $F(R)$.
Given any integral rectangle $R$ of the form given in \eqref{E:integral-rectangle}, and $i\in\{0, 1, \dots, F(R)-1\}$, we define the cube $Q(R,i)$ by
\begin{equation}\label{E:partitioning-rect-into-cubes}
    Q(R,i)=[0, \frac{c}{F(R)}]^d+p(R) + i\frac{c}{F(R)}.
\end{equation}
Note that the cubes $Q(R,i)$ are non-overlapping and form a partition of $R$.

We let $R_0= [0, \frac{1}{K}]\times [0, 1]^d$.
We let $Q(R_0, i_0)= Q_i$ for $i_0\in\{0, \dots, K-1\}$.
We let $R_{i_0}^{z_1} = R(Q_{i_0},z_1)$ for $z_1\in \Reflat$.
We let $Q_{i_0, i_1}^{z_1} = Q(R_{i_0}^{z_1}, i_1)$ whenever $i_0, i_1 \in \{0, \dots, K-1\}$ and $z_1\in\Reflat$.
We continue inductively.
Suppose we have defined $R_{i_0, i_1, \dots, i_{k-1}}^{z_1,\dots, z_{k}}$ and $Q_{i_0, \dots, i_k}^{z_1, \dots, z_k}$ for $k\geq 2$, $i_0, \dots, i_k \in \{0, \dots, K-1\}$ and $z_1,\dots, z_{k}\in \Reflat$.
We let
\begin{equation*}
    R_{i_0, i_1, \dots, i_{k}}^{z_1,\dots, z_{k+1}}=R(Q_{i_0, \dots, i_k}^{z_1, \dots, z_k}, z_{k+1});\quad Q_{i_0, \dots, i_{k+1}}^{z_1, \dots, z_{k+1}} =  Q(R_{i_0, i_1, \dots, i_{k}}^{z_1,\dots, z_{k+1}}, i_{k+1}),
\end{equation*}
for $z_{k+1}\in \Reflat$ and $i_{k+1}\in\{0, \dots, K-1\}$.
Geometrically, $Q_{i_0}$'s naturally partition $R_0$ into subcubes.
Each subcube $Q_{i_0}$ contains a rescaled and translated copy of the reference lattice $\Reflat(Q_{i_0})= \frac{1}{K} \Reflat + p(Q_{i_0})$.
Each point $\tilde{z}\in \Reflat(Q_{i_0})$ is a principal vertex of one of the rectangles $R(Q_{i_0},z_1)=R_{i_0}^{z_1}$.
In fact, $p(R_{i_0}^{z_1}) = p(Q_{i_0})+\frac{z_1}{K}$ (note that the side length of each $Q_{i_0}$ is equal to $\frac{1}{K}$.)
Continuing in the same fashion, each rectangle $R_{i_0}^{z_1}$ naturally partitions into subcubes $Q^{z_1}_{i_0, i_1}$ and in each of these subcubes we consider the rectangles $R^{z_1, z_2}_{i_0, i_1}$ with principal vertices at in the rescaled lattice $\frac{1}{KM} \frac{1}{K} \Reflat + p(Q^{z_1}_{i_0, i_1})$.

There are two details to keep in mind.
Firstly, a detail of more technical nature - the scaling of the sides of the cubes.
At level ``$0$'', we have that the side length of $Q_{i_0}$ is equal to $\frac{1}{K}$, but at level ``$k\geq 1$'', the side length of $Q_{i_0, \dots, i_k}^{z_1, \dots, z_k}$ is equal to $\frac{1}{K}\frac{1}{(KM)^k}$.
In other words, the rescaling is by factor $\frac{1}{MK}$ at each step, while the ``initial cubes'' have side length $\frac{1}{K}$.
Secondly, and this is very important, each of the rectangles $R$ that we construct satisfies $F(R)=K$.
That is, the proportion of the long side to the short sides stays constant.

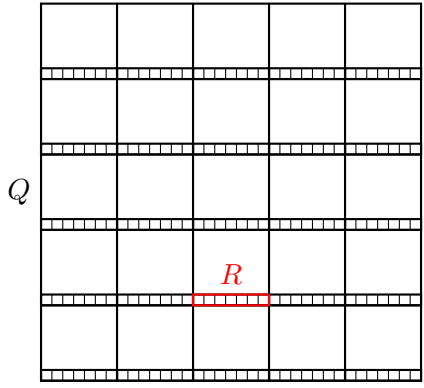
\begin{figure}[ht]
    \centering
    \begin{tikzpicture}
    \def\size{5} 
    \def\Rposx{2} 
    \def\Rposy{1}
    \def\listallR{0,1, 2, 3 , 4}
    \def\listQ{1,2,3,4} 
    \def\K{7} 
    \def\listR{1, 2, 3, 4, 5, 6} 
    \draw[thick] (0,0) rectangle (\size,\size);
    \node[left] at (0,\size/2) {$Q$};

    \foreach \x in \listQ {
        \draw[thick] (\x, 0) -- (\x, \size);
        \draw[thick] (0, \x) -- (\size, \x);
    }
    \foreach \z in \listallR {
        \foreach \p in \listallR{
            \draw[thick] (\z, \p) rectangle (\z+1, \p+1/\K);
            \foreach \x in \listR {
        \draw (\z+\x/\K, \p) -- (\z+\x/\K,\p+1/\K);
    }
        }
    }
    \draw[thick, red] (\Rposx, \Rposy) rectangle (\Rposx+1, \Rposy+1/\K);
    \node[above] at (\Rposx +0.5,\Rposy+0.15) {\textcolor{red}{$R$}};
\end{tikzpicture}
    \caption{
    A picture of a cube for $d=2$, $M=5$ and $K=7$.
The denoted rectangle $R$ (in red) is the rectangle $R(Q,(\frac{2}{M},\frac{1}{M}))$.
    Note that for any $z\in\Reflat$, $R(Q,z)\subset Q$ and moreover, if $K$ is very large, then the measure of the set $Q\setminus \bigcup_z R(Q,z)$ is nearly the measure of $Q$.
In the case $d=2$, we have explicitly $\H^2(Q\setminus \bigcup_z R(Q,z))=\H^2(Q)(1-\frac{1}{K}) = r^2 - \frac{r^2}{K}$.
    For each of the indicated rectangles on the picture, their natural partition into subcubes is also indicated, cf. \eqref{E:partitioning-rect-into-cubes}.
For example, $R=\bigcup_{i=1}^7 Q(R,i)$, where each $Q(R,i)$ is one of the very small squares inside the red rectangle $R$.
    If $Q=Q_{i_0, \dots, i_k}^{z_1, \dots, z_k}$,
    then each of the small squares is of the form
    $Q_{i_0, \dots, i_k, i_{k+1}}^{z_1, \dots, z_{k+1}}$
    and $R=R_{i_0, \dots, i_k}^{z_1, \dots, z_k, \scriptscriptstyle{(\frac{2}{M},\frac{1}{M})}}$.
    }
    \label{fig:constr-with-all-subrect}
\end{figure}

\begin{definition}
    Given integers $K, M \geq 2$, the rectangle $R_0$ is called \emph{admissible of order $0$} and each of the rectangles of the form $R_{i_0, i_1, \dots, i_{k-1}}^{z_1,\dots, z_{k}}$ is called \emph{admissible of order $k$}.
Each cube $Q_i$ is called \emph{admissible of order $0$} and each cube of the form $Q_{i_0, i_1, \dots, i_{k}}^{z_1,\dots, z_{k}}$ is called \emph{admissible of order $k$}.
    If a cube or a rectangle is admissible of some order $k$, it is called admissible.
We use the following notation
    \begin{enumerate}
        \item $\mathcal{R}_k$ for admissible rectangles of order $k$,
        \item $\mathcal{R}_{\leq k}$ for admissible rectangles of order at most $k$,
        \item $\mathcal{Q}_k$ for admissible cubes of order $k$,
        \item $\mathcal{Q}_{\leq k}$ for admissible cubes of order at most $k$.
    \end{enumerate}
    The set
    \begin{equation*}
        \{p(R_{i_0, \dots, i_{k}}^{z_1, \dots, z_{k+1}}): z_{k+1} \in \Reflat, i_{k} \in \{0, \dots, K-1\}\}
    \end{equation*}
    is called \emph{the lattice of $R_{i_0, i_1, \dots, i_{k-1}}^{z_1,\dots, z_{k}}$}.
Its elements are called \emph{lattice points in $R_{i_0, i_1, \dots, i_{k-1}}^{z_1,\dots, z_{k}}$}.
\end{definition}

We remark that the dependence of the entire construction, as well as the notion of admissibility on the parameters $K$ and $M$ is supressed in our notation.
The quantities $K$ and $M$ will be chosen at the end of the argument.

If one were to translate and rescale the rectangle $R_{i_0, i_1, \dots, i_{k-1}}^{z_1,\dots, z_{k}}$ to the rectangle $R_0$, say via a map $f\colon\reals^d \to \reals^d$, then the lattice of the rectangle $R_0$ is the image of the lattice of the rectangle $R_{i_0, i_1, \dots, i_{k-1}}^{z_1,\dots, z_{k}}$.
In particular, each such lattice is merely a rescaled and translated version of the lattice
\begin{equation*}
    \frac{1}{K}\Reflat \cup (\frac{1}{K}\Reflat + \frac{1}{K})\cup \dots \cup (\frac{1}{K}\Reflat + \frac{K-1}{K}).
\end{equation*}
Thus, each such lattice partitions into $K$ lattices, which are merely rescaled and translated copies of the reference lattice $\Reflat$ itself.
We recommend consulting also Figure \ref{fig:lattice}.
Note that since each admissible rectangle is a merely a translation of a kind of rectangle that appears as the domain of $h$ in Lemma \ref{L:dichotomy}, we may apply the lemma to each of these rectangles (of course, we need $K$ and $M$ to be large enough so that the lemma may be applied).
The lattice of the rectangle $R_{i_0, i_1, \dots, i_{k-1}}^{z_1,\dots, z_{k}}$ corresponds precisely to the lattice that appears in statement \ref{Enum:DKK2} of the aforementioned lemma after the identification of $R_{i_0, i_1, \dots, i_{k-1}}^{z_1,\dots, z_{k}}$ with $[0,\frac{c}{K}]\times[0,c]^{d-1}$ via a translation.

\begin{lemma}\label{L:orders-and-subsets}
    Let $k\in\nat\cup\{0\}$.
Then $\bigcup \mathcal{R}_k \supset \bigcup \mathcal{R}_{k+1}$.
Suppose $Q$ is an admissible cube of order $k$.
Then
    \begin{equation*}
        \mathcal{L}^d(Q\setminus \bigcup \mathcal{R}_{k+1})\geq \mathcal{L}^d(Q)(1-\frac{1}{K^{d-1}}).
    \end{equation*}
\end{lemma}
\begin{proof}
    The inclusion follows immediately from the construction.
We show that the estimate on the measure holds.
Firstly, the number of elements of $\mathcal{R}_{k+1}$ that intersect $Q$ in a set of non-zero measure is equal to $\#\Reflat$, cardinality of $\Reflat$, which is a consequence of the construction.
Of course, $\#\Reflat= M^d$.
Secondly, each $R\in \mathcal{R}_{k+1}$ satisfies, by definition,
    \begin{equation*}
        \mathcal{L}^d(R)= (r\frac{1}{MK})^{d-1}r\frac{1}{M},
    \end{equation*}
    where $r$ is the side length of $Q$.
Thus,
    \begin{equation*}
        \mathcal{L}^d(Q\cap\bigcup \mathcal{R}_{k+1})\leq M^d (r\frac{1}{MK})^{d-1}r\frac{1}{M}= r^d \frac{1}{K^{d-1}}= \mathcal{L}^d(Q)\frac{1}{K^d}.
    \end{equation*}
    The required estimate is therefore obtained via
    \begin{equation*}
        \mathcal{L}^d(Q\setminus \bigcup \mathcal{R}_{k+1})= \mathcal{L}^d(Q) - \mathcal{L}^d(Q\cap\bigcup \mathcal{R}_{k+1})\geq \mathcal{L}^d(Q)- \mathcal{L}^d(Q)\frac{1}{K^d}.
    \end{equation*}
\end{proof}

For a map $h\colon R \to \ell^2$ and an admissible rectangle $R$, we let
\begin{equation*}
    A_h(R)= \frac{\lVert h(l(R))-h(r(R))\rVert}{\lVert (l(R))-(r(R))\rVert}.
\end{equation*}

The following is essentially Lemma \ref{L:dichotomy} with a changed notation.

\begin{lemma}\label{L:dichotomy-good-rectangles}
    For every $L\geq 0$, $\varepsilon>0$ and $d\in\nat$, $d\geq 2$, there exists an $N_0, M\in\nat$, $\varphi>0$ such that for any $K\geq N_0$ the following holds.
Suppose $h\colon [0,1]^d\to \ell^2$ is a finitely $L$-biLipschitz map.
Then for each $k\in\nat\cup\{0\}$ and any admissible rectangle $R$ of order $k$, one of the following statements holds.
    \begin{enumerate}
        \item\label{Enum:DKK1-const1} There exists some $i\in\{0, \dots, K-2\}$ such that for every $x\in Q(R, i)$,
        \begin{equation*}
            \lVert h(x) - h(x+\frac{1}{K(KM)^k}e_1)-\frac{1}{K}(h(l(R))-h(r(R))\rVert \leq \frac{1}{K(KM)^k}\varepsilon;
        \end{equation*}
        \item\label{Enum:DKK1-const2} There exists some admissible rectangle $R'\in \mathcal{R}_{k+1}$ with $R'\subset R$ such that
        \begin{equation}\label{E:stretching-on-lattice}
            A_h(R')>(1+\varphi)A_h(R).
        \end{equation}
        Equivalently, there exists $z\in\Reflat$ and $i\in\{0,1, \dots, K-1\}$ such that the rectangle $R'=R(Q(R,i),z)$ satisfies \eqref{E:stretching-on-lattice}.
    \end{enumerate}
\end{lemma}
\begin{proof}
    Let $M, N_0, \varphi$ be the constants obtained in Lemma \ref{L:dichotomy}.
    Let us write
    \begin{equation*}
        R=[0,c]\times[0, \frac{c}{K}]^{d-1} + l(R),
    \end{equation*}
    where $c=\frac{1}{(MK)^k}$.
    Applying Lemma \ref{L:dichotomy} to the function $x\mapsto h(x-l(r))$, it follows that if statement \ref{Enum:DKK1-const1} does not hold, it must be the case that there exists a point $y$ which lies simultaneously in the lattice $\frac{c}{KM}\mathbb{Z}^d$ and in the set
    \begin{equation*}
        [0, c- \frac{c}{NM}]\times [0, \frac{c}{K}-\frac{c}{KM}]
    \end{equation*}
    such that
    \begin{equation*}
        \frac{\lVert h(l(R)+y+\frac{c}{KM}e_1)-h(l(R)+y)\rVert}{\frac{c}{KM}}> (1+\varphi) \frac{\lVert h(l(R)+ce_1)-h(l(R))\rVert}{c}.
    \end{equation*}
    Consider the rectangle $R'\in \mathcal{R}_{k+1}$ uniquely determined by $l(R')=l(R)+y$.
Then $r(R')=l(R)+y+ \frac{c}{KM}e_1$ and therefore the rectangle $R'$ has the required properties.
\end{proof}

\begin{figure}[ht]
    \centering
    \begin{tikzpicture}
    \def\bs{1.8}
    \def\num{7} 
    \def\M{5} 
    \def\choice{3} 
    \def\locsq{2} 
    \def\xcoor{2} 
    \def\ycoor{3} 
    \def\dsize{0.04} 
    \def\blist{1,2,...,\num} 
    \def\llist{1,2,...,\M}
    \draw (0,0) rectangle (\num*\bs,\bs);
    \node[left] at (0,\bs/2) {$R$};
    \foreach \x in \blist {
        \draw (\x*\bs, 0) -- (\x*\bs, \bs);
    }

    \foreach \sq in \blist {
        \foreach \x in \llist {
        \foreach \y in \llist {
            \fill[black] (\x*\bs/5-\bs/5+\sq*\bs-\bs,\y*\bs/5-\bs/5) circle (\dsize);
        }
    }

    }
    \foreach \x in \llist {
        \foreach \y in \llist {
            \fill[blue] (\x*\bs/5-\bs/5+\choice*\bs,\y*\bs/5-\bs/5) circle (\dsize);
        }
    }
    \draw[red] (\locsq*\bs + \xcoor*\bs/\M,\ycoor*\bs/\M) rectangle (\locsq*\bs + \xcoor*\bs/\M +\bs/\M,\ycoor*\bs/\M+\bs/\M/\num);
\end{tikzpicture}
    \caption{A picture of a rectangle together with its lattice, in the case $d=2$, $K=7$ and $M=5$.
If $R=R_{i_0, i_1, \dots, i_{k-1}}^{z_1,\dots, z_{k}}$, then the part of lattice in blue is a rescaled and translated copy of the reference lattice $\Reflat$.
    Explicitly, it is equal to $\frac{1}{K}\frac{1}{(KM)^k}\Reflat+p(R_{i_0, i_1, \dots, i_{k-1}}^{z_1,\dots, z_{k}}) + (0, \frac{3}{K(KM)^k})$.
The $3$ appears because the blue lattice lies in the fourth square from the left and the indexing of the squares starts with $0$.
A different way of writing the blue lattice is $\frac{1}{K}\frac{1}{(KM)^k}\Reflat+p(Q_{i_0, i_1, \dots, i_{k-1}, \scriptscriptstyle 3}^{z_1,\dots, z_{k}})$.
    We further remark that for any pair of horizontally adjacent lattice points, say $p,q$, with $p$ left of $q$, there is a rectangle of the form $R_{i_0, i_1, \dots, i_{k}}^{z_1,\dots, z_{k+1}}$ such that $p$ is the left endpoint of $R_{i_0, i_1, \dots, i_{k}}^{z_1,\dots, z_{k+1}}$ while $q$ is the right endpoint of $R_{i_0, i_1, \dots, i_{k}}^{z_1,\dots, z_{k+1}}$.
    The tiny red rectangle is the rectangle $R_{i_0, i_1, \dots, i_{k-1}, \scriptscriptstyle 3}^{z_1,\dots, z_{k}, \scriptscriptstyle (2,3)}$.
    }
     \label{fig:lattice}
\end{figure}
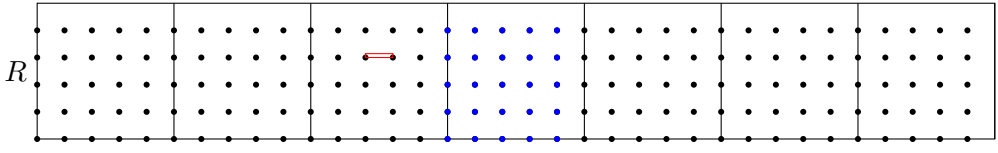

\begin{definition}
    Suppose $\varepsilon>0$ and $\varphi>0$.
Let $h\colon [0,1]^d\to \ell^2$ be a function.
We say that an admissible rectangle $R$ \emph{has $\varepsilon$-property $1$ (with respect to $h$)}, if statement \ref{Enum:DKK1-const1} of Lemma \ref{L:dichotomy-good-rectangles} holds.
We say that $R$ \emph{has $\varphi$-property $2$ (with respect to $h$)}, if statement \ref{Enum:DKK1-const2} of Lemma \ref{L:dichotomy-good-rectangles} holds.
\end{definition}

\begin{lemma}\label{L:not-2}
    Let $L\geq 1$ and $d\in \nat$, $d\geq 2$.
Then, for any $\varphi>0$, there exists some $k_0\in\nat$ such that the following is true.
    For any $L$-biLipschitz $h\colon [0,1]^d\to \ell^2$, there exists a rectangle $R\in\mathcal{R}_{\leq k_0}$ such that $R$ does \emph{not} have $\varphi$-property $2$.
\end{lemma}
\begin{proof}
    Let $k_0\in\nat$ be such that $(1+\varphi)^{k_0}\geq L^2$.
    Suppose, for contradiction, that every rectangle $R\in \mathcal{R}_{\leq k_0}$ has $\varphi$-property $2$.
Then, by definition, we find a sequence of rectangles $R_1, R_2, \dots, R_{k_0}$, each $R_k\in\mathcal{R}_k$ with $R_0\supset R_1 \supset R_2\supset \dots\supset R_{k_0}$ and such that
    \begin{equation*}
        A_h(R_k)>(1+\varphi) A_h(R_{k-1})\quad \text{for each $k=1,2,\dots, k_0$.}
    \end{equation*}
    Thus, $A_h(R_{k_0})> (1+\varphi)^{k_0} A_h(R_0)\geq (1+\varphi)^{k_0} \frac{1}{L}$.
Since $A_h(R_{k_0})\leq L$, this implies
    \begin{equation*}
        L> (1+\varphi)^{k_0} \frac{1}{L},
    \end{equation*}
    which contradicts the choice of $k_0$.
\end{proof}

\begin{corollary}\label{C:good-rectangle}
    Suppose $d\geq 2$, $L\geq 1$ and $\varepsilon>0$.
Then there exists $N_0, M\in\nat$ and $k_0\in\nat$ such that for every $K\geq N_0$, the following is true.
If $h\colon [0,1]^d\to \ell^2$ is a finitely $L$-biLipschitz map, then there exists $R\in \mathcal{R}_{\leq k_0}$ which has $\varepsilon$-property $1$.
\end{corollary}
\begin{proof}
    Let $N_0$ and $M$ and $\varphi$ be from Lemma \ref{L:dichotomy-good-rectangles}.
Find $k_0\in\nat$ from Lemma \ref{L:not-2}.
By Lemma \ref{L:dichotomy-good-rectangles}, for each rectangle $R\in\mathcal{R}_{\leq k_0}$, R has $\varepsilon$-property $1$ or $\varphi$-property $2$.
By Lemma \ref{L:not-2}, there exists some $R\in\mathcal{R}_{\leq k_0}$ which does not have $\varphi$-property $2$.
Thus, $R$ has $\varepsilon$-property $1$.
\end{proof}

\begin{theorem}\label{T:finding-good-squares}
    Suppose $S\geq 0$ and $\varepsilon>0$ and $d\in\nat$, $d\geq 2$.
Then there exists some $N_0, M, k_0\in\nat$ such that for any $K\in\nat$ with $K\geq N_0$, the following is true.
Let $\sum_i(f_i, \pi_i)$ be a regular formal Lipschitz sum with $s(\sum_i(f_i,\pi_i))\leq S$.
Then there exists an admissible rectangle $R$ of order $k\leq k_0$ and some $n\in\{0,1,\dots K-2\}$ such that the pair of cubes $Q=Q(R,n)$ and $Q'=Q(R,n+1)$ of order $k$ has the following property.
    For each $i\in\nat$, a vector $W_i\in\reals^d$ exists such that if we define
    \begin{equation*}
        \tilde{\pi}_i(x)=\pi_i(x+\tau)-W_i,
    \end{equation*}
    for $\tau = \frac{1}{K(KM)^k}e_1$ (the vector which translates $Q$ to $Q'$), then
    \begin{equation*}
        \sum_i L_i^{d-2} \lVert \pi_i - \tilde{\pi}_i\rVert_{\ell^\infty(\partial Q)}^2\leq r^2 \varepsilon^2,
    \end{equation*}
    where $r=\frac{1}{K(KM)^k}$ is the side length of $Q$ and $Q'$.
    The choice $W_i= \frac{1}{K}(\pi_i(l(R))-\pi_i(r(R)))$ suffices.
\end{theorem}
\begin{proof}
    Let $L=\sqrt{1+dS}$ and find $M, N_0$ and $k_0$ from Corollary \ref{C:good-rectangle}.
Suppose $K\in\nat$ and $K\geq N_0$.
Let $\sum_i(f_i,\pi_i)$ be a regular formal Lipschitz sum with $s(\sum_i(f_i,\pi_i))\leq S$.
Let $h$ be given by the formula from Lemma \ref{L:def-h-into-ell2} and recall that then $h$ is finitely $L$-biLipschitz.
Therefore, by Corollary \ref{C:good-rectangle}, there exists an admissible rectangle $R\in\mathcal{R}_k$ for some $k\leq k_0$ such that $h$ has $\varepsilon$-property $1$ on $R$.
By definition, this means that there is some $n\in\{0,\dots, K-2\}$ an admissible cube $Q=Q(R,n)$ of order $k$, such that
    \begin{equation*}
        \lVert h(x+\tau)- h(x) - \frac{1}{NM}(h(\Rright)-h(\Rleft))\rVert^2 \leq \left(\frac{\varepsilon}{NM(NM)^k}\right)^2.
    \end{equation*}
    By the definition of $h$, in particular
    \begin{equation}\label{E:h-norm-est-on-good-square}
        \begin{split}
            &\sum_{i=1}^\infty \sum_{j=1}^d L^{d-2}_i [\pi_i^j(x+\tau)-\pi_i^j(x)-\frac{1}{NM}(\pi_i^j(\Rright)-\pi_i^j(\Rleft))]^2\\
            &\leq \left(\frac{\varepsilon}{NM(NM)^k}\right)^2,    
        \end{split}
    \end{equation}
    for every $x\in Q$.
    We let $W_i= \frac{1}{NM}(\pi_i(\Rright)-\pi_i(\Rleft))$.
Note that now the left hand side of \eqref{E:h-norm-est-on-good-square} reads as
    \begin{equation*}
        \begin{split}
            \sum_{i=1}^\infty \sum_{j=1}^d L^{d-2}_i [\pi_i^j(x+\tau)-\pi_i^j(x)-W_i^j]^2&=\sum_{i=1}^\infty L_i^{d-2} \lVert\pi_i(x+\tau) - \pi_i(x) -W_i \rVert^2\\
            &=\sum_i L_i^{d-2} \lVert \pi_i(x) - \tilde{\pi}_i(x)\rVert^2,
        \end{split}
    \end{equation*}
    where the Pythagoras's theorem was used in the first equality.
    Therefore, the theorem follows from \eqref{E:h-norm-est-on-good-square} and $r=\frac{1}{NM(NM)^k}$.
\end{proof}

\begin{corollary}\label{C:est-on-good-square}
     Suppose $S\geq 0$ and $\varepsilon>0$ and $d\in\nat$, $d\geq 2$.
Then there exists some $N_0, M, k_0\in\nat$ such that for any $K\in\nat$ with $K\geq N_0$, the following is true.
Let $\sum_i(f_i, \pi_i)$ be a regular formal Lipschitz sum with $s(\sum_i(f_i,\pi_i))\leq S$.
Then there exists an admissible rectangle of order $k\leq k_0$ and some $n\in\{0,1,\dots K-2\}$ such that the pair of cubes $Q=Q(R,n)$ and $Q'=Q(R,n+1)$ of order $k$ satisfies the estimate
     \begin{equation*}
         |\int_Q \sum_i f_i\det\Diff\pi_i- \int_{Q'} \sum_i f_i \det\Diff\pi_i| \leq r^{d+1} (\sqrt{d}+1)S+r^dc_d\sqrt{S}\varepsilon.
     \end{equation*}
\end{corollary}
\begin{proof}
    The rectangle and the cubes are found in Theorem \ref{T:finding-good-squares}.
The estimate then follows from the combination of the estimate in Theorem \ref{T:finding-good-squares} and Lemma \ref{L:principal-est-adjacent-squares}.
\end{proof}

Finally, we are prepared to construct particularly ill-behaved right hand sides for the PDE \eqref{E:intro-the-equation1}.

Suppose we are given some $k\in\nat$, a bounded function $\rho\colon [0,1]^d\to\reals$.
We now construct a \emph{refinement of $\rho$ at depth $k$} in the following way.
If $x\in [0,1]^d$ does not lie inside any admissible rectangle of order $k$, we let $\refine_k(\rho)(x)=\rho(x)$.
Suppose $x\in R$ for some $R\in \mathcal{R}_k$.
Then there exists some $i\in\{0, \dots, K-1\}$, such that $x\in Q(R,i)$.
If such $i$ can be found even, we let $\refine_k(\rho)(x)=1$ and we let $\refine_k(\rho)(x)=2$ otherwise\footnote{For almost every $x$, the $i$ is unique.
There can only be multiple $i$'s if $x$ lies on the face of one of the cubes.}.
We use the notation
\begin{equation*}
    \refine_{[k]}(\rho)=\refine_k(\refine_{k-1}(\dots(\refine_1(\rho)))).
\end{equation*}

For illustration, see Figure \ref{fig:rho-1}, where a particular instance of the function $\refine_1(\rho)$ is indicated.

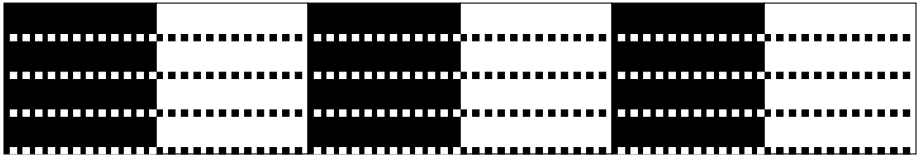
\begin{figure}[ht]
    \centering
    \begin{tikzpicture}
    \def\bs{2}
    \def\num{6} 
    \def\M{4} 
    \def\alllist{0,1,2,3,4,5} 
    \def\elist{0,2,4} 
    \def\olist{1,3,5} 
    \def\llist{0,1,2,3} 

    \foreach \x in \elist {
        \draw[fill] (\x*\bs, 0) rectangle (\x*\bs+\bs, \bs);
    }
    \foreach \x in \olist {
        \draw (\x*\bs, 0) rectangle (\x*\bs+\bs, \bs);
    }
    \foreach \sq in \olist{
    \foreach \x in \llist{
    \foreach \y in \llist {
    \foreach \ssq in \elist
    \draw[fill] (\sq*\bs+\x*\bs/\M+\ssq*\bs/\M/\num,\y*\bs/\M) rectangle (\sq*\bs+\x*\bs/\M + \ssq*\bs/\M/\num+\bs/\M/\num,\y*\bs/\M + \bs/\M/\num);
    }
    }
    }
    \foreach \sq in \elist{
    \foreach \x in \llist{
    \foreach \y in \llist {
    \foreach \ssq in \olist
    \draw[fill, white] (\sq*\bs+\x*\bs/\M+\ssq*\bs/\M/\num,\y*\bs/\M) rectangle (\sq*\bs+\x*\bs/\M + \ssq*\bs/\M/\num+\bs/\M/\num,\y*\bs/\M + \bs/\M/\num);
    }
    }
    }
\end{tikzpicture}
    \caption{Picture of $\refine_1(\rho)$ in case $K=6$, $M=4$ and $d=2$.
The whole rectangle is the initial rectangle $R_0$.
The black regions correspond to sets on which $\refine_1(\rho)=1$ and the white regions correspond to sets on which $\refine_1(\rho)=2$.}
    \label{fig:rho-1}
\end{figure}

\begin{lemma}\label{L:w*-con-subsets}
    Suppose $\rho\in L^\infty([0,1]^d)$ with $\lVert \rho \rVert_\infty\leq 2$ and $V$ is a weak$^*$ open neighbourhood of $\rho$.
Then there exists some $K_0$ such that for every $K\geq K_0$ and for every $k\in\nat$, the function $\refine_{[k]}\rho\in V$.
\end{lemma}
\begin{proof}
    Recall the definition of the initial rectangle $R_0=[0,\frac{1}{K}]\times [0,1]^{d-1}$.
We clearly have $\mathcal{L}^d(R_0)\to 0$ as $K\to \infty$.
By the construction, $\rho=\refine_{[k]}(\rho)$ a.e.~outside $R_0$.
It follows therefore that for any fixed $k\in\nat$, $\refine_{[k]}(\rho)\to \rho$ weak$^*$ as $K\to \infty$.
Since $V$ is weak$^*$ open, the required $K_0$ exists by the definition of limit.
\end{proof}

\begin{lemma}\label{L:differences-of-rho-neighbouring}
    Suppose $k_0 \geq 0$ and $\rho \in L^\infty([0,1]^d)$.
Suppose $R$ is admissible of order $k\leq k_0$, $n\in\{0, 1, \dots, K-2\}$ and $Q=Q(R,n)$, $Q'=Q(R,n+1)$.
Let $r>0$ be the side length of $Q$ and $Q'$.
Then for $\rho_{k_0}=\refine_{[k_0]}\rho$, we have
    \begin{equation}\label{E:integral-rho-neighbouring}
        |\int_Q \rho_{k_0} - \int_{Q'}\rho_{k_0}| \geq r^d (1-5\frac{1}{K^{d-1}})
    \end{equation}
\end{lemma}
\begin{proof}
    Suppose $n$ is even.
If $n$ is odd, the proof is analogous.
    By construction, $\rho_k=1$ a.e.~on $Q$ and $\rho_k=2$ a.e.~on $Q'$.
    Using Lemma \ref{L:orders-and-subsets} and the construction, we see that each $\rho_{k+1}, \rho_{k+2}, \dots, \rho_{k_0}$ differs from $\rho_k$ only on the set $\bigcup \mathcal{R}_{k+1}$ and, moreover,
    \begin{equation}\label{E:smaller-order-est}
        \mathcal{L}^d(Q \setminus \bigcup \mathcal{R}_{k+1})\geq r^d(1-\frac{1}{K^{d-1}}); \quad \mathcal{L}^d(Q' \setminus \bigcup \mathcal{R}_{k+1})\geq r^d(1-\frac{1}{K^{d-1}}).
    \end{equation}
    Let $E= Q\setminus \bigcup \mathcal{R}_{k+1}$ and $E'=Q'\setminus \bigcup \mathcal{R}_{k+1}$.
Then we have
    \begin{equation*}
        \begin{split}
            \left|\int_Q \rho - \int_{Q'}\rho\right|
            &= \left|\int_E \rho - \int_{E'}\rho + \int_{Q\setminus E}\rho-\int_{Q'\setminus E'}\rho\right|\\
            &\geq \left|\int_E \rho - \int_{E'}\rho\right|-\left|\int_{Q\setminus E}\rho\right|-\left|\int_{Q'\setminus E'}\rho\right|.
        \end{split}
    \end{equation*}
    Since $|\rho|\leq 2$ a.e.~and using the inequalities \eqref{E:smaller-order-est}, we have
    \begin{equation*}
        |\int_{Q\setminus E}\rho|+|\int_{Q'\setminus E'}\rho|\leq r^d 2 \cdot 2 \frac{1}{K^{d-1}}.
    \end{equation*}
    Since $\rho=1$ a.e.~on $E$ and $\rho = 2$ a.e.~on $E'$, we have
    \begin{equation*}
        |\int_E \rho - \int_{E'}\rho|= r^d(1-\frac{1}{K^{d-1}}).
    \end{equation*}
    The combination of the three estimates above implies \eqref{E:integral-rho-neighbouring}.
\end{proof}

Here we recall some basic facts about weak$^*$ topology on $L^\infty([0,1]^d)$.
Firstly, let us fix a standard mollificaiton kernel $\psi\colon \reals^d \to \reals$.
We shall assume that $\psi$ is even, smooth, of unit $L^1$-norm and supported inside the unit ball of $\reals^d$.
For $\delta>0$ we let $\psi_\delta(x)= \delta^d\psi(\frac{x}{\delta})$.
Given $\rho \in L^\infty([0,1]^d)$, we denote its extension to $\reals^d$ by $0$ outside of $[0,1]^d$ by $\rho$ as well.
Then $\psi_\delta * \rho \to \rho$ weak$^*$ in $L^\infty(\reals^d)$ and therefore $(\psi_\delta * \rho)_{|[0,1]^d}\to \rho$ weak$^*$ in $L^\infty([0,1]^d)$ and we have the norm estimate
\begin{equation*}
    \lVert \psi_\delta * \rho \rVert_{\infty}\leq \lVert \rho \rVert_{\infty}.
\end{equation*}
In particular, restrictions of smooth maps are weak$^*$ dense in $L^\infty([0,1]^d)$.

Given any measurable set $E\subset [0,1]^d$, the functional
\begin{equation*}
    \rho \mapsto \int_E \rho
\end{equation*}
is weak$^*$ continuous on $L^\infty(Q)$ by definition of weak$^*$ topology and from the fact that the functional is an action of the $L^1([0,1]^d)$ function $\chi_E$ on $L^\infty([0,1]^d)$.

The point of the following lemma is that if all the information we require of a given $\rho\in L^\infty([0,1]^d)$ are its averages over a finite number of cubes, then a weak$^*$ open set of $\rho$'s can be considered.
We shall use the notation $L^\infty = L^\infty([0,1]^d)$ and we recall that $\lambda B_{L^\infty}$ stands for the ball around $0$ of radius $\lambda$ in $L^\infty$.

\begin{lemma}\label{L:smoothing}
    Suppose $\sigma\in 2 B_{L^\infty}$ and $V$ is a weak$^*$ open neighbourhood of $\sigma$.
Then there is $K_0\in \nat$ such that for every
    $\varepsilon>0$, $K\in \nat$, $K\geq K_0$ $M\in\nat$, $k_0 \in \nat$ we have the following.
Let $\mathcal{Q}^2_{\leq k_0}(\textnormal{adj})$ be the collection of all the ordered pairs $(Q, Q')$ where $R\in \mathcal{R}_{\leq k_0}$, $n\in\{0, \dots, K-2\}$,
    $Q=Q(R,n)$, $Q'=Q(R,n+1)$.
For any pair $(Q, Q')$, the side lengths of $Q$ and $Q'$ are equal and denoted by $r(Q)$.
The set
    \begin{equation*}
    \begin{split}
        U=\left\{\rho \in L^\infty([0,1]^d):\left|\int_Q \rho - \int_{Q'}\rho\right| > r^d(Q)(1-\varepsilon - \frac{5}{K^{d-1}})\;\text{for all $(Q,Q')\in \mathcal{Q}^2_{\leq k_0}(\textnormal{adj})$}\right\},
    \end{split}
    \end{equation*}
    is a weak$^*$ open set containing a smooth function $\rho\in V$ with $\lVert \rho \rVert_\infty \leq 2$.
\end{lemma}
\begin{proof}
    Firstly, the fact that side lengths of $Q$ and $Q'$ agree whenever $(Q,Q')\in \mathcal{Q}^2_{\leq k_0}(\textnormal{adj})$ is an immediate consequence of the fact that $Q$ and $Q'$ are of the same order.

    By Lemma \ref{L:w*-con-subsets}, we find $K_0$ such that the map $\rho_{k_0}=\refine_{[k_0]}(\sigma) \in V$ for every $k_0\geq 0$ and $K\geq K_0$.
    By Lemma \ref{L:differences-of-rho-neighbouring}, the set $U$ contains the function $\rho_{k_0}$.
Finally, $\rho_{k_0}\in 2 B_{L^\infty}$ by construction.
    Thus, if we show that $U$ is weak$^*$ open, using also the discussion above, we see that for some $\delta>0$, the mollified map $(\psi_\delta*\rho_{k_0})_{|[0,1]^d} \in U\cap V \cap 2B_{L^\infty}$ and therefore we may take for $\rho$ this mollified map.
To show that $U$ is weak$^*$ open, we write
    \begin{equation*}
        U=\bigcap_{(Q,Q')\in \mathcal{Q}^2_{\leq k_0}(\textnormal{adj})}\left\{\rho \in L^\infty([0,1]^d):\left|\int_Q \rho - \int_{Q'}\rho\right| > r^d(1-\varepsilon - \frac{1}{K^{d-1}})\right\}
    \end{equation*}
    By the discussion above, and by continuity of substraction and absolute value, we see that each function
    \begin{equation*}
        \rho \mapsto \left|\int_Q \rho - \int_{Q'}\rho\right|
    \end{equation*}
    is weak$^*$ continuous.
Since $\mathcal{Q}^2_{\leq k_0}(\textnormal{adj})$ is a finite set, $U$ is a finite intersection of weak$^*$ open sets and therefore weak$^*$ open.
\end{proof}

\begin{theorem}\label{T:obtaining-final-rho}
    Let $d\in \nat, d\geq 2$.
    For any $\eta>0$, $\sigma \in 2B_{L^\infty}$ and a weak$^*$ open neighbourhood $V$ of $\sigma$, there exists $N_1\in\nat$ such that if $K\in\nat$, $K\geq N_1$, $M\in\nat$, then for any $k_0\in\nat$, there exists a weak$^*$ open set $U\subset L^\infty([0,1]^d)$ having the following properties.
Whenever $R\in \mathcal{R}_{\leq k_0}$, $n\in\{0, 1, \dots, K-2\}$, $Q=Q(R,n)$ and $Q'=Q(R,n+1)$, then
        \begin{equation*}
        \left|\int_Q \rho - \int_{Q'}\rho\right| > r^d(1-\eta) \quad \text{for every $\rho \in U$,}
    \end{equation*}
    where $r$ is the side length of $Q$ and $Q'$.
Moreover, $U$ contains a smooth map $\rho\in V$ of norm at most $2$.
\end{theorem}
\begin{proof}
Let $K_0\in\nat$ be from Lemma \ref{L:w*-con-subsets} and let
    $N_1\geq K_0$ be such that $\frac{5}{N_1^{d-1}}\leq \frac{1}{2}\eta$ and let $\varepsilon=\frac{1}{2}\eta$.
Let $K\geq N_1$ and $M\in\nat$.
Let $k_0\in \nat$.
    We find $U$ and $\rho$ from Lemma \ref{L:smoothing}.
Then \emph{every} $\rho \in U$ has the required property by Lemma \ref{L:smoothing} and the fact that
    \begin{equation*}
        1- \varepsilon - \frac{1}{N_1^{d-1}}\leq 1- \frac{1}{2} \eta + \frac{1}{2}\eta = 1-\eta.
    \end{equation*}
\end{proof}

We say that a formal Lipschitz sum $\sum_i(f_i, \pi_i)$ \emph{solves} the equation \eqref{E:intro-the-equation1} if the sequence $(f_i, \pi_i)$ solves the equation in the sense of Definition \ref{D:solution}.
Note that the following theorem immediately implies Proposition \ref{P:intro-strong-non-surj}.

\begin{theorem}\label{T:main-counterexample}
    For any $\sigma \in 2B_{L^\infty}$ together with a weak$^*$ open neighbourhood $V$ and for any $S\geq 0$, there exists a smooth map $\rho_0\in 2 B_{L^\infty}\cap V$ and a weak$^*$ open set $U$ containing $\rho_0$ such that any formal Lipschitz sum $\sum_i(f_i,\pi_i)$ solving the PDE \eqref{E:intro-the-equation1} with a right hand side $\rho \in U$ must satisfy
    \begin{equation*}
        s(\sum_i(f_i, \pi_i))>S.
    \end{equation*}
\end{theorem}
\begin{proof}
    If $S=0$, the statement is obvious.
Otherwise, we let
    $\varepsilon=\frac{1}{c_d \sqrt{S}}$ and find $N_0$, $M$ and $k_0$ from Corollary \ref{C:est-on-good-square}.
    Let $\eta=\frac{1}{4}$ and find $N_1\in \nat$ from Theorem \ref{T:obtaining-final-rho}.
Now, let us choose some $K\in\nat$, $K\geq \max\{N_0, N_1\}$ such that also $\frac{1}{K}(\sqrt{d}+1)S\leq \frac{1}{4}$.
    Let $U$ and $\rho$ be from Theorem \ref{T:obtaining-final-rho} and denote $\rho_0=\rho$.
Let $\rho \in U$ be arbitrary and suppose $\sum_i(f_i,\rho_i)$ is a formal Lipschitz sum solving the PDE \eqref{E:intro-the-equation1}.
Due to Proposition \ref{P:regular} and Lemma \ref{L:equivalences}, we may assume, without loss of generality, that $\sum_i(f_i, \pi_i)$ is regular.
For contradiction, we shall assume that
    \begin{equation*}
        s(\sum_i(f_i, \pi_i))\leq S.
    \end{equation*}
    Then by Corollary \ref{C:est-on-good-square}, (adopting the notation for $Q$ and $Q'$ from there),
    \begin{equation*}
        \left|\int_Q \rho - \int_{Q'}\rho\right|\leq r^{d+1}(\sqrt{d}+1)S+r^d c_d \sqrt{S}\varepsilon.
    \end{equation*}
    On the other hand, by Theorem \ref{T:obtaining-final-rho},
    \begin{equation*}
        \left|\int_Q \rho - \int_{Q'}\rho\right| \geq r^d(1-\eta).
    \end{equation*}
    So in combination, we arrive at
    \begin{equation*}
        r^d(1-\eta) \leq r^{d+1}(\sqrt{d}+1)S+r^d c_d \sqrt{S}\varepsilon,
    \end{equation*}
    i.e.
    \begin{equation*}
        (1-\eta)\leq r (\sqrt{d}+1)S + c_d \sqrt{S}\varepsilon
    \end{equation*}
    where we may write $r$ in terms of $K$ and $M$ and obtain
    \begin{equation*}
        (1-\eta) \leq (\frac{1}{K(KM)^k})(\sqrt{d}+1)S+c_d \sqrt{S}\varepsilon,
    \end{equation*}
    where $k\leq k_0$ is the order of $Q$ and $Q'$.
Thus, by our choice of $\varepsilon$, $\eta$ and $K$,
    \begin{equation*}
        \frac{3}{4}\leq \frac{1}{4} + \frac{1}{4},
    \end{equation*}
    a contradiction.
\end{proof}

Let us now recall Definition \ref{D:pot-forms}.
The space $\mathbb{P}^d$ is an appropriately quotiented space of formal Lipschitz sums and the map $\Em\colon \mathbb{P}^d \to L^\infty$ is given by
\begin{equation*}
    \Em([\sum_i(f_i, \pi_i)])=\sum_i f_i \det\Diff\pi_i\in L^\infty.
\end{equation*}
In the language of functional analysis, Theorem \ref{T:main-counterexample} tells us that for any $\sigma \in 2 B_{L^\infty}$ and a weak$^*$ open neighbourhood $V$ of $\sigma$, we find $\rho_0 \in V \cap 2B_{L^\infty}$ together with a weak$^*$ open neighbourhood $U$ such that $U\cap \Em(SB_{\mathbb{P}^d})=\emptyset$.
This immediately implies that
\begin{equation*}
    2 B_{L^\infty}\setminus \Em(SB_{\mathbb{P}^d})\supset\mathcal{U}= \bigcup_{\sigma, V} U_{\sigma, V},
\end{equation*}
where the union is taken over all $\sigma\in 2B_{L^\infty}$ and weak$^*$ open neighbourhoods $V$ of $\sigma$ and $U_{\sigma, V}$ is any open set having the property described for $U$ above.
In particular, the set $\mathcal{U}$ is weak$^*$ dense and open.
It follows by taking complements, that
\begin{equation*}
    2 B_{L^\infty}\cap\Em(SB_{\mathbb{P}^d})
\end{equation*}
can be covered by a weak$^*$ nowhere dense set (meaning intuitively that it is very small).
Thus,
\begin{equation*}
    2B_{L^\infty}\cap \Em(\mathbb{P}^d),
\end{equation*}
may be covered by a countable union of weak$^*$ nowhere dense sets, which by definition makes it a first category set in the sense of Baire.
Since $2B_{L^\infty}$ equipped with the weak$^*$ topology is a completely metrisable space, the Baire category theorem may be applied in this case to assert that
\begin{equation*}
    2 B_{L^\infty}\setminus\Em({\mathbb{P}^d})
\end{equation*}
is weak$^*$ dense inside $2 B_{L^\infty}$.
This is in the spirit of the results of Dymond et al \cite{DKK}.
Since the Baire categorical considerations are a matter of applying a standard theorem, we state the intermediate result in a convenient form.

\begin{theorem}\label{T:nowhere-dense}
    Let $S\geq 0$ and $d\geq 2$ and suppose in the following we are working with spaces and functions defined over $[0,1]^d$.
Then the set
    \begin{equation*}
        \{\sum_if_i \det\Diff\pi_i: \sum_i\max\{\lVert f_i \rVert_\infty, \tLip(f_i)\}\prod_{j=1}^d \tLip(\pi_i^j)\leq S\}\cap B_{L^\infty}
    \end{equation*}
    is weak$^*$ nowhere dense in $B_{L^\infty}$.
\end{theorem}

Theorem \ref{T:intro-reg} follows trivially from the following corollary.

\begin{corollary}\label{C:failure-of-everything}
    The operator $\Em\colon \mathbb{P}^d \to L^\infty$ is strongly non-surjective.
The prescribed Jacobian equation has strong non-regularity.
The Conjecture \ref{Con:Lang-TD} fails, i.e.~there exists a non-flat metric $d$-current supported in $[0,1]^d$.
\end{corollary}
\begin{proof}
    Once we show that the operator $\Em$ is strongly non-surjective, the remaining statements follow from Corollary \ref{C:non-surj0} and Theorem \ref{T:f-a-characterisation}.
In the following we apply a rescaled version of Theorem \ref{T:main-counterexample} for the special case of $\sigma=1$ and $V=L^\infty$.
For every $n\in\nat$, there exists a smooth function $\rho_n \in \frac{1}{n}B_{L^\infty}$ together with a weak$^*$ open set $U\subset L^\infty$, $\rho_n \in U$, such that for any formal Lipschitz sum $\sum_i (f_i, \pi_i)$ with
    \begin{equation*}
        s(\sum_i(f_i, \pi_i))\leq 1,
    \end{equation*}
    it holds that
    \begin{equation*}
        \sum_i f_i \det\Diff \pi_i \not = \rho \quad\text{on a positive measure set for  every $\rho\in U$.}
    \end{equation*}
    In other words,
    \begin{equation*}
        \Em(B_{\mathbb{P}^d})\cap U = \emptyset.
    \end{equation*}
    Since $\rho_n$ is smooth, it holds that $\rho \in \Em(\mathbb{P}^d)$, by taking for example $f_1 = \rho$, $\pi_1 = \Id$ and $f_i=0, \pi_i = 0$ for $i\geq 2$.
Thus, $\rho_n \in \frac{1}{n}B_{L^\infty(Q)}\cap \Em(\mathbb{P}^d)$ and the definition of strong non-surjectivity is thus verified.
\end{proof}

\section{The Flat Chain Conjecture of Lang in all dimensions and codimensions}\label{S:all-dim}
 The purpose of this section is to prove Theorem \ref{T:fcc-fails-all-cases}. We start with the following property of flat chains.

 \begin{theorem}\label{T:flat-structure}
    If $d\in \nat$, $k\in\{1, \dots, d\}$ and $T \in \mathbb{F}_k(B)$ for some ball in $\reals^d$ and $\partial T = 0$, then we may find $S\in\mathbb{F}_{k+1}(B)$ with $\partial S = T$ and
        \begin{equation*}
            \lVert S \rVert <\infty.
        \end{equation*}
 \end{theorem}
 \begin{proof}
    Using \cite[p. 382]{Federer} we may find $X\in \mathbb{F}_k(B)$ and $Y\in \mathbb{F}_{k+1}(B)$ such that both $X$ and $Y$ are of finite mass and
    \begin{equation*}
        T = X + \partial Y.
    \end{equation*}
    By taking the boundary of both sides, we obtain $0= \partial X$ so that $X$ is a normal current without boundary. Therefore, there exists some $R\in \mathbb{F}_{k+1}(B)$ having finite mass and satisfying $\partial R = X$. Indeed, this follows from the well-known ``cone construction'', details of which are available in \cite[4.1.11]{Federer} or alternatively in the metric setting in \cite[Definition 10.1 and Propositon 10.2]{AKcur}. It follows that
    \begin{equation*}
        T = \partial R + \partial Y = \partial (R+Y),
    \end{equation*} 
    hence, the statement follows by taking $S = R+ Y$.
 \end{proof}

We are now prepared to start considering the various cases of Theorem \ref{T:fcc-fails-all-cases}.
We begin with the negative direction for $k\geq 2$, which is obtained as an immediate consequence of the previous sections.

\begin{corollary}\label{C:k-geq2}
     Suppose $d\in\nat$ and $k\in\nat\cup \{0\}$, $2\leq k\leq d$.
Then there exists a compactly supported metric $k$-current $T$ in $\reals^d$ that is not a flat chain.
\end{corollary}
\begin{proof}
    By Theorem \ref{T:main-counterexample}, there exists a metric $k$-current $T$ in $[0,1]^k$ that is not a flat chain.
Let $I\colon \reals^k \to \reals^d$ be any isometry.
The current $I_\# C_k(T)$ is not a flat chain by Lemma \ref{L:flat-injection}.
By Lemma \ref{L:commutativity-pw}, $I_\# C_k(T)$ is a metric current (i.e.~it lies in the range of $C_k$).
We produced a compactly supported metric $k$-current in $\reals^d$ which is not a flat chain, concluding the proof.
\end{proof}

We continue by disproving the conjecture for $1$-currents in $\reals^d$ for $d\geq 2$.

\begin{lemma}\label{L:1-current-case}
    Suppose $T$ is a compactly supported metric $2$-current in $\reals^2$ which is not flat.
Then $\partial T$ is not flat.
\end{lemma}
\begin{proof}
    Suppose $\partial T$ is a flat chain.
Then by Theorem \ref{T:flat-structure} we may write $\partial C_2 T=\partial S$, where $S$ is a flat $2$-chain of finite mass and with compact support.
Now $C_2 T - S$ is a compactly supported current without boundary, therefore there is a $(d+1)$-current $X$ such that $\partial X = C_2 T -S$.
The only $(d+1)$-current is $0$, so $C_2 T = S$.
This contradicts the fact that $T$ is non-flat and $S$ is flat.
\end{proof}

\begin{corollary}\label{C:case-k=1}
    If $d\in\nat$, $d\geq 2$, then there exists a metric $1$-current in $\reals^d$ without boundary, which is not a flat chain.
\end{corollary}
\begin{proof}
    By Lemma \ref{L:1-current-case}, upon taking any non-flat metric $2$-current $T$ in $\reals^2$ with compact support $\partial T$ is a non-flat metric $1$-current. Moreover, such a non-flat metric $2$-current $T$ exists by Corollary \ref{C:k-geq2}.
If $d=2$, we are therefore finished.
If $d>2$, we may consider an arbitrary isometry $I\colon \reals^2 \to \reals^d$ and the current $I_\# \partial T$ will have the required properties.
\end{proof}

Now we observe what happens for $1$-currents in $\reals$.

\begin{theorem}\label{T:k=1=d}
    Every compactly supported metric $1$-current in $\reals$ is a flat chain.
\end{theorem}
\begin{proof}
    It suffices to show that every metric $1$-current $T$ supported in $[0,1]$ is of finite mass.
The statement then follows by elementary considerations, which the reader can carry out as an exercise.
Alternatively, one can obtain the statement immediately by applying either the results of
    DePhilippis and Rindler \cite{DeRindler} or Schioppa \cite{Sch}.

    Thus, we continue by showing finite mass.
    Let us denote $\diff x_1 = \diff t$ in this case, as the underlying space is $\reals$.
By the basic properties of the comparison map \cite[Theorem 5.5]{UL}, it suffices to show that $C_1(T)$ is of finite mass.
By the Riesz representation theorem, it suffices to show that there is some constant $C_T\in(0, \infty)$ such that
    \begin{equation*}
        |C_1T(\omega \diff t)|\leq C_T \lVert \omega \rVert_\infty \quad \text{for every smooth $\omega\colon \reals\to\reals$.}
    \end{equation*}
    We fix such an $\omega$ and define
    \begin{equation*}
        W(t)=\int_0^t \omega(s)\diff s \quad\text{for $t\in \reals$.}
    \end{equation*}
    Then $W$ is a smooth map with $\tLip(W)=\lVert \omega \rVert_\infty$.
We observe that $W'\diff t=\diff W$ and therefore
    \begin{equation*}
        T(1,W)=C_1(T)(1 \diff W)=C_1(T)(\omega\diff t).
    \end{equation*}
    But then by taking for $C_T$ the norm of $T$ in $\mathbb{P}_1$, we have
    \begin{equation*}
        |C_1(T)(\omega\diff t)|=|T(1,W)|\leq C_T \tLip(W)=C_T \lVert \omega \rVert_\infty.
    \end{equation*}
\end{proof}

We have now proven all of the cases of Theorem \ref{T:fcc-fails-all-cases} except for the case of $0$-currents.
This statement is essentially just Corollary \ref{C:density-of-normal-in-D0}, but we give a detailed explanation.

\begin{theorem}\label{T:case-k=0}
    Every metric $0$-current in $\reals^d$ with compact support is a flat chain.
\end{theorem}
\begin{proof}
    Let $K$ be a compact set containing the support of a metric $0$-current $T$.
The flat norm of a $0$-form $\omega$ (which is merely a real-valued function) in $K$ is equal to
    \begin{equation*}
        \lVert \omega \rVert_{\mathbb{F}^0(K)}=\max\{\lVert \omega \rVert_{\infty}, \tLip(\omega)\},
    \end{equation*}
    where the supremum and the Lipschitz constant are taken on $K$.
    By Corollary \ref{C:density-of-normal-in-D0}, we find a sequence of measures $\mu_i$, such that
    \begin{equation*}
        \sup_{\pi} |T(\pi)-\int \pi \;\diff\mu_i|\to 0,
    \end{equation*}
    where the supremum is taken over Lipschitz functions $\pi$ on $K$ with
    \begin{equation*}
        \max\{\lVert \pi \rVert_\infty, \tLip(\pi)\}\leq 1.
    \end{equation*}
    If we let $N_i$ be the $0$-current given by for
    \begin{equation*}
    N_i(\omega)=\int \omega \;\diff \mu_i\quad \text{for}\;
        \omega\in C^\infty_c(\reals^d, \Wedge^0 \reals^d)=C^\infty_c(\reals^d, \reals),
    \end{equation*}
    then $N_i$ is of finite mass and therefore normal.
    It follows that
    \begin{equation*}
        \sup_{\lVert\omega\rVert_{\mathbb{F}^0}\leq 1} |T(\omega)-N_i(\omega)|\to 0,
    \end{equation*}
    i.e.~$N_i$ converge to $T$ in the flat norm.
\end{proof}

\begin{proof}[Proof of Theorem \ref{T:fcc-fails-all-cases}]
        The case $k=0$ is Theorem \ref{T:case-k=0}, the case $k=1$ and $d\geq 2$ is a special case of Corollary \ref{C:case-k=1}, the case $d\geq k\geq2$ is Corollary \ref{C:k-geq2}.
Finally, the case $k=d=1$ is Theorem \ref{T:k=1=d}.
\end{proof}

\section{The general case of the lack of regularity for the exterior derivative equation}\label{S:prescribed-exterior}

The main result of this section is that the prescribed exterior derivative equation lacks Lipschitz regularity.
This appears to only be explicitly available in the literature in the case when the exterior derivative coincides with divergence \cite{PreissDiv, McM, BB}.
This is discussed in more depth in the introductory Section \ref{S:intro}, Subsection \ref{SS:intro-reg}.
Here we only reitarate, that the dualisation methods employed by Bourgain and Brezis \cite{BB} and McMullen \cite{McM} together with the non-inequality of Ornstein \cite{Orn} could likely be adjusted to yield similar results for the general exterior derivative operator.
Given that this is not available in the literature and that these results can be immediately obtained from our main non-regularity result, we provide here the full statement with rigorous proof.

We denote by $L_{\textnormal{cl}}^\infty([0,1]^d, \Wedge^k\reals_d)$ the subspace of distributionally closed $L^\infty$ differential $k$-forms, i.e.~those forms $\omega$ for which $\diff \omega = 0$ distributionally.
It is easily verified, that this subspace is weak$^*$ closed in $L^\infty([0,1]^d, \Wedge^k\reals_d)$.

\begin{theorem}
    Consider the equation for differential $k$-forms in $\reals^d$:
    \begin{equation}\label{E:ext-diff}
        \diff \omega = \eta.
    \end{equation}
    If $k\in\{1,2, \dots, d-1\}$, then there exists $\eta\in L_{\textnormal{cl}}^\infty(\reals^d, \Wedge^k\reals_d)$, supported inside $Q=[0,1]^d$, such that no Lipschitz $\omega$ solving \eqref{E:ext-diff} a.e.~in $Q$ exists.
\end{theorem}
\begin{proof}
    Suppose first that $\eta$ is of the form
    \begin{equation}\label{E:eta-def}
        \eta=\eta_1 \diff x_1 \wedge \dots \wedge \diff x_{k+1}
    \end{equation}
    and observe that since $k+1\leq d$, if $\eta_1$ is not identically zero, then neither is $\eta$.

    Let $\omega$ be a Lipschitz map solving \eqref{E:ext-diff} a.e.
For every $I=(i_1, \dots, i_k)$ where $i_1< \dots < i_k$, $i_1, \dots, i_k\in\{1, \dots, d\}$, we find a Lipschitz real-valued map $\omega_I$ such that
    \begin{equation*}
        \omega=\sum_I \omega_I \diff x^I=\sum_I \omega_I \diff x_{i_1}\wedge \dots \wedge \diff x_{i_k}.
    \end{equation*}
    For any $j=\{1, \dots, k+1\}$, let
    \begin{equation*}
        I^j=(1, \dots, j-1, j+1, \dots, k+1).
    \end{equation*}
    Then, by considering the projection of $\diff \omega$ to the basis vector $\diff x_1 \wedge\dots \wedge \diff x_{k+1}$,
    the equality \eqref{E:ext-diff} in particular implies
    \begin{equation}\label{E:from-ext-to=det}
        \sum_{j=1}^{k+1} (-1)^{j-1} \frac{\partial \omega_{I^j}}{\partial e_j}=\eta_1.
    \end{equation}
    Finally, since the equality holds almost everywhere, by Fubini's theorem, it also holds $\mathcal{L}^{k+1}$-a.e.~on at least one (in fact almost every) section of the cube $[0,1]^d$ by the plane spanned by $e_1, \dots , e_{k+1}$.
We may assume, without loss of generality, that it holds on $[0,1]^{k+1}\subset \reals^{k+1} \subset \reals^d$.
Let
    \begin{equation*}
        \pi_j^j=(-1)^{j-1}{\omega_{I^j}}_{|[0,1]^{k+1}} \quad \text{for $j\in\{1, \dots, k+1\}$}
    \end{equation*}
    and
    \begin{equation*}
        \pi_i^j= x_j \quad \text{whenever $i\not=j$, $i,j\in\{1, \dots, k+1\}$.}
    \end{equation*}

    It is now immediately verified that the definition of $\pi_i^j$ together with \eqref{E:from-ext-to=det} imply
    \begin{equation}\label{E:pde-ext-to-det}
        \sum_{i=1}^{k+1} \det\Diff \pi_i = \eta_1 \quad \text{a.e.~in $[0,1]^{k+1}$.}
    \end{equation}
    Since $k+1\geq 2$, using Corollary \ref{C:failure-of-everything}, we may find $\eta_1 \in L^\infty([0,1]^{k+1})$ with $\lVert \eta_1\rVert_\infty\leq 1$ such that no sequence of Lipschitz maps $\pi_i$ satisfies \eqref{E:pde-ext-to-det}.
We extend $\eta_1$ to $[0,1]^d$ via
    \begin{equation*}
        \eta_1(x_1, \dots, x_d)=\eta_1(x_1, \dots, x_{k+1}).
    \end{equation*}
    Finally, we let $\eta$ by given by \eqref{E:eta-def}.
By the argument above, the proof is finished once we show that $\eta$ is distributionally closed.

    The argument for this is completely standard, but for the sake of completeness we provide it.
Firstly, via a mollification argument, we may find a sequence of smooth functions satisfying $\eta_1^i \to \eta_1$ weak$^*$ in $L^\infty([0,1]^{k+1})$.
Once again, we extend
    \begin{equation*}
        \eta_1^i(x_1, \dots ,x_d) = \eta^i_1(x_1, \dots, x_{k+1}).
    \end{equation*}
    Finally, we let
    \begin{equation*}
        \eta^i = \eta_1^i \diff x_1 \wedge \dots \wedge \diff x_{k+1}.
    \end{equation*}
    We observe that the weak$^*$ convergence of $\eta_1^i$ implies also the weak$^*$ convergence $\eta^i \to \eta$ inside
    \begin{equation*}
        L^\infty([0,1]^d, \Wedge^{k+1} \reals^d).
    \end{equation*}
    Since the space $L^\infty_{\operatorname{cl}}([0,1]^d, \Wedge^{k+1} \reals^d)$
     is weak$^*$ closed, it suffices to show that each $\eta^i$ is closed.
This follows from the definition of the exterior derivative and the fact that
     \begin{equation*}
         \frac{\partial}{\partial e_{j}}\eta^i=0
     \end{equation*}
     for all $j>k+1$, which is a consequence of the definition.
\end{proof}

\newpage
\bibliographystyle{abbrvnat}
\renewcommand{\bibname}{Bibliography}
\bibliography{bibliography}

@article {B,
    AUTHOR = {Bate, David},
     TITLE = {Purely unrectifiable metric spaces and perturbations of
              {L}ipschitz functions},
   JOURNAL = {Acta Math.},
  FJOURNAL = {Acta Mathematica},
    VOLUME = {224},
      YEAR = {2020},
    NUMBER = {1},
     PAGES = {1--65},
      ISSN = {0001-5962},
   MRCLASS = {31E05 (26A16 28A75 49Q15)},
  MRNUMBER = {4086714},
MRREVIEWER = {Jeremy T. Tyson},
       DOI = {10.4310/acta.2020.v224.n1.a1},
       URL = {https://doi.org/10.4310/acta.2020.v224.n1.a1},
}

@book {Daco,
    AUTHOR = {Dacorogna, Bernard},
     TITLE = {Direct methods in the calculus of variations},
    SERIES = {Applied Mathematical Sciences},
    VOLUME = {78},
   EDITION = {Second},
 PUBLISHER = {Springer, New York},
      YEAR = {2008},
     PAGES = {xii+619},
      ISBN = {978-0-387-35779-9},
   MRCLASS = {49-02 (49J10 49J45 74B20)},
  MRNUMBER = {2361288},
MRREVIEWER = {Pietro Celada},
}

@book {Federer,
    AUTHOR = {Federer, Herbert},
     TITLE = {Geometric measure theory},
    SERIES = {Die Grundlehren der mathematischen Wissenschaften, Band 153},
 PUBLISHER = {Springer-Verlag New York, Inc., New York},
      YEAR = {1969},
     PAGES = {xiv+676},
   MRCLASS = {28.80 (26.00)},
  MRNUMBER = {0257325},
MRREVIEWER = {J. E. Brothers},
}

@article {AKcur,
    AUTHOR = {Ambrosio, Luigi and Kirchheim, Bernd},
     TITLE = {Currents in metric spaces},
   JOURNAL = {Acta Math.},
  FJOURNAL = {Acta Mathematica},
    VOLUME = {185},
      YEAR = {2000},
    NUMBER = {1},
     PAGES = {1--80},
      ISSN = {0001-5962},
   MRCLASS = {49Q15 (49Q20)},
  MRNUMBER = {1794185},
MRREVIEWER = {Giovanni Bellettini},
       DOI = {10.1007/BF02392711},
       URL = {https://doi.org/10.1007/BF02392711},
}

@book {Weaver,
    AUTHOR = {Weaver, Nik},
     TITLE = {Lipschitz algebras},
   EDITION = {Second},
 PUBLISHER = {World Scientific Publishing Co. Pte. Ltd., Hackensack, NJ},
      YEAR = {2018},
     PAGES = {xiv+458},
      ISBN = {978-981-4740-63-0},
   MRCLASS = {46-02 (26A16 46Bxx 46Exx 46H05 46J10)},
  MRNUMBER = {3792558},
MRREVIEWER = {Antonio Jim\'{e}nez-Vargas},
}

@book {Conway,
    AUTHOR = {Conway, John B.},
     TITLE = {A course in functional analysis},
    SERIES = {Graduate Texts in Mathematics},
    VOLUME = {96},
   EDITION = {Second},
 PUBLISHER = {Springer-Verlag, New York},
      YEAR = {1990},
     PAGES = {xvi+399},
      ISBN = {0-387-97245-5},
   MRCLASS = {46-01 (47-01)},
  MRNUMBER = {1070713},
}

@article {UL,
    AUTHOR = {Lang, Urs},
     TITLE = {Local currents in metric spaces},
   JOURNAL = {J. Geom. Anal.},
  FJOURNAL = {Journal of Geometric Analysis},
    VOLUME = {21},
      YEAR = {2011},
    NUMBER = {3},
     PAGES = {683--742},
      ISSN = {1050-6926,1559-002X},
   MRCLASS = {49Q15},
  MRNUMBER = {2810849},
       DOI = {10.1007/s12220-010-9164-x},
       URL = {https://doi.org/10.1007/s12220-010-9164-x},
}

@book {Heinonen,
    AUTHOR = {Heinonen, Juha},
     TITLE = {Lectures on {L}ipschitz analysis},
    SERIES = {Report. University of Jyv\"askyl\"a{} Department of
              Mathematics and Statistics},
    VOLUME = {100},
 PUBLISHER = {University of Jyv\"askyl\"a, Jyv\"askyl\"a},
      YEAR = {2005},
     PAGES = {ii+77},
      ISBN = {951-39-2318-5},
   MRCLASS = {49Q15 (26B35 30C35 46E35 49J52 52A01 54E40)},
  MRNUMBER = {2177410},
MRREVIEWER = {Jan\ Mal\'y},
}

@article {PS,
    AUTHOR = {Pankka, Pekka and Soultanis, Elefterios},
     TITLE = {Metric currents and polylipschitz forms},
   JOURNAL = {Calc. Var. Partial Differential Equations},
  FJOURNAL = {Calculus of Variations and Partial Differential Equations},
    VOLUME = {59},
      YEAR = {2020},
    NUMBER = {2},
     PAGES = {Paper No. 76, 38},
      ISSN = {0944-2669,1432-0835},
   MRCLASS = {49Q15 (30L10 53C23)},
  MRNUMBER = {4083197},
MRREVIEWER = {Alexander\ O.\ Ivanov},
       DOI = {10.1007/s00526-020-01741-5},
       URL = {https://doi.org/10.1007/s00526-020-01741-5},
}

@article {DeRindler,
    AUTHOR = {De Philippis, Guido and Rindler, Filip},
     TITLE = {On the structure of {$\mathscr{A}$}-free measures and applications},
   JOURNAL = {Ann. of Math. (2)},
  FJOURNAL = {Annals of Mathematics. Second Series},
    VOLUME = {184},
      YEAR = {2016},
    NUMBER = {3},
     PAGES = {1017--1039},
      ISSN = {0003-486X,1939-8980},
   MRCLASS = {49Q20 (28B20 46A22)},
  MRNUMBER = {3549629},
MRREVIEWER = {Pei\ Biao\ Zhao},
       DOI = {10.4007/annals.2016.184.3.10},
       URL = {https://doi.org/10.4007/annals.2016.184.3.10},
}

@article {Sch,
    AUTHOR = {Schioppa, Andrea},
     TITLE = {Metric currents and {A}lberti representations},
   JOURNAL = {J. Funct. Anal.},
  FJOURNAL = {Journal of Functional Analysis},
    VOLUME = {271},
      YEAR = {2016},
    NUMBER = {11},
     PAGES = {3007--3081},
      ISSN = {0022-1236,1096-0783},
   MRCLASS = {53C23 (49Q15)},
  MRNUMBER = {3554700},
MRREVIEWER = {Enrico\ Le Donne},
       DOI = {10.1016/j.jfa.2016.08.022},
       URL = {https://doi.org/10.1016/j.jfa.2016.08.022},
}

@article {FF,
    AUTHOR = {Federer, Herbert and Fleming, Wendell H.},
     TITLE = {Normal and integral currents},
   JOURNAL = {Ann. of Math. (2)},
  FJOURNAL = {Annals of Mathematics. Second Series},
    VOLUME = {72},
      YEAR = {1960},
     PAGES = {458--520},
      ISSN = {0003-486X},
   MRCLASS = {28.80 (53.45)},
  MRNUMBER = {123260},
MRREVIEWER = {L.\ C.\ Young},
       DOI = {10.2307/1970227},
       URL = {https://doi.org/10.2307/1970227},
}

@article {DKK,
    AUTHOR = {Dymond, Michael and Kalu\v{z}a, Vojt\v{e}ch and Kopeck\'a, Eva},
     TITLE = {Mapping {$n$} grid points onto a square forces an arbitrarily
              large {L}ipschitz constant},
   JOURNAL = {Geom. Funct. Anal.},
  FJOURNAL = {Geometric and Functional Analysis},
    VOLUME = {28},
      YEAR = {2018},
    NUMBER = {3},
     PAGES = {589--644},
      ISSN = {1016-443X,1420-8970},
   MRCLASS = {28A75 (26A16 28A33)},
  MRNUMBER = {3816520},
       DOI = {10.1007/s00039-018-0445-z},
       URL = {https://doi.org/10.1007/s00039-018-0445-z},
}

@article {BK,
    AUTHOR = {Burago, D. and Kleiner, B.},
     TITLE = {Separated nets in {E}uclidean space and {J}acobians of
              bi-{L}ipschitz maps},
   JOURNAL = {Geom. Funct. Anal.},
  FJOURNAL = {Geometric and Functional Analysis},
    VOLUME = {8},
      YEAR = {1998},
    NUMBER = {2},
     PAGES = {273--282},
      ISSN = {1016-443X,1420-8970},
   MRCLASS = {26B35 (26B12 54E99)},
  MRNUMBER = {1616135},
MRREVIEWER = {Maciej\ Sablik},
       DOI = {10.1007/s000390050056},
       URL = {https://doi.org/10.1007/s000390050056},
}

@article {Ye,
    AUTHOR = {Ye, Dong},
     TITLE = {Prescribing the {J}acobian determinant in {S}obolev spaces},
   JOURNAL = {Ann. Inst. H. Poincar\'e{} C Anal. Non Lin\'eaire},
  FJOURNAL = {Annales de l'Institut Henri Poincar\'e{} C. Analyse Non
              Lin\'eaire},
    VOLUME = {11},
      YEAR = {1994},
    NUMBER = {3},
     PAGES = {275--296},
      ISSN = {0294-1449,1873-1430},
   MRCLASS = {35J60 (46E35 58C35)},
  MRNUMBER = {1277896},
MRREVIEWER = {Jie\ Yang},
       DOI = {10.1016/S0294-1449(16)30185-8},
       URL = {https://doi.org/10.1016/S0294-1449(16)30185-8},
}

@article {DM,
    AUTHOR = {Dacorogna, Bernard and Moser, J\"urgen},
     TITLE = {On a partial differential equation involving the {J}acobian
              determinant},
   JOURNAL = {Ann. Inst. H. Poincar\'e{} C Anal. Non Lin\'eaire},
  FJOURNAL = {Annales de l'Institut Henri Poincar\'e{} C. Analyse Non
              Lin\'eaire},
    VOLUME = {7},
      YEAR = {1990},
    NUMBER = {1},
     PAGES = {1--26},
      ISSN = {0294-1449,1873-1430},
   MRCLASS = {58G20 (35J60)},
  MRNUMBER = {1046081},
MRREVIEWER = {Jan\ Rogulski},
       DOI = {10.1016/S0294-1449(16)30307-9},
       URL = {https://doi.org/10.1016/S0294-1449(16)30307-9},
}

@article {McM,
    AUTHOR = {McMullen, C. T.},
     TITLE = {Lipschitz maps and nets in {E}uclidean space},
   JOURNAL = {Geom. Funct. Anal.},
  FJOURNAL = {Geometric and Functional Analysis},
    VOLUME = {8},
      YEAR = {1998},
    NUMBER = {2},
     PAGES = {304--314},
      ISSN = {1016-443X,1420-8970},
   MRCLASS = {58C07 (26B35)},
  MRNUMBER = {1616159},
MRREVIEWER = {Jie\ Yang},
       DOI = {10.1007/s000390050058},
       URL = {https://doi.org/10.1007/s000390050058},
}

@article {PreissDiv,
    AUTHOR = {Preiss, David},
     TITLE = {Additional regularity for {L}ipschitz solutions of {PDE}},
   JOURNAL = {J. Reine Angew. Math.},
  FJOURNAL = {Journal f\"ur die Reine und Angewandte Mathematik. [Crelle's
              Journal]},
    VOLUME = {485},
      YEAR = {1997},
     PAGES = {197--207},
      ISSN = {0075-4102,1435-5345},
   MRCLASS = {35B65 (35D10 35J15)},
  MRNUMBER = {1442194},
       DOI = {10.1515/crll.1997.485.197},
       URL = {https://doi.org/10.1515/crll.1997.485.197},
}

@article {Orn,
    AUTHOR = {Ornstein, Donald},
     TITLE = {A non-equality for differential operators in the {$L\sb{1}$}
              norm},
   JOURNAL = {Arch. Rational Mech. Anal.},
  FJOURNAL = {Archive for Rational Mechanics and Analysis},
    VOLUME = {11},
      YEAR = {1962},
     PAGES = {40--49},
      ISSN = {0003-9527},
   MRCLASS = {47.65},
  MRNUMBER = {149331},
MRREVIEWER = {H.\ Mirkil},
       DOI = {10.1007/BF00253928},
       URL = {https://doi.org/10.1007/BF00253928},
}

@book {ABS,
    AUTHOR = {Ambrosio, Luigi and Bru\'e, Elia and Semola, Daniele},
     TITLE = {Lectures on optimal transport},
    SERIES = {Unitext},
    VOLUME = {130},
      NOTE = {La Matematica per il 3+2},
 PUBLISHER = {Springer, Cham},
      YEAR = {2021},
     PAGES = {ix+250},
      ISBN = {978-3-030-72161-9; 978-3-030-72162-6},
   MRCLASS = {49-01 (49Q22)},
  MRNUMBER = {4294651},
MRREVIEWER = {Hugo\ Lavenant},
       DOI = {10.1007/978-3-030-72162-6},
       URL = {https://doi.org/10.1007/978-3-030-72162-6},
}

@article {RiYe,
    AUTHOR = {Rivi\`ere, Tristan and Ye, Dong},
     TITLE = {Resolutions of the prescribed volume form equation},
   JOURNAL = {NoDEA Nonlinear Differential Equations Appl.},
  FJOURNAL = {NoDEA. Nonlinear Differential Equations and Applications},
    VOLUME = {3},
      YEAR = {1996},
    NUMBER = {3},
     PAGES = {323--369},
      ISSN = {1021-9722,1420-9004},
   MRCLASS = {35J60 (35B65)},
  MRNUMBER = {1404587},
MRREVIEWER = {Jes\'us\ Hern\'andez},
       DOI = {10.1007/BF01194070},
       URL = {https://doi.org/10.1007/BF01194070},
}

@book {AFP,
    AUTHOR = {Ambrosio, Luigi and Fusco, Nicola and Pallara, Diego},
     TITLE = {Functions of bounded variation and free discontinuity
              problems},
    SERIES = {Oxford Mathematical Monographs},
 PUBLISHER = {The Clarendon Press, Oxford University Press, New York},
      YEAR = {2000},
     PAGES = {xviii+434},
      ISBN = {0-19-850245-1},
   MRCLASS = {49-02 (49J45 49K10 49Qxx)},
  MRNUMBER = {1857292},
MRREVIEWER = {J.\ E.\ Brothers},
}

@article {dLM,
    AUTHOR = {de Leeuw, Karel and Mirkil, Hazleton},
     TITLE = {Majorations dans {L{$\sb{\infty }$}} des op\'erateurs
              diff\'erentiels \`a{} coefficients constants},
   JOURNAL = {C. R. Acad. Sci. Paris},
  FJOURNAL = {Comptes Rendus Hebdomadaires des S\'eances de l'Acad\'emie des
              Sciences},
    VOLUME = {254},
      YEAR = {1962},
     PAGES = {2286--2288},
      ISSN = {0001-4036},
   MRCLASS = {47.65 (34.99)},
  MRNUMBER = {139964},
MRREVIEWER = {H.\ G.\ Garnir},
}

@article {BB,
    AUTHOR = {Bourgain, Jean and Brezis, Ha\"im},
     TITLE = {On the equation {${\rm div}\, Y=f$} and application to control
              of phases},
   JOURNAL = {J. Amer. Math. Soc.},
  FJOURNAL = {Journal of the American Mathematical Society},
    VOLUME = {16},
      YEAR = {2003},
    NUMBER = {2},
     PAGES = {393--426},
      ISSN = {0894-0347,1088-6834},
   MRCLASS = {35F05 (35B10 35F15 42B05 47N20)},
  MRNUMBER = {1949165},
MRREVIEWER = {Florin\ Iacob},
       DOI = {10.1090/S0894-0347-02-00411-3},
       URL = {https://doi.org/10.1090/S0894-0347-02-00411-3},
}

@article {Horgan,
    AUTHOR = {Horgan, C. O.},
     TITLE = {Korn's inequalities and their applications in continuum
              mechanics},
   JOURNAL = {SIAM Rev.},
  FJOURNAL = {SIAM Review. A Publication of the Society for Industrial and
              Applied Mathematics},
    VOLUME = {37},
      YEAR = {1995},
    NUMBER = {4},
     PAGES = {491--511},
      ISSN = {0036-1445},
   MRCLASS = {73C02 (35Q72 76D05)},
  MRNUMBER = {1368384},
       DOI = {10.1137/1037123},
       URL = {https://doi.org/10.1137/1037123},
}

@article {KK,
    AUTHOR = {Kirchheim, Bernd and Kristensen, Jan},
     TITLE = {On rank one convex functions that are homogeneous of degree
              one},
   JOURNAL = {Arch. Ration. Mech. Anal.},
  FJOURNAL = {Archive for Rational Mechanics and Analysis},
    VOLUME = {221},
      YEAR = {2016},
    NUMBER = {1},
     PAGES = {527--558},
      ISSN = {0003-9527,1432-0673},
   MRCLASS = {49Q20 (26B25 46G05 46N10 49J10 49J45)},
  MRNUMBER = {3483901},
MRREVIEWER = {Bianca\ Stroffolini},
       DOI = {10.1007/s00205-016-0967-1},
       URL = {https://doi.org/10.1007/s00205-016-0967-1},
}

@misc{BKV,
      title={New applications of Hadamard-in-the-mean inequalities to incompressible variational problems}, 
      author={Jonathan Bevan and Martin Kružík and Jan Valdman},
      year={2024},
      eprint={2412.18467},
      archivePrefix={arXiv},
      primaryClass={math.AP},
      url={https://arxiv.org/abs/2412.18467}, 
}

@article {WL,
    AUTHOR = {Davis, William J. and Lindenstrauss, Joram},
     TITLE = {On total nonnorming subspaces},
   JOURNAL = {Proc. Amer. Math. Soc.},
  FJOURNAL = {Proceedings of the American Mathematical Society},
    VOLUME = {31},
      YEAR = {1972},
     PAGES = {109--111},
      ISSN = {0002-9939,1088-6826},
   MRCLASS = {46.10},
  MRNUMBER = {288560},
MRREVIEWER = {A.\ Pe\l czy\'nski},
       DOI = {10.2307/2038522},
       URL = {https://doi.org/10.2307/2038522},
}

@article {BWY,
    AUTHOR = {Basso, Giuliano and Wenger, Stefan and Young, Robert},
     TITLE = {Undistorted fillings in subsets of metric spaces},
   JOURNAL = {Adv. Math.},
  FJOURNAL = {Advances in Mathematics},
    VOLUME = {423},
      YEAR = {2023},
     PAGES = {Paper No. 109024, 54},
      ISSN = {0001-8708,1090-2082},
   MRCLASS = {53C23 (49Q15)},
  MRNUMBER = {4577269},
MRREVIEWER = {Debanjan\ Nandi},
       DOI = {10.1016/j.aim.2023.109024},
       URL = {https://doi.org/10.1016/j.aim.2023.109024},
}

@article {KRW,
    AUTHOR = {Koumatos, Konstantinos and Rindler, Filip and Wiedemann, Emil},
     TITLE = {Differential inclusions and {Y}oung measures involving
              prescribed {J}acobians},
   JOURNAL = {SIAM J. Math. Anal.},
  FJOURNAL = {SIAM Journal on Mathematical Analysis},
    VOLUME = {47},
      YEAR = {2015},
    NUMBER = {2},
     PAGES = {1169--1195},
      ISSN = {0036-1410,1095-7154},
   MRCLASS = {49J45 (28B05 46G10 49J10 74B20)},
  MRNUMBER = {3325758},
MRREVIEWER = {Liviu\ Constantin\ Florescu},
       DOI = {10.1137/140968860},
       URL = {https://doi.org/10.1137/140968860},
}

@incollection {CLMS,
    AUTHOR = {Coifman, R. and Lions, P.-L. and Meyer, Y. and Semmes, S.},
     TITLE = {Compacit\'e{} par compensation et espaces de {H}ardy},
 BOOKTITLE = {S\'eminaire sur les \'Equations aux {D}\'eriv\'ees
              {P}artielles, 1989--1990},
     PAGES = {Exp. No. XIV, 10},
 PUBLISHER = {\'Ecole Polytech., Palaiseau},
      YEAR = {1990},
      ISBN = {2-73-020211-0},
   MRCLASS = {35A25 (42B30 46E15)},
  MRNUMBER = {1073189},
MRREVIEWER = {Nicolas\ Lerner},
}

@article {Standa,
    AUTHOR = {Hencl, Stanislav},
     TITLE = {Sobolev homeomorphism with zero {J}acobian almost everywhere},
   JOURNAL = {J. Math. Pures Appl. (9)},
  FJOURNAL = {Journal de Math\'ematiques Pures et Appliqu\'ees. Neuvi\`eme
              S\'erie},
    VOLUME = {95},
      YEAR = {2011},
    NUMBER = {4},
     PAGES = {444--458},
      ISSN = {0021-7824,1776-3371},
   MRCLASS = {46E35 (26B35)},
  MRNUMBER = {2776377},
MRREVIEWER = {Luigi\ Greco},
       DOI = {10.1016/j.matpur.2010.10.012},
       URL = {https://doi.org/10.1016/j.matpur.2010.10.012},
}

@article {SoWe,
    AUTHOR = {Sormani, Christina and Wenger, Stefan},
     TITLE = {The intrinsic flat distance between {R}iemannian manifolds and
              other integral current spaces},
   JOURNAL = {J. Differential Geom.},
  FJOURNAL = {Journal of Differential Geometry},
    VOLUME = {87},
      YEAR = {2011},
    NUMBER = {1},
     PAGES = {117--199},
      ISSN = {0022-040X,1945-743X},
   MRCLASS = {53C23 (49Q15)},
  MRNUMBER = {2786592},
MRREVIEWER = {Luca\ Granieri},
       URL = {http://projecteuclid.org/euclid.jdg/1303219774},
}

@article {Allen,
    AUTHOR = {Allen, Brian},
     TITLE = {Inverse mean curvature flow and the stability of the positive
              mass theorem},
   JOURNAL = {Comm. Anal. Geom.},
  FJOURNAL = {Communications in Analysis and Geometry},
    VOLUME = {31},
      YEAR = {2023},
    NUMBER = {10},
     PAGES = {2413--2470},
      ISSN = {1019-8385,1944-9992},
   MRCLASS = {53E10 (53C20)},
  MRNUMBER = {4785593},
MRREVIEWER = {Huijuan\ Wang},
       DOI = {10.4310/cag.2023.v31.n10.a4},
       URL = {https://doi.org/10.4310/cag.2023.v31.n10.a4},
}

@article {PaSteFlows,
    AUTHOR = {Paolini, Emanuele and Stepanov, Eugene},
     TITLE = {Flows of measures generated by vector fields},
   JOURNAL = {Proc. Roy. Soc. Edinburgh Sect. A},
  FJOURNAL = {Proceedings of the Royal Society of Edinburgh. Section A.
              Mathematics},
    VOLUME = {148},
      YEAR = {2018},
    NUMBER = {4},
     PAGES = {773--818},
      ISSN = {0308-2105,1473-7124},
   MRCLASS = {35R05 (35F05 35Q35 35R06 35R15 49Q15)},
  MRNUMBER = {3841499},
MRREVIEWER = {Rodica\ Luca},
       DOI = {10.1017/S0308210517000312},
       URL = {https://doi.org/10.1017/S0308210517000312},
}

@incollection {Ball,
    AUTHOR = {Ball, John M.},
     TITLE = {Some open problems in elasticity},
 BOOKTITLE = {Geometry, mechanics, and dynamics},
     PAGES = {3--59},
 PUBLISHER = {Springer, New York},
      YEAR = {2002},
      ISBN = {0-387-95518-6},
   MRCLASS = {74-02 (49J40 74B20 74G65)},
  MRNUMBER = {1919825},
MRREVIEWER = {Georg\ K.\ Dolzmann},
       DOI = {10.1007/0-387-21791-6\_1},
       URL = {https://doi.org/10.1007/0-387-21791-6_1},
}

@article {DeL,
    AUTHOR = {De Lellis, Camillo},
     TITLE = {Some fine properties of currents and applications to
              distributional {J}acobians},
   JOURNAL = {Proc. Roy. Soc. Edinburgh Sect. A},
  FJOURNAL = {Proceedings of the Royal Society of Edinburgh. Section A.
              Mathematics},
    VOLUME = {132},
      YEAR = {2002},
    NUMBER = {4},
     PAGES = {815--842},
      ISSN = {0308-2105,1473-7124},
   MRCLASS = {49Q20 (49J10)},
  MRNUMBER = {1926918},
MRREVIEWER = {Ilaria\ Fragal\`a},
       DOI = {10.1017/S030821050000189X},
       URL = {https://doi.org/10.1017/S030821050000189X},
}

@book {dR,
    AUTHOR = {de Rham, Georges},
     TITLE = {Vari\'et\'es diff\'erentiables. {F}ormes, courants, formes
              harmoniques},
    SERIES = {Publications de l'Institut Math\'ematique de l'Universit\'e{}
              de Nancago [Publications of the Mathematical Institute of the
              University of Nancago]},
    VOLUME = {III},
      NOTE = {Actualit\'es Scientifiques et Industrielles, No. 1222.
              [Current Scientific and Industrial Topics]},
 PUBLISHER = {Hermann \& Cie, Paris},
      YEAR = {1955},
     PAGES = {vii+196},
   MRCLASS = {53.0X},
  MRNUMBER = {68889},
MRREVIEWER = {P.\ A.\ Smith},
}

@misc{dePauw,
      title={Quantified compactness in Lipschitz-free spaces of $[-1,1]^n$}, 
      author={Thierry De Pauw},
      year={2025},
      eprint={2504.19100},
      archivePrefix={arXiv},
      primaryClass={math.FA},
      url={https://arxiv.org/abs/2504.19100}, 
}

@misc{MM,
      title={A simple proof of the $1$-dimensional flat chain conjecture}, 
      author={Andrea Marchese and Andrea Merlo},
      year={2024},
      eprint={2411.15019},
      archivePrefix={arXiv},
      primaryClass={math.AP},
      url={https://arxiv.org/abs/2411.15019}, 
}

@article {Muller,
    AUTHOR = {M\"uller, Stefan},
     TITLE = {Higher integrability of determinants and weak convergence in
              {$L^1$}},
   JOURNAL = {J. Reine Angew. Math.},
  FJOURNAL = {Journal f\"ur die Reine und Angewandte Mathematik. [Crelle's
              Journal]},
    VOLUME = {412},
      YEAR = {1990},
     PAGES = {20--34},
      ISSN = {0075-4102,1435-5345},
   MRCLASS = {49J45 (58E35 73C50 73V25)},
  MRNUMBER = {1078998},
MRREVIEWER = {Ricardo\ S.\ Kubrusly},
       DOI = {10.1515/crll.1990.412.20},
       URL = {https://doi.org/10.1515/crll.1990.412.20},
}

@misc{BCTVW,
      title={Structure of Metric $1$-currents: approximation by normal currents and representation results}, 
      author={David Bate and Emanuele Caputo and Jakub Takáč and Phoebe Valentine and Pietro Wald},
      year={2025},
      eprint={2508.08017},
      archivePrefix={arXiv},
      primaryClass={math.MG},
      url={https://arxiv.org/abs/2508.08017}, 
}

@misc{ArBou,
      title={Structural properties of one-dimensional metric currents: SBV-representations, connectedness and the flat chain conjecture}, 
      author={Adolfo Arroyo-Rabasa and Guy Bouchitté},
      year={2025},
      eprint={2508.08212},
      archivePrefix={arXiv},
      primaryClass={math.AP},
      url={https://arxiv.org/abs/2508.08212}, 
}

@misc{MarcheseSurvey,
      title={A PDE perspective on the flat chain conjecture}, 
      author={Andrea Marchese},
      year={2025},
      eprint={2511.06822},
      archivePrefix={arXiv},
      primaryClass={math.AP},
      url={https://arxiv.org/abs/2511.06822}, 
}

@article {DeMaMa,
    AUTHOR = {De Masi, Luigi and Marchese, Andrea},
     TITLE = {A refined {L}usin type theorem for gradients},
   JOURNAL = {J. Funct. Anal.},
  FJOURNAL = {Journal of Functional Analysis},
    VOLUME = {289},
      YEAR = {2025},
    NUMBER = {11},
     PAGES = {Paper No. 111152, 17},
      ISSN = {0022-1236,1096-0783},
   MRCLASS = {49Q15 (49Q20)},
  MRNUMBER = {4940029},
       DOI = {10.1016/j.jfa.2025.111152},
       URL = {https://doi.org/10.1016/j.jfa.2025.111152},
}
\end{document}